\documentclass[a4paper]{amsart} 
\usepackage{graphicx,epsfig,amscd,graphics,xypic}
\usepackage{amssymb}
\usepackage{amscd}
\usepackage{times}
\usepackage{enumitem}
\usepackage{mathrsfs}
\usepackage{psfrag}
\usepackage[active]{srcltx}
\usepackage{color}
\usepackage{wrapfig}
\usepackage[applemac]{inputenc}

\usepackage{lscape}

\usepackage{float}

\usepackage{pstricks,pstricks-add,pst-math,pst-xkey}
\usepackage{subfigure}

\usepackage{hyperref}

\usepackage{graphicx,epsfig,amscd,graphics,xypic}

\usepackage[dvips,all]{xy}

\usepackage{graphics}

\usepackage[margin=1in]{geometry}

\newtheorem{thm}{Theorem}[section]

\newtheorem*{thm*}{Theorem}   
\newtheorem*{lemma*}{Lemma}

\newtheorem{prop}[thm]{Proposition}

\newtheorem*{ex*}{Exemple \ref{exJ2} (continued)}

\newtheorem*{Ex:Ce/e3}{Example \ref{Ex:Ce/e3} (continued)}

\newcommand{\p}{{\mathbb P}}

\renewcommand{\ss}{\operatorname{ss}}

\newcommand{\rk}{\operatorname{rk}}

\newcommand{\map}{\dasharrow}

\def\S{Section~}

\newcommand{\Rad}{\operatorname{{\rm Rad}}}

\newcommand{\bk}{\bigskip}
\newcommand{\sk}{\smallskip}
\newcommand{\mk}{\medskip}

\def\og{\leavevmode\raise.3ex\hbox{$\scriptscriptstyle\langle\!\langle$~}}
\def\fg{\leavevmode\raise.3ex\hbox{~$\!\scriptscriptstyle\,\rangle\!\rangle$}}

\def\ro[#1]{{\textcolor{red}{#1}}}

\newcommand{\incl}[1][r]
  {\ar@<-0.2pc>@{^(-}[#1] \ar@<+0.2pc>@{-}[#1]}

\newcommand{\eq}[1][r]
   {\ar@<-3pt>@{-}[#1]
    \ar@<-1pt>@{}[#1]|<{}="gauche"
    \ar@<+0pt>@{}[#1]|-{}="milieu"
    \ar@<+1pt>@{}[#1]|>{}="droite"
    \ar@/^2pt/@{-}"gauche";"milieu"
    \ar@/_2pt/@{-}"milieu";"droite"}

\begin{document}

\title{
Quadro-quadric Cremona transformations \\ in low dimensions
via the $JC$-correspondence}  

\author{Luc PIRIO${}^\dagger$}
\thanks{${}^\dagger$ Partially  supported by the franco-italian research network GRIFGA and by the P.R.A. of the Universit\`a degli Studi di Catania.}
\author{Francesco RUSSO}

\address{Luc PIRIO, IRMAR, UMR 6625 du CNRS, Universit Rennes1, Campus de beaulieu, 
35000 Rennes, France}

\email{luc.pirio@univ-rennes1.fr} 
\address{Francesco RUSSO, Dipartimento di Matematica e Informatica, Universit\`a degli Studi di Catania,  Viale A. Doria 6, 95125  Catania, Italy}
\email{frusso@dmi.unict.it}

\maketitle

\begin{abstract} 
 We apply the  results of \cite{PR2} to study  
 quadro-quadric Cremona transformations in low-dimensional projective spaces.
 In particular  we describe new very simple families of such birational maps and obtain  complete and explicit classifications in dimension four and five. 
 \end{abstract}

\section*{Introduction}
The study of Cremona transformations is a quite venerable subject which  received  a lot of contributions classically (we refer to \cite{SBCDELSS} for a rich  overview of  the field  up to 1928, to the classical reference \cite{hudson}
and to \cite{SR} for an approach to the subject with classical methods but with a view to higher dimensional cases).  More recently there was a renewed interest in the subject, also  for the relations with complex dynamics which inspired the solution to the longstanding problem about the simplicity of the group of Cremona transformations in the plane, answered in the negative in \cite{Cantat}. The situation in the plane can be considered quite clear, also considering the classification of  finite subgroups of the Cremona group in the plane completed recently by Blanc (see \cite{Blanc} and the references therein). In higher dimension the structure  is more complicated
and also the generators of the Cremona groups are unknown in contrast to the case of the plane
where the classical Noether-Castelnuovo Theorem assures that the ordinary quadratic transformation and projective transformations generate the group. Thus since the very beginning
the study of Cremona transformations in higher dimensional projective spaces was developed only for particular classes (see for example \cite{EinShB,  pan, PRV}) or for specific degrees of the homogeneous polynomials defining the map. From this point of view the first cases of interest  are those defined by quadratic polynomials (see \cite{semple,brunoverra,PRV}) and thus the {\it simplest} examples of Cremona transformations, different from projective automorphisms, are those
whose inverse is also defined by quadratic polynomials, dubbed {\it quadro-quadric} Cremona transformations. In \cite{PRV} it is obtained the complete classification of quadratic Cremona transformations of $\p^3$ while \cite{semple} considers the case of $\p^4$, describing the general base locus scheme of these Cremona maps (see also the discussion in Section \ref{S:qqP4-Bruno-Verra}). 
The more recent preprint \cite{brunoverra} deals with the classification of general quadro-quadric Cremona transformations in 
$\p^4$ and $\p^5$ and provides some series of examples in arbitrary dimension.\sk 

A completely new approach to the subject of quadro-quadric Cremona transformations was began in \cite{PR1} and completed in \cite{PR2}, where general results and structure Theorems for these maps were presented. Indeed, 
in \cite{PR2}, we  proved that for every $n\geq 3$, there are equivalences between:
 \begin{itemize}
 \item irreducible $n$-dimensional non degenerate projective varieties $\boldsymbol{X}\subset \p^{2n+1}$ different from rational normal scrolls and 3-covered by rational cubic curves, up to projective equivalence;\sk 
\item $n$-dimensional complex {\bf{J}}ordan algebras of rank three, up to isomorphisms;\sk
 \item quadro-quadric {\bf{C}}remona transformations of $\p^{n-1}$, up to linear equivalence\footnote{Two Cremona maps $f_1,f_2$ of $\mathbb P^n$ are said to be {\it linearly equivalent} if $f_2=g \circ f_1\circ h$ for some linear automorphisms $g,h\in {\rm Aut}(\p^n)$.}.\sk 
 \end{itemize}

The  equivalence between these sets  has been named the  `{\it $XJC$-correspondence}' in \cite{PR2}\footnote{Note  that  Jordan algebras are considered  up to isomorphisms in the present paper whereas they were originally considered up to isotopies  in \cite{PR2}. If the later setting  is more natural from a categorical point of view (see \cite[Remark 4.2]{PR2}),  the $XJC$-correspondence still holds true at the set-theoretical level when considering rank 3 Jordan algebras up to isomorphisms. The reason behind this is 
the well-known fact that over the field of complex numbers,  two isotopic Jordan algebras are actually isomorphic  (cf. \cite[Problem 7.2.(6)]{McCrimmon-book}).}. In this text, we use mainly one part of this correspondence, that we call the $JC$-correspondence, that essentially asserts that for every $n\geq 3$: 
{\it any quadro-quadric Cremona transformation $f:\p^{n-1}\dashrightarrow \p^{n-1}$  is linearly equivalent to the projectivization of the adjoint  map $x\mapsto x^{\#}$ of  a $n$-dimensional rank 3 Jordan algebra, the isomorphism  class of which is uniquely determined by the linear equivalence class of $f$} (see Section \ref{Notation} for notation and definitions).

The $JC$-correspondence  does not only offer a conceptual interest but also provides   an effective   way to study concretely quadro-quadric Cremona transformations.
Indeed, the theory of Jordan algebras is now well developed and formalized so that one has at disposal several quite deep  results and powerful algebraic tools to study them conceptually and/or  effectively. 
\sk 

Armed by the powerful algebraic machinery of the theory of Jordan algebras, in this paper we
\begin{itemize}
\item  describe a general algebraic method to construct new quadro-quadric Cremona transformations starting from known ones (Section \ref{S:newQQCremona});\mk
\item use the above  method to construct, starting from the standard Cremona transformation of $\mathbb P^2$, a very simple countable family of quadro-quadric Cremona transformations of $\mathbb P^n$ for arbitrary $n\geq 2$ (Section \ref{S:verysimplefamily});\mk 
\item  use the previous construction  to produce  continuous families of quadro-quadric Cremona transformations starting from those associated to  simple rank three Jordan algebras (Section \ref{S:ContinuousFamily});\mk 
\item explain two distinct general constructions of  cubo-cubic Cremona transformations starting from a quadro-quadric one (Section \ref{S:From22to33});
\mk
\item give complete classifications of rank 3 Jordan algebras in dimension 3, 4 and 5 (in Sections \ref{S:JordanDim3}, \ref{S:JordanDim4} and \ref{S:JordanDim5}) and deduce from them the complete classifications, up to linear equivalence, of quadro-quadric Cremona transformations of $\mathbb P^n$ for $n=2,3$ and $4$ (see Table 2, Table 3 and Table 6 respectively); 
\mk
\item provide a detailed analysis of the quadro-quadric Cremona transformations of  $\mathbb P^4$ (in Section \ref{S:Bir22P4}).   For each  Cremona transformation of Table 6, we describe as geometrically and concisely as possible the associated base locus scheme and the homaloidal system of quadrics, determine its type and  its  multidegree (see Table 7, Table 8 and Table 9).   We also offer pictures  of these base locus schemes (see Figure 5, Figure 6);
\mk 
\item give (without proof) a complete list of  involutorial  normal forms for elements of quadro-quadric Cremona transformations of $\p^5$ (cf. Table 11). 
\mk
\end{itemize}
\smallskip

\hspace{-0.4cm}{\bf Acknowledgments.} The first author is very grateful to Professor H. Petersson  for having copied  and sent  to him Wesseler's report \cite{wesseler}.

\section{Notation and definitions}${}^{}$
\label{Notation}

\subsubsection{General notation and definitions}
If $f=[f_0:\cdots:f_n]:\p^n\map\p^n$ is a    Cremona transformation, then 
$\mathcal B_f\subset\p^n$ will be the {\it  base locus scheme of $f$}, that is the scheme defined by the $n+1$ homogeneous forms  $f_0,\ldots,f_n$ that are assumed without  common factor. If all these are of  degree $d_1\geq 1$
and if the inverse $f^{-1}:\p^n\map\p^n$  of $f$ is defined by forms of degree $d_2$ without  common factorwe say that $f$ has {\it bidegree} $(d_1, d_2)$ and we put ${\rm bdeg}(f)=(d_1,d_2)$.   The set of Cremona transformations  of $\p^{n}$ of bidegree $(d_1,d_2)$ will be indicated by ${\bf Bir}_{d_1,d_2}(\mathbb P^{n})$.

By definition, a {\it quadro-quadric Cremona transformation} is just a Cremona transformation of bidegree  $(2,2)$.

Two Cremona transformations  $f, \tilde f\in {\bf Bir}_{d_1,d_2}(\mathbb P^{n})$ are said to be {\it linearly equivalent} (or just {\it equivalent} for short) if there exist  projective transformations  $\ell_1,\ell_2: \mathbb P^{n}\rightarrow \mathbb P^{n}$ such that $ \tilde f= \ell_1 \circ f\circ  \ell_2$.

We define the {\it type of a Cremona transformation}  $f:\mathbb P^n\dashrightarrow \mathbb P^n$ as
the irreducible component of the Hilbert scheme of $\mathbb P^n$ to which  $\mathcal B_f$ belongs. We put
$
P_i(t)={t+i \choose i}
$
 so that $P_0=1$, $P_1(t)=t+1$ and $P_2(t)=\frac{(t+2)(t+1)}{2}$.
\mk

By definition, a {\it Jordan algebra} is a commutative complex algebra $J$  with a unity $e$ such that the  {\it Jordan identity}
$x^2(xy)=x(x^2y)$
holds for every $x,y\in J$ (see \cite{jacobson, McCrimmon-book}).  Here we shall also assume that $J$ is finite dimensional.
It is well known that a Jordan algebra is power-associative.
 By definition, the {\it rank} $\rk(J)$  of  $J$ 
is the complex dimension of the (associative) subalgebra  
$\langle x \rangle $ of $J$ spanned by the unity $e$ and by
 a general element $x\in  J$.

The simplest examples of Jordan algebras are those constructed from associative algebras. Let $A$ be a non-necessarily commutative associative algebra with a unity.  Denote by $A^+$ the vector space $A$  with the symmetrized product $a\cdot a'=\frac{1}{2}(aa'+a'a)$.  Then $A^+$ is a Jordan algebra. Note that   $A^+=A$ if  $A$ is commutative.

By $\mathcal J_{q,r}^m$  we will indicate a $m$-dimensional direct sum $\mathbb C\oplus W$ endowed with the Jordan product  $(\lambda,w)\cdot (\lambda',w')=(\lambda \lambda'-q(w,w'),\lambda w'+\lambda'w)$, where $q(\cdot,\cdot)$ is the polarization of a quadratic form $q$ on $W$,  of rank $r$. When we write $\mathcal J_{q}^m$ we will assume that $W=\mathbb C^{m-1}$ and that $q(x)=\sum_{i=1}^{m-1}x_i^2$ 
   in the standard system of coordinates $x=(x_1,\ldots,x_{m-1})$ on $\mathbb C^{m-1}$. Otherwise when the integer $r<m-1$ is specified we shall assume that $q(x)=\sum_{i=1}^{r}x_i^2$.  Except when $W=0$, a Jordan algebra $\mathcal J_{q,r}^m$ has rank 2. 
\smallskip

A Jordan algebra of rank 1 is isomorphic to $\mathbb C$ (with the standard multiplicative product). It is a classical and simple fact that any rank 2 Jordan algebra is isomorphic to an algebra $\mathcal J^m_{q,r}$ defined above.
In this paper, we will mainly consider Jordan algebras of rank 3.  
These are the simplest Jordan algebras which have not been yet classified in arbitrary dimension.
Due to the $JC$-correspondence  their classification is  equivalent to that of  quadro-quadric Cremona transformations in arbitrary dimensions, showing the complexity of the problem. 
\smallskip

Let $J$ be a rank 3 Jordan algebra. The general theory specializes in this case and ensures the existence of a linear form $T:J\rightarrow \mathbb C$ (the {\it generic trace}), of a quadratic form $S\in {\rm Sym}^2(J^*)$ and of a cubic form 
$N\in {\rm Sym}^3(J^*)$  (the {\it generic norm}) such that 
$x^3-T(x)x^2+S(x)x-N(x)e=0$
for every $x\in J$.  Moreover, 
$x$ is invertible in $ J$ if and only if $N(x)\neq 0$ and in this case  
$x^{-1}={N(x)}^{-1}{x^\#}$, where $x^\#$ stands for the {\it adjoint} of $x$ defined by 
$x^\#=x^2-T(x)x+S(x)e$. The adjoint  satisfies the identity 
$(x^\#\big)^\#=N(x)x.$

The algebra $M_n(\mathbb C)$ of $n\times n$ matrices with complex entries is associative hence  $M_n(\mathbb C)^+$ is  a Jordan algebra. 
  According to  Cayley-Hamilton Theorem, it is of rank $n$,  the generic trace of $M\in M_n(\mathbb C)^+$ is the usual one,  $N(M)=\det(M)$ and the adjoint is the matrix one, that is the transpose of the cofactor matrix of $M$.\smallskip 

 For $x=(\lambda,w)\in
\mathcal J^m_{q,r}$, one has   $x^2-T(x)x+N(x)e=0$ with 
 $T(x)=3\lambda$ and  $ N(x)=\lambda^2+q(w)$. Thus it can be verified that  $x^\#=(\lambda,-w)$ in this case.
\smallskip 
 
More generally let $J$ be a power-associative algebra with $r=\rk(J)\geq 2$. Then defining analogously the adjoint of an element $x\in J$, the identity $(x^\#\big)^\#=N(x)^{r-2}x$  holds so that   the projectivization of the adjoint $[\#]:\mathbb P(J)\map\mathbb P(J)$ 
is a birational involution  of bidegree $(r-1,r-1)$ on $\mathbb P(J)$.

 The inverse map
$x\mapsto x^{-1}={N(x)}^{-1}{x^{\#}}$  on $J$   naturally induces a birational involution 
$\widetilde \jmath:\p(J\times\mathbb C)\map
\p(J\times \mathbb C)$ of bidegree $(r,r)$, defined by $\widetilde \jmath([x,r])=[rx^{\#},N(x)]$.
Such  maps were classically investigated by  N. Spampinato and C. Carbonaro Marletta,
see \cite{Spampinato,Carbonaro1,Carbonaro2}, producing examples of interesting Cremona involutions
in higher dimensional projective spaces. It is easy to see that letting $\widetilde J=J\times\mathbb C$, then for $(x,r)\in\widetilde J$
one has $(x,r)^{\#}=(rx^{\#},N(x))$ so that the map $\widetilde \jmath$  is the adjoint map
of the algebra $\widetilde J$. A Cremona transformation 
 of bidegree  $(r-1,r-1)$ 
will be called {\it of Spampinato type} if it is linearly equivalent to the adjoint of a 
direct product  $  J\times\mathbb C$ where $J$ is a power-associative algebra of rank $r\geq 3$. \sk

We will denote by   $J,J', \mathcal J, ...$ complex Jordan algebras of finite dimension and 
$J\times J'$ will stand for  the  direct product of Jordan algebras. We shall write
$R$ for the radical ideal $\Rad(J)$ of  $J$ when no confusion arises. Then $J_{\rm ss}$ will be the  {\it semi-simple part of  $J$}, isomorphic to $J/R$.

With $\mathcal B(J)\subset \mathbb P(J)$ we shall indicate the  base locus scheme of the birational involution $[\#]$ associated to $J$ and 
$\mathcal I_{\mathcal B(J)}$   will be the  graded ideal  generated by the homogenous polynomials  defining  the adjoint of $J$.

\smallskip


\subsubsection{Some non-reduced punctual schemes in projective spaces}
For $n,k$ such that  $0\leq k <n$ and for any nonnegative $t$, one defines  a {\it $t$-multiple  $\mathbb P^k$ in a $\mathbb P^n$} is a subscheme of $\mathbb P^n$ with homogeneous ideal equal to  the $t$-th power of the  ideal of a linear subspace of dimension $k$ in $\mathbb P^n$. \footnote{This terminology is not standard. For instance, it does not correspond exactly to  the one used in \cite{EisenbudHarris}.}

We shall also provide some pictures of the base locus schemes of quadro-quadric Cremona maps in dimension $2,3$ and $4$.  A lot  of these base loci  contain a non-reduced punctual scheme 
as an irreducible component of particular type, which we now describe and picture in this subsection.\medskip 

\paragraph{$-$ {\it Scheme $\tau$.}}
If $\tau_0={\rm Proj}\big( {\mathbb C[a,b]}/{(b,a^2)}\big) $ is the length 2  affine scheme ${\rm Spec}\big( {\mathbb C[\epsilon]}/{(\epsilon ^2)}\big)$ naturally embedded in $\mathbb C\subset \mathbb P^1$, 
we will designate by  $\tau$ any  scheme in a projective space $\mathbb P^n$ obtained as the image $\iota_*(\tau_0)$ for a linear embedding $\iota:\mathbb P^1\hookrightarrow \mathbb P^n$.  More intuitively, 
 $\tau$ is a punctual scheme in a projective space formed by a point $p$ plus another point  infinitely near to it. Since geometrically $\tau$ is nothing but a tangent  direction at one point, this scheme will be pictured as follows in the sequel: 
  \begin{figure}[!h]
\centering
\psfrag{tau}[][][1]{$  $}
 \includegraphics[width=1.5cm,height=1cm]{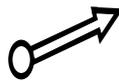}
 \vspace{-0.35cm} 
\caption{Pictural representation of $\tau$}
\end{figure}

  \paragraph{$-$ {\it Scheme $\xi$.}} 
Let $ \xi_0\simeq  
{\rm Spec}( {\mathbb C[a,b]}/{(a^2,ab,b^2)})$ be a double point in $\mathbb P^2$. We will designate by $\xi$ 
any  scheme in a projective space $\mathbb P^n$ obtained as the image $\iota_*(\xi_0)$ for a linear embedding $\iota:\mathbb P^2\hookrightarrow \mathbb P^n$.  Intuitively, 
 $\tau$ is a punctual scheme in a projective space formed by a point $p$ plus two distinct points  infinitely near of $p$. Geometrically $\tau$ is nothing but two tangent  directions at a given point of a projective space, it is reasonable to picture  this scheme as follows in the sequel: 
  \begin{figure}[H]
\centering
\psfrag{tau}[][][1]{$  $}
 \includegraphics[width=2cm,height=1.5cm]{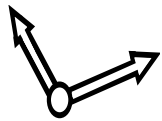}
 \vspace{-0.35cm} 
\caption{Pictural representation of  $\xi$}
\end{figure}

The reader has to be aware that this representation is a bit misleading: indeed, if $p$ stands for the support of  $\xi$,  there is no distinguished tangent directions at $p$ associated to $\xi$. What is intrinsically attached to $\xi$ is the pencil of tangent directions spanned by the two represented in Figure 2. The projective tangent space of $\xi$ at $p$ coincides with the  2-plane $\langle \xi \rangle$  spanned by $\xi$.  
\medskip

  \paragraph{$-$ {\it Scheme $\eta$.}} 
Let $\eta_0={\rm Proj}\big( {\mathbb C[a,b]}/{(b,a^3)}\big) $ be the length 3  affine scheme ${\rm Spec}\big( {\mathbb C[\epsilon]}/{(\epsilon ^3)}\big)$ naturally embedded in $\mathbb C\subset \mathbb P^1$. 
We will designate by  $\eta$ any  scheme in a projective space $\mathbb P^n$ obtained as the image $\nu_*(\eta_0)$ for a quadratic Veronese embedding $\nu:\mathbb P^1\hookrightarrow \mathbb P^n$.  More intuitively, 
 $\eta$ is a punctual scheme in a projective space formed by a point $p=\eta_{\rm red}$ lying on a conic $C\subset \mathbb P^n$, plus another point  $p'$ infinitely near to $p$ on $C$, plus a third point $p''$ still on $C$, infinitely near to $p'$. In geometrical terms  $\eta$ is nothing but an osculating flag of the second order\footnote{The {\it osculating flag of order $k$} of a (reduced) curve $C\subset \mathbb P^n$ at a smooth point $c$ is the flag $\{c\}=T^{(0)}_c C\subset 
 T^{(1)}_c C\subset  \cdots \subset T^{(k)}_c C$ where for any non-negative integer  $l$, $T^{(l)}_c C\subset \mathbb P^n$  stands for the osculating space of order $l$ to $C$ at $c$.}  to a smooth conic thus    this scheme will be pictured as follows in the sequel: 
  \begin{figure}[H]
\centering
\psfrag{tau}[][][1]{$  $}
 \includegraphics[width=2cm,height=2cm]{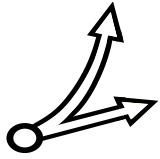}
 \vspace{-0.35cm} 
\caption{Pictural representation of  $\eta$}
\end{figure}
By definition, the {\it tangent line to $\eta$}, noted by $T_p \eta$,  is the line determined by $p$ and $p'$, which is included 
in the  2-plane $\langle \eta \rangle\subset \mathbb P^n$  spanned by $\eta$, which is also the projective tangent space to $\eta$ at $p$.  

Let $H$ be a hypersurface of  $ \mathbb P^n$. We will say that $H$ {\it osculates} (or {\it is osculating}) along $\eta$ if $\eta\subset H$ as subschemes of $\mathbb P^n$. More geometrically, this means that if $C'\subset \mathbb P^n$ is any curve passing through $p$ with second osculating flag at this point  equal to the one associated to $\eta$ then the intersection multiplicity between  $C' $ and $H$ at $p$ is at least 3. 
\medskip

 \paragraph{$-$ {\it Scheme $\chi$.}}  Let's define 
$\chi={\rm Proj}\big(  {\mathbb C[x,y,z,t]}/{(x^2, xy,y^2 , z^2 , xz, 2yz - xt )}   \big)\subset \mathbb P^3$. 
This scheme can be characterized from an algebraic and geometric point of view as follows:  according to \cite[p. 445]{NotariSpreafico}, up to projective equivalence, there exist only two  punctual degree four schemes in $\p^3$ defined by  six quadratic equations. The projective tangent tangent space at the point  of one of them is the whole $\mathbb P^3$ while $\chi$ is the other one. By definition, the {\it tangent plane to $\chi$} is the projective tangent space to $\chi$ at the supporting point $p=\chi_{\rm red}$. It is a 2-plane denoted by $T_p\chi$ that is strictly contained in the linear span $\langle \chi\rangle =\p^3$ of $\chi$.   

More intuitively and geometrically, $\chi$ can be obtained as follows:  let $\ell$ be a line  in $\p^3$ that is transverse to a scheme $\eta$ described just above (this means that $\ell$ passes through $p=\eta_{\rm red}$ and is distinct from $T_p \eta$).  For every $x\in  \ell\setminus \{  p \}$, let $\chi_x$ be  the union of $\eta$ with $x$.  Then cleary, $\chi$ is the flat limit of the schemes $\chi_x$ when $x$ tends to $p$. 

This geometrical description of $\chi$ explains that this subschme will be pictured as follows in the sequel: 

  \begin{figure}[H]
\centering
\psfrag{tau}[][][1]{$  $}
 \includegraphics[width=2cm,height=2cm]{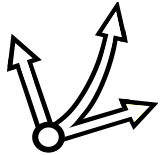}
 \vspace{-0.35cm} 
\caption{Pictural representation of  $\chi$}
\end{figure}

The reader has to be aware that this picture is a bit misleading: the presence of two rectilinear arrows in it does not mean that there are two distinguished tangent directions at $p$ associated to $\chi$. 
What is intrinsically attached to $\chi$ is the 2-plane spanned by the represented two tangent directions, that is nothing but the  tangent plane to $\chi$.
\medskip 

Let $H\subset \p^3$ be a surface.  The schematic condition  that $H$ contains a scheme $\chi$ translates geometrically as the fact that $H$ is tangent to a 2-plane $\pi$ at a fixed point $p\in \pi$ and osculates at order two to a smooth conic $C$ tangent to $\pi$ at this point.



  \section{\bf Constructions and new families of quadro-quadric Cremona transformations}
We now use some the theory of  non-associative algebras to construct algebraically  or/and  geometrically new families of involutorial Cremona transformations in arbitrary dimensions.

 \subsection{Construction of  new quadro-quadric Cremona transformations from known ones}
 \label{S:newQQCremona}

As far as we know, very little  is known about explicit examples of quadro-quadric Cremona transformations of $\mathbb P^{m}$, at least for large enough  $m$.  In fact, there is almost no description in the  literature of such birational maps in higher  dimension.  A notable exception is the  series of quadratic elementary Cremona transformations (and their possible degenerations), that have been considered in arbitrary dimension by classical authors being obtained easily from an irreducible quadric hypersurface in $\p^{m+1}$ by two different projections from  smooth points.  Recently,    Bruno and Verra have described  ({\it cf.} Proposition 6.2 of \cite{brunoverra}) two new families of Cremona transformations of bidegree $(2,2)$, one in each  dimension $m\geq 4$, by specifying the associated base locus schemes that can be :   \smallskip  
\begin{itemize}
\item[(A)] the union $\Pi_2\cup \Pi_3\cup \ell \subset \p^m$ where  $\Pi_i$ is a projective subspace of codimension $i$ for $i=2,3$ and where $\ell$ is a line such that $\Pi_3$ and $\ell$ are disjoint and each intersects $\Pi_2$ in dimension $m-4$ and 0 respectively;\smallskip 
\item[(B)] the schematic union of a double linear subspace $\mathcal P$ of dimension ${m-3}$ in a 
 hyperplane $H$ with a smooth conic $C$  that is tangent to  $H$ at the point $C\cap \mathcal P_{\! \rm red}$. \medskip 
\end{itemize}

By the $JC$-Correspondence recalled above and proved in \cite{PR2}, we know that,   that modulo composition by  linear automorphisms,  quadro-quadric Cremona transformations are nothing but adjoints of  rank 3 Jordan algebras. Here, we use this to describe a general construction of new Cremona transformations of bidegree $(2,2)$ starting from known ones.  This seems to be unknown despite its simplicity (in particular, see  Section \ref{S:verysimplefamily} below). \medskip

Let  $J$ be a Jordan algebra and $M_1,M_2$ be two Jordan $J$-(bi)modules (see \cite{jacobson}): for $i=1,2$, $M_i$ is a (non-unital)  Jordan algebra and there is a bilinear product $J\times M_i\rightarrow M_i: (x,m_i)\mapsto x \cdot m_i$  
such that $J_i=J\oplus M_i$ with the product defined explicitly by 
$$
(x,m_i)\cdot (x',m_i')=(xx',xm_i'+x'm_i+m_im_i')
$$
is a Jordan algebra with unity $(e,0)$,  where $e$ stands for the unity of $J$. 

Then  defines the {\it `gluing of $M_1$ and $M_2$ along $J$} as  the space  
$$J(M_1,M_2)=J\oplus M_1\oplus M_2$$ 
endowed with 
the product $\bullet$  explicitly defined by 
$$
(x,m_1,m_2)\bullet (x',m_1',m_2')=\big(xx',x\cdot m_1'+x'\cdot m_1+m_1m_1',x\cdot m_2'+x'\cdot m_2+m_2m_2'\big).$$
It can be verified that this product verifies the Jordan identity and makes of  
$J(M_1,M_2)$ a Jordan algebra.  Its unity  is $(e,0,0)$ and it  contains in a natural way $J,J_1$ and $J_2$ as Jordan subalgebras. \sk 

Let us remark that if $M_3$ is a Jordan $J$-module satisfying the same properties shared by  $M_1$ and $M_2$  and if $J_3$ stands for 
  $J\oplus M_3$ with the Jordan product induced by the $J$-module structure on $M_3$, then the three Jordan algebras $J_1(M_2,M_3)$, $J_2(M_1,M_3)$ and $J_3(M_1,M_2)$ identify naturally.  We will designate the corresponding algebra by $J(M_1,M_2,M_3)$.  Of course, this construction generalizes and can be iterated as many times as wanted. 
  \mk

 Now assume that $J_1$ and $J_2$ have the same rank $r$ 
and that there exists a norm (in the sense of \cite[II.\S 5]{bk}) $N\in {\rm Sym}^r(J^*)$  on $J$ such that $N(x)=N_1(x,m_1)=N_2(x,m_2)$ for every $(x,m_1,m_2)\in J\times M_1\times M_2$, where $N_i\in {\rm Sym}^r(J_i^*)$ 
is  the generic norm  of $J_i$ for $i=1,2$. 
For every $x\in J$ and $m_i\in M_i$ ($i=1,2$), let $m_i^{\# x}$ be the projection onto $M_i$ of $(x,m_i)^\#$. For instance, when $r=3$, if $T=dN_e\in J^*$ is the trace associated to $N$, then  for $i=1,2$,  one has:  
$$
m_i^{\# x}=2 x\cdot m_i+m_i^{2}-T(x)   m_i .
$$

 Then one   proves easily the following result:

\begin{prop}
\label{P:gluingJ}
Under the assumptions above, the  algebra $J(M_1,M_2)$ has rank $r$ and the adjoint is given by 
$$
(x,m_1,m_2)^\#=\big(x^\#, m_1^{\# x}, m_2^{\# \varphi(x)}\big). 
$$
\end{prop}

Let us now explain how this result can be used in practice to construct families of Jordan algebras. 
For any automorphism $\varphi:J\rightarrow J$ of Jordan algebras, 
 let $M_2^\varphi$ be the (non-unital) Jordan algebra $M_2$ endowed with the new bilinear map $ J\times M_2\rightarrow M_2\, ,  (x,m_2)\mapsto x \cdot _{\varphi}m_2= \varphi(x)\cdot m_2$.   This $\mathbb C$-bilinear map makes of $M_2^\varphi$ a Jordan $J$-module, as 
 one verifies without difficulty.  \sk 
 
Assume that  $J$ is semi-simple  and that  $M_1,\ldots,M_k$ are radicial Jordan modules (meaning that ${\rm Rad}(J_i)=M_i$ for every $i$ if $J_i$ stands for the Jordan algebra structure on $J\oplus M_i$ induced by the $J$-module structure of $M_i$).  If $\varphi_1,\ldots,\varphi_m$ are automorphism of $J$, it follows from the preceding considerations that 
$$
J(M_1^{\varphi_1}, \ldots,M_m^{\varphi_m})
$$
is a Jordan algebra, of the same rank that $J$.  Note that $
J(M_1^{\varphi_1}, \ldots,M_m^{\varphi_m})
$ 	and $
J(M_1, M_2^{\varphi_2\circ {\varphi_1}^{-1}}, \ldots,M_m^{\varphi_m\circ {\varphi_1}^{-1}}) $ are clearly isomorphic, hence one can restrict to the case when $\varphi_1={\rm Id}$ without any real loss of generality.\sk

Let us consider the preceding construction in the case when $r=3$,   the one we are more interested in. 
Proposition \ref{P:gluingJ} provides  a way of constructing  new rank 3 Jordan algebras starting from  finitely many rank 3 Jordan algebras 
$J_1,\ldots,J_m$  with isomorphic semi-simple parts, yielding  new quadro-quadric Cremona transformations arising from quadro-quadric Cremona transformations  with the same semi-simple part (see \cite{PR2} for definitions).  Let us  make explicit this construction in terms of Cremona transformations, which is very simple but quite surprisingly  unknown.
 \smallskip

For $i=1,\ldots,m$, let $F_i$ be a quadratic  affine lift to $V_i$ of   a quadro-quadric Cremona transformation 
$f_i:\mathbb P(V_i)\dashrightarrow \mathbb P(V_i)$   (with $\dim V_i\geq 3$).  Assume that the $F_i$'s have the same semi-simple part (see \cite[\S 5.1]{PR2}). In particular there exists  complex vector spaces $R_1,\ldots,R_m$ and $V$  such   $V_i=V\oplus R_i$  for every $i=1,\ldots,m$. Moreover there is a semi-simple quadro-quadric map $F_{\ss}\in {\rm Sym}^2(V^*)\otimes V$  (see \cite[Table 2]{PR2}) and bilinear applications $\mathcal F_i:V_i\times R_i\rightarrow R_i$  
such that for every $i=1,\ldots,m$, the map $F_i$ can be written  as
\begin{equation}
\label{E:Fi}
F_i\big({\boldsymbol x},{\boldsymbol y}_{i}\big)=\big(F_{\ss}({\boldsymbol x}),\mathcal F_i\big({\boldsymbol x},{\boldsymbol y}_{i}\big)\big)
\end{equation}
in some systems of linear coordinates ${\boldsymbol x}=(x_1,\ldots,x_n)$ on $V$  and ${\boldsymbol y}_{i}=(y_{i1},\ldots,y_{ir_i})$ on  $R_i$  (where $r_i=\dim R_i$ for every $i$)\footnote{To simplify the formula, we have abused the notation writing 
 $\mathcal F_i({\boldsymbol x},{\boldsymbol y}_{i})$ instead of $\mathcal F_i\big(({\boldsymbol x},{\boldsymbol y}_{i}), {\boldsymbol y}_{i}\big)$ in 
(\ref{E:Fi}).}.  To simplify, we assume that $F_{\ss}$ is involutorial  so that there exists a cubic form $N(x)$ such that $F_{\ss}(F_{\ss}(x))=N(x)x$ for every $x\in V$.  For $k=2,\ldots,m$, let $\varphi_k$ be a linear automorphism of $V$ 
such that $\varphi_k\circ F_{\ss}(x)=F_{\ss}\circ \varphi_k(x)$ and $N(\varphi_k({x}))=N(x)$ for every $x\in V$.  Then setting 
$R=R_1\oplus \cdots\oplus  R_m$, 
one defines the `{\it gluing of $f_1,\ldots,f_m$ along their semi-simple part by mean of $\varphi_2,\ldots,\varphi_m$}' as the rational map
$$f_1\boldsymbol{\cdot} f_2^{\varphi_2}\boldsymbol{\cdot} \cdots   \boldsymbol{\cdot} {f}_{m-1}^{\varphi_{m-1}}\boldsymbol{\cdot}  f_m^{\varphi_m} : \mathbb P(V\oplus R)\dashrightarrow \mathbb P(V\oplus R)$$
which is the projectivization of the affine quadratic map given  by the following formula  
$$
F_1\boldsymbol{\cdot} F_2^{\varphi_2}\boldsymbol{\cdot} \cdots   \boldsymbol{\cdot}F_m^{\varphi_m}
({\boldsymbol x},{\boldsymbol y})=\Big( 
F_{ss}({\boldsymbol x}), \mathcal F_1\big({\boldsymbol x},{\boldsymbol y}_{1}\big), \mathcal F_2\big(\varphi_{2}({\boldsymbol x}),{\boldsymbol y}_{2}\big),\ldots,
\mathcal F_m\big(\varphi_{m}({\boldsymbol x}),{\boldsymbol y}_{m}\big)
\Big)
$$
in the linear coordinates $({\boldsymbol x},{\boldsymbol y})=({\boldsymbol x},{\boldsymbol y}_{1},\ldots,{\boldsymbol y}_{m})$ on $V\oplus R$. Using  
Proposition \ref{P:gluingJ}, one can prove that  
$$
f_1\boldsymbol{\cdot} f_2^{\varphi_2}\boldsymbol{\cdot} \cdots   \boldsymbol{\cdot}f_m^{\varphi_m}
\in {\bf Bir}_{2,2}\big(\mathbb P(V\oplus R)\big). $$

   \subsection{Some new families of quadro-quadric Cremona transformations}
   We use now the method presented above to construct new explicit families of quadro-quadric Cremona transformations in arbitrary dimensions.
    
 \subsubsection{A family of very simple new elements of ${\bf Bir}_{2,2}(\mathbb P^{n})$ ($n\geq 2$)}
         \label{S:verysimplefamily}
Here we construct new quadro-quadric Cremona transformations that are particularly simple. 
For $i=1,2,3$, let $A_i$ be a complex vector space of finite dimension $\alpha_ i\geq 0$ and let $A$ be their direct sum: $A=A_1\oplus A_2\oplus A_3$.  One sets  $\alpha=\dim A=\alpha_1+\alpha_2+\alpha_3\geq 0$ and 
$\underline{\alpha}=(\alpha_1,\alpha_2,\alpha_3)$.  
Choosing some linear coordinates $ {\boldsymbol x}=(x_1,x_2,x_3)$ on $\mathbb C^3$ and 
${\boldsymbol a}_i=(a_{i1},\ldots,a_{i\alpha_i})$ on $A_i$ for $i=1,2,3$, one defines  a quadratic map  
$F^{\underline{\alpha}}$ on $  \mathbb C^{3}\oplus A=\mathbb C^{3+\alpha}$ by setting 
\begin{equation}
\label{E:Falpha}
F^{\underline{\alpha}}({\boldsymbol x},{\boldsymbol a}_1,{\boldsymbol a}_2,{\boldsymbol a}_3)=\Big(x_2x_3,x_1x_3,x_1x_2,x_1{\boldsymbol a}_1,x_2{\boldsymbol a}_2,x_3{\boldsymbol a}_3\Big)
\end{equation}
for every $({\boldsymbol x},{\boldsymbol a})\in \mathbb C^3\oplus A$, with of course  $x_i{\boldsymbol a}_i=(x_ia_{i1},\ldots,x_ia_{i\alpha_i})$ for $i=1,2,3$. 

One verifies immediatly that $F^{\underline{\alpha}}(F^{\underline{\alpha}}({\boldsymbol x},{\boldsymbol a}))=x_1x_2x_3({\boldsymbol x},{\boldsymbol a})$ for every $({\boldsymbol x},{\boldsymbol a})\in \mathbb C^3\oplus A$.  This implies that the projectivization of $F^{\underline{\alpha}}$ is an involutive quadro-quadric Cremona transformation:  for every $\underline{\alpha}\in \mathbb N^3$, one has 
$$
f^{\underline{\alpha}}=[F^{\underline{\alpha}}]\in  {\bf Bir}_{2,2}(\mathbb P^{2+\alpha}). 
$$

The simplest case  when ${\underline{\alpha}}=(0,0,0)$ is well known: 
$f^{(0,0,0)}$ is nothing but the standard Cremona involution of $\mathbb P^2$.  
The quadro-quadric maps $f^{(m-3,1,0)}$ for $m\geq 3$ are the ones considered  recently by Bruno and Verra and recalled in $({\rm A})$ at the beginning  of Section \ref{S:newQQCremona} above. 
Therefore although some particular cases have been considered recently, 
  it seems that the general quadratic involution $f^{\underline{\alpha}}$ considered above has been overlooked notwithstanding  its definition in coordinates \eqref{E:Falpha} is certainly one of the simplest that could be imagined.

    \subsubsection{A family of elements of ${\bf Bir}_{2,2}(\mathbb P^{2n})$ ($n\geq 1$) constructed from the algebra $\mathbb C[\varepsilon]/(\varepsilon^3)$}
  Let us  consider the  rank 3 associative algebra $\mathbb C[\varepsilon]/(\varepsilon^3)$ as a Jordan algebra, whose adjoint
  is easily seen to be $$(a,b,c)^\#=(a^2,-ab,b^2-ac)=:F(a,b,c)$$ in the system of linear coordinates associated to the basis $(1,\varepsilon, \varepsilon^2)$. 
  
 For  $n\geq 1$, let $  F_n: \mathbb C^{2n+1}\rightarrow \mathbb C^{2n+1}$   be the quadratic map defined by 
  $$
  F_n(a,b_1,c_1,\ldots,b_n,c_n)=(a^2, -ab_1,b_1^2-ac_1, \ldots, 
  -ab_n,b_n^2-ac_n)
  $$
   in some coordinates $(a,b_1,c_1,\ldots,b_n,c_n)$ on $\mathbb C^{2n+1}$.  The map $F_n$ is `involutorial' in the sense that
   $$F_n\big(F_n(a,b_1,\ldots,c_n)\big)
   =a^3 (a,b_1,\ldots,c_n    )$$
   for every $(a,b_1,\ldots,c_n    )\in \mathbb C^{2n+1}$. 
Then $f_n=[F_n]\in {\bf Bir}_{2,2}(\mathbb P^{2n})$ for every $n\geq 1$.

    \subsubsection{Construction of continuous families of quadro-quadric Cremona transformations}
    \label{S:ContinuousFamily}
All the families of quadro-quadric Cremona transformations presented in the two previous subsections are countable. This  a consequence of  the fact that the semi-simple parts of the afore-mentioned examples have {\it small} automorphism groups.

In fact, using the general construction of Section \ref{S:newQQCremona},  it is not difficult to construct  a continuous family
of Cremona transformations of bidegree $(2,2)$ by starting with a rank 3 Jordan algebra   $J$ 
whose semi-simple part $J_{\ss}$ is a rank 3 simple Jordan algebra.  Indeed, in this case, the automorphism group ${\rm Aut}(J_{\ss})$ of $J_{\ss}$  is a simple Lie group of positive dimension. Then  if $J$ is an extension of $J_{\ss}$ by a non-trivial  radical  $R$ of dimension $r>0$, we can construct the  family 
$$
\big\{   J_{\ss}(R, R^\varphi)  \, \big|  \, \varphi\in {\rm Aut}(J_{\ss}) \, \big\}.
$$
of rank 3 Jordan algebras of dimension $\dim (J_{\ss})+2r$. 
By considering the associated adjoint maps of these algebras,  we get  a family of 
quadro-quadric Cremona transformations of  $\mathbb P^{\dim (J_{\ss})+2r-1}$  
parametrized by the simple Lie group ${\rm Aut}(J_{\ss})$.

 \subsection{From quadro-quadric Cremona transformations to cubo-cubic ones}
\label{S:From22to33}
By definition, 
a {\it cubo-cubic Cremona transformation} is a birational map of a projective space of bidegree $(3,3)$. Such maps have been studied  in low dimension (see \cite{hudson,pan} for instance) but except in dimension 2 and 3 no general classification or structure results are known for them.
 We indicate below two distinct general constructions of cubo-cubic Cremona maps  starting from quadro-quadric ones. 
\mk

The first one is based on Spampinato's remark recalled in Section \ref{Notation}.

\begin{prop}\label{Spamp} Let $f: \mathbb P^n \dashrightarrow \mathbb P^n,\, x\mapsto [f_0(x):\cdots: f_n(x)]$ be a Cremona transformation of bidegree $(2,2)$.  If $N(x)=0$ is a cubic equation cutting out the  secant scheme of 
 the base locus scheme  of $f$ (see \cite{PR2}), then    
  \begin{align*}
 \mathbb P^{n+1}& \dashrightarrow \mathbb P^{n+1} \\
[x:r] & \longmapsto \big[ r f_0(x):\cdots: r f_n(x) : N(x) \big]
\end{align*}
is  a Cremona transformation of bidegree $(3,3)$ of Spampinato type. 
\end{prop}
\smallskip

Of course, via the previous construction one can produce  from a  given   $f\in {\bf Bir}_{2,2}(\p^n)$, a  Cremona transformation on $\mathbb P^{n+k}$ of bidegree $(2+k,2+k)$, for any $k\geq 1$.   
For example from this point of view,  the {\it standard  involution} of $\mathbb P^{n-1}$ that is the birational map
$$
[x_1:x_2:\ldots: x_n]\longmapsto \big[x_2x_3\ldots x_n : x_1x_3\ldots x_n:\ldots:x_1x_2\ldots x_{n-1}  \big]
$$
of bidegree $(n-1,n-1)$, which is the projectivization of the adjoint of the associative and commutative $n$-dimensional unital algebra $\mathbb C\times\mathbb C\times\dots\times\mathbb C$, is obtained from the {\it ordinary quadratic} transformation of $\p^2$
$$
[x_1:x_2: x_3]\longmapsto \big[x_2x_3 : x_1x_3:x_1x_2  \big]
$$
via Spampinato construction, that is via direct product.
\smallskip

The second construction of a cubo-cubic Cremona transformation from a given $f\in {\bf Bir}_{2,2}(\p^n)$ is also classical but less elementary and its generalization to higher degree is not clear at all. It is again based on an algebraic construction\footnote{  \label{FootNote}
Let $J$ be a rank 3 Jordan algebra with trace $T(x)$ and  cubic norm $N(x)$. By \cite[\S \!8.v]{allison}, the 
space  of {\it Zorn matrices with coefficients in $J$} defined by 
$$
Z_2(J)=\bigg\{  
\begin{pmatrix}
a & x \\ y & b
\end{pmatrix}  \Big\lvert \begin{tabular}{l}
$a,b\in \mathbb C$\\
$x,y\in J$
\end{tabular}
\bigg\}\, , 
$$ 
together with the  product $\bullet$ and the involution $M\mapsto \overline{M}$ given respectively by (where  $x\#y=(x+y)^\#-x^\#-y^\#$ for every $x,y\in J$) 
$$
\begin{pmatrix}
a & x \\ y & b
\end{pmatrix}\bullet  \begin{pmatrix}
a' & x' \\ y' & b'
\end{pmatrix}=
\begin{pmatrix}
aa'+T(x,y') & ax'+b'x +y\#y' \\ a'y+by' +x\#x'& bb'+T(x',y)
\end{pmatrix}\qquad \mbox{and}\qquad 
\overline{\begin{pmatrix}
a & x \\ y & b
\end{pmatrix}}=
\begin{pmatrix}
b & x \\ y & a
\end{pmatrix}
$$ 
is a {\it structurable algebra}, meaning that the triple product 
$[M,N,P]=(M\bullet \overline{N})\bullet P+(P\bullet \overline{N})\bullet M-(P\bullet \overline{M})\bullet N$ satisfies some particular algebraic  identities.  Moreover, the subspace $\{ M\in Z_2(J)\lvert \overline{M}=-M   \}$ of skew-elements is 1-dimensional and spanned by the diagonal matrix , noted by 
$\boldsymbol{\sigma}$, with scalar diagonal coefficients 1 and -1.   Then it can be proved (see formula (1.5) in \cite{allison2}) that the projectivization of the map 
$ M\mapsto \boldsymbol{\sigma} \bullet [M,\boldsymbol{\sigma}\bullet M, M]
$
is an involutorial  cubo-cubic Cremona transformation of $\mathbb PZ_2(J)$.}, but we shall describe it geometrically.  The $XJC$-Correspondence assures that  we can associated to $f=[F]\in {\bf Bir}_{2,2}(\p^{n-1})$ an irreducible  projective variety $X_f\subset\mathbb P^{2n+1}$, of dimension $n$, defined as the closure of the image of the affine parametrization 
 $\boldsymbol{x}\mapsto [1:\boldsymbol{x}:F(\boldsymbol{x}):N(\boldsymbol{x})]$,  notation as in Proposition \ref{Spamp}, and which is  3-covered by twisted cubic curves according to Proposition 3.3 of \cite{PR2}. 

Moreover, in  \cite[Corollary 5.4]{PR2}) we remarked that  $X_f$ is an OADP-variety: a general point $p\in \mathbb P^{2n+1}$ belongs to a unique secant line to $X_f$,  denoted by $\ell_p$.  For such a point $p$, there exists a unique unordered pair $\{  a_p,b_p \}$ of two distinct points of $X_f$ such that $\ell_p= \langle  a_p,b_p\rangle$. Thus  one can define  $ p'$ as the projective harmonic conjugate  of the triple $(a_p,b_p,p)$ on the projective line $\ell_p$: $ p'$ is the unique point on $\ell_p$ such that  ${\rm Cr}(a_p,b_p; p,p')=-1$\footnote{
  Here ${\rm Cr}(\cdot , \cdot ; \cdot ,\cdot)$ stands for the cross-ratio of two pairs of  points on the projective line $\ell_p$.}.  
  \begin{prop}
For any $f\in {\bf Bir}_{2,2}(\p^{n-1})$, $n\geq 3$, the map 
\begin{align*} 
\Phi_f: \mathbb P^{2n+1} & \dashrightarrow \mathbb P^{2n+1}\\
p & \longmapsto p'
\end{align*} defined  above 
is an involutorial  cubo-cubic Cremona transformation of $\mathbb P^{2n+1}$.
  \end{prop}
\begin{proof} From the fact that for every $q\in \ell_p\setminus\{a_p,b_p\}$ the line $\ell_p$ is the unique secant lines to $X_f$ passing through $q$ and from  well-known properties of the cross-ratio it follows that  $\Phi_f\circ \Phi_f={\rm Id}$ as a rational map so that $\Phi_f$ is a birational involution of $\mathbb P^{2n+1}$.  
  
  To prove that $\Phi_f$ is cubo-cubic, one can relate it with the cubic map considered in  Footnote  \ref{FootNote} by verifying that the arguments and formulae of \cite{KaYa}, that concern a priori only the semi-simple case, are in fact valid in full generality.  
  
  A more geometrical approach is the following one: 
  for  $p\in \mathbb P^{2n+1}$ general, the points $a_p$ and $b_p$  are general points of $X_f$. In particular, since $X_f$ is 3-covered by twisted cubics, there exist  twisted cubics curves included in $X_f$ and passing through $a_p$ and $b_p$. Let $C_p$ be such a curve. 
   Since $C_p$ is a twisted cubic, it is an OADP variety in its span $\langle C_p\rangle \simeq \mathbb P^3$and $\Phi_f$ induces  a Cremona involution $\Phi_{C_p}: \langle C_p\rangle \dashrightarrow \langle C_p\rangle $. Since $\Phi_{C_p}$ is an involution of bidegree $(3,3)$ in $\mathbb P^3$, we deduce that $\Phi_f$ is of the same type, concluding the proof.
 \end{proof}

Since $X_f\subset\p^{2n+1}$ is an OADP-variety, the projection $\pi_x$ from a general tangent space $T_xX$ induces a birational map
$\pi_x:X_f\map\p^n$. Thus, for $x_1, x_2\in X_f$ general points, the birational map  $\pi_{x_1}^{-1}\circ\pi_{x_2}:\p^n\map\p^n$ is easily seen to be a Cremona transformation  of {\it Spampintato type} linear equivalent to the map considered in Proposition \ref{Spamp}, see also \cite[Section 3.2.2]{PR2}.

\section{\bf An algebraic approach to  the   description of  ${\bf Bir}_{2,2}(\mathbb P^{n})$ for $n$ small}

Via the  $JC$-correspondence one can classify quadro-quadric Cremona transformation on $\mathbb P^{n-1}$ by using the classification  of $n$-dimensional rank 3 Jordan algebras.  After the pioneering work by Albert and  Jacobson (among many others) in the 50's, we have
 now  a wide range of particularly powerful  algebraic tools  to study Jordan algebras. In what follows, we will present shortly some  notions and results to be applied  to determine  rank 3 Jordan algebras in low dimension. In particular we shall obtain the complete classification of rank 3 Jordan algebras of dimension 5, from which we will deduce the complete list of involutorial normal forms for quadro-quadric Cremona transformations of   $\mathbb P^4$.

\subsection{Some classical  notions and tools in the theory of Jordan algebras}
\label{S:notions-tools-Jordan-algebras}
The following material is very classical and it is  presented in most of the standard references  on Jordan algebras (as \cite{jacobson,schafer,bk,McCrimmon-book} for instance).

\subsubsection{Nilpotent and nil algebras} Let $A$ be a complex algebra, only assumed to be commutative, but not necessarily neither Jordan nor with a unity.  If $A_1,A_2$ are two subsets of $A$, one sets $A_1A_2= A_2A_1=\{    a_1a_2 \mid  a_i\in A_i, \, i=1,2\, \}$.   Then one defines inductively $A^k$ for $k\geq 1$ by setting $A^1=A$ and $A^k=\cup_{  0< p<k} A^p A^{k-p}$ for $ k>1$.  The algebra  $A$ is  said to be {\it $k$-nilpotent} if $A^k=0$ but $A^{k-1}\neq 0$ and  it is said {\it nilpotent} if it is $k$-nilpotent for a certain $k\geq 1$.  By definition, $A$ is a {\it nil algebra} if it is a {\it $k$-nil algebra} for a certain positive integer $k$, i.e. if $a^k=0$ for every $a\in A$. 
\medskip

 Now let $R$ be the nontrivial radical of a Jordan algebra or rank $\rho$. Then it can be proved that $r^\rho=0$ for every $r\in R$ so that  $R$ is a $k$-nil algebra for a certain $k\in \{2,\ldots,\rho\}$.  In fact, much more is true  since Albert proved that $R$ is nilpotent and not only nil, a result not used in the sequel.

 \subsubsection{Peirce decomposition} 
 By definition, an {\it idempotent } of a Jordan algebra $J$ is a nonzero  element $u\in J$ verifying $u^2=u$.  
 For instance, the unity $e$ of $J$ is idempotent.  Let $u$ be a fixed idempotent.   One proves that the multiplication $L_u$ by $u$ satisfies the relation $L_u(L_u-I_d)(2L_u-I_d)=0$,
 yielding  the direct sum decomposition $J= J_0(u)\oplus  J_1(u)\oplus   J_{\frac{1}{2}}(u)$, where $ J_{\lambda}(u)= {\rm Ker}( L_u-\lambda\, I_d)=\{   v\in  J\,  \mid \,  u\cdot v=\lambda\, v\}$ for $\lambda\in \{  0,1, {1}/{2}  \}$. 

Two elements $u_1,u_2\in  J$ are {\it orthogonal} if $u_1\cdot u_2=0$. In this case, if both are idempotents, then their sum  $u_1+u_2$  is idempotent too. An idempotent $u$ is {\it irreducible} when it cannot be written $u=u_1+u_2$ where $u_1,u_2$ are two orthogonal  idempotents. Since $ J$ has finite dimension, one verifies easily that any idempotent $u$ 
admits a {\it irreducible decomposition by orthogonal idempotents}, that is can be written as $u=u_1+\cdots +u_m$ where $u_1,\ldots, u_m$ are pairwise orthogonal irreducible idempotents (then  $m$ is well-defined and  depends only on $u$). 
\smallskip 

Let $e=e_1+\cdots +e_m$ be such a decomposition for the unity $e$  of $ J$.
Then for $i,j=1,\ldots,m$ distinct, one sets 
$
J_{ii}= J_{1}(e_i)=\{   x\in  J\,  \mid \,  e_i\cdot x=x \}$  and $ J_{ij}=J_{\frac{1}{2}}(e_i)\cap  J_{\frac{1}{2}}(e_j)=\{   x\in  J\,  \mid \,  e_i\cdot x=e_j\cdot x=\frac{1}{2}x \}$.

\begin{prop}
\label{P:Peirce}
There is  a direct decomposition 
\begin{equation}
\label{E:PeirceDecomposition}
J=\bigoplus_{1\leq i\leq j\leq m}  J_{ij}.
\end{equation}

Moreover, for every distinct $i,j,k,\ell\in \{ 1,\ldots,m\}$, one has 
\begin{align}
\label{E:Jii}
( J_{ii})^2\subset  J_{ii}\, , && J_{ii}\cdot  J_{ij}\subset  J_{ij}\, ,  &&& ( J_{ij})^2\subset
 J_{ii}\oplus J_{jj}\, , 
 &&  J_{ij}\cdot  J_{jk}\subset  J_{ik}    
\end{align}
and 
\begin{align}
\label{E:Jij}
 J_{ii}\cdot  J_{jj}= J_{ij}\cdot  J_{kk}= J_{ij}\cdot  J_{k\ell}=0.
\end{align}

Finally,   setting $R={\rm {\rm Rad}}( J)$, one has 
\begin{equation}
\label{E:RJii}
R=\Big(   \bigoplus_{i=1}^m {\rm Rad}( J_{ii})\Big) \oplus \Big(   R \cap \big(\bigoplus_{ i< j } J_{ij}\big) \Big).
\end{equation}
\end{prop}

At least theoretically,  the material presented in this  subsection should be  sufficient to  classify  rank 3 Jordan algebras of `reasonable dimension'. More rigorously, the notions just introduced above  reduce the classification of Jordan algebras to some problems in linear algebra, which although  simple from a conceptual point of view, 
 can become immediately quite complicated  as soon as  the  dimension increases.
\smallskip

Let us mention here the reference \cite{wesseler} where the author presents the  classification of complex Jordan algebras of dimension less or equal to $ 6$.   The first named author has written a text \cite{PIRIO} based on \cite{wesseler} and providing complete proofs of the classification of rank 3 Jordan algebra of dimension less or equal to 6.
Since the report \cite{wesseler} is not well-known, has not been published (and it will not be published), we decide  to present some details in the next subsection on the classification of   Jordan algebras in dimension 4 and 5,  following  very closely    Wesseler's approach and arguments.

\subsubsection{Classification of nil Jordan algebras of low dimension}
\label{S:Classif3nil}
A Jordan algebra with radical of codimension 1 is isomorphic to the unitalization  of its radical. 
Then the classification of  rank 3 Jordan algebras asks in particular to know all the 
nil Jordan algebras of nilindex less or equal to 3.  
We recall below the classification of  these latter in low dimension (for further references on this, the reader can consult \cite{GerstenhaberMyung,ElguetaSuazo,MaVaNou}).

\begin{prop}
A nontrivial nilalgebra   of nilindex at most  3  and of dimension $n$ less or equal to  $ 4$ is isomorphic to an algebra with basis $(v_1,\ldots,v_n)$ 
in the following table: 
\begin{table}[H]
  \centering
  \hspace{-0.4cm}
  \begin{tabular}{|c|c|l|} 
   \hline
 {\bf Dimension}   &     {\bf Algebra}   & \; \; {\bf Non trivial products }    \\ \hline\hline
  2 &  $\mathcal R^2$ &   \;\,   $v_1^2=v_2$  \\ \hline \hline
 3 &    $\mathcal R^3_1 $  &  \;\,   $v_1^2=v_3$ \\ \hline 
  3 &    $\mathcal R^3_2 $  &     \;\,     $v_1^2=v_2^2=v_3$\\ \hline \hline
    4 &    $\mathcal R^4_1 $  &   \;\,       $v_1^2=v_3$,  $v_1v_2=v_4$\\ \hline
       4 &    $\mathcal R^4_2 $  &   \;\,       $v_1^2=v_2^2=v_3$,  $v_1v_2=v_4$\\ \hline
   4 &    $\mathcal R^4_3 $  &   \;\,       $v_1^2=v_4$ \\ \hline 
    4 &    $\mathcal R^4_4 $  &  \;\,        $v_1^2=v_2^2=v_4$\\ \hline
    4 &    $\mathcal R^4_5 $  &  \;\,        $v_1^2=v_2^2=v_3^2=v_4$\\ \hline
   \end{tabular} \vspace{0.25cm}
\label{Table:R}
\caption{Classification  of nilalgebras of nilindex $\leq 3$ in low dimensions.  }
\end{table}
\end{prop}

If $\dim(R)=2$, then the conclusion is trivial. Let us prove the above classification for $\dim R=4$, letting the remaining case to the reader. 
First of all we recall some useful results of  \cite{GerstenhaberMyung}: 
since $R$ is nilpotent and non trivial, one has $R^2\subsetneq  R$, implying $\dim R^2\in \{1,2,3    \}$ ($R^2=0$ must be excluded since 
$R$ is not trivial by hypothesis).  Assume that $\dim R^2=3$. Then there exists $r\in R$ such that $R=\mathbb C  r \oplus R^2$.  Then $R^2=\mathbb C r^2+r R^2+(R^2)^2\subset \mathbb C r^2+ R^3$, yielding 
$R^2=\mathbb C r^2+ R^3$.  This implies  $R={\rm Span}\{   r,r^2\} + R^3$. By repeating these arguments, one proves that 
$R={\rm Span}\{    r,\ldots,r^{k-1}\} + R^{k}$ for every $k\geq 1$.  Since $r^3=0$ and since  $R$ is nilpotent, we would deduce 
$R={\rm Span}\{    r,r^2\}$ and $\dim(R)=2$, contrary to our assumption.
Thus necessarily  $\dim R^2\in \{1,2   \}$.

Assume first  $\dim R^2=2$.  Let $r,r'\in R$ such that $r^2 $and  $rr'$ span $R^2$.  Then $(r')^2=\alpha \, r^2+\beta\,  rr'$ with $\alpha,\beta\in \mathbb C$. By replacing $r'$ by $r'-{\beta}r/2$, we can suppose $(r')^2=\alpha\, r^2$.  If  $\alpha\neq 0$, replacing $r$ by $r/{\sqrt{\alpha}} $, we can also suppose  $\alpha=1$.   Then one obtains two cases: the algebras $\mathcal R_1^4$ and $\mathcal R_2^5$ in the table above.

Assume now  $\dim R^2=1$. Let $r_1$ be such that $R^2=\mathbb C \, r_4$ where $r_4=(r_1)^2$  and choose $r_2 $ and $r_3$ such that $(r_1,\ldots,r_4)$ is a basis of $R$.  The product on $R$ is determined by 
the quadratic form $\varphi$ on $R'={\rm Span}\{ r_1,r_2,r_3 \}$  defined by the relation $rr'=\varphi(r,r') r_4$ for $r,r'\in R'$. Note that $\varphi$ is non-trivial since $\varphi(r_1)=1$.  Moreover, one verifies that isomorphic quadratic forms on $R'$ induce isomorphic Jordan algebras. Hence there are only three possibilities corresponding to the possible values 1, 2 or 3 for the rank of the quadratic form.  The corresponding algebras are denoted by $\mathcal R^4_3, \mathcal R^4_4$ and $\mathcal R^4_5$ in T{\small{ABLE}} 1 above.

\subsection{\bf Rank 3 Jordan algebras of dimension 3 and quadro-quadric Cremona transformations of $\mathbb P^2$}
\label{S:JordanDim3}
It is an easy exercise to determine all Jordan algebra of dimension 3 by using the
 $JC$-correspondence: a quadro-quadric Cremona transformation of the projective plane is given by its base locus scheme that is a non-degenerate  0-dimensional subscheme of length 3 in $\mathbb P^2$. It is immediate to see that there are exactly three such subschemes (up to isomorphisms) and that they belong to the same irreducible component of ${\rm Hilb}^3(\mathbb P^2)$. Consequently, up to isomorphisms,  there are three rank 3 Jordan algebras  of dimension 3.  
 
 The classification is collected in the following table: 

\begin{table}[H]
  \centering
  \hspace{-0.4cm}
  \begin{tabular}{|c|c|c|c|c|} 
   \hline
     {\bf Algebra}   & $\boldsymbol{\dim R}$  & {\bf Semi-simple part }& {\bf Adjoint} $\boldsymbol{(x,y,z)^\#}$ &{\bf Base locus}    \\ \hline
     \hline
$\mathbb C\times\mathbb C\times\mathbb C$
 &   $0$  &  $\mathbb C\times\mathbb C\times\mathbb C$
  & $(yz,xz,xy)$ & 
  \begin{tabular}{c}\vspace{-0.3cm}\\
  \includegraphics[width=6mm,height=6mm]{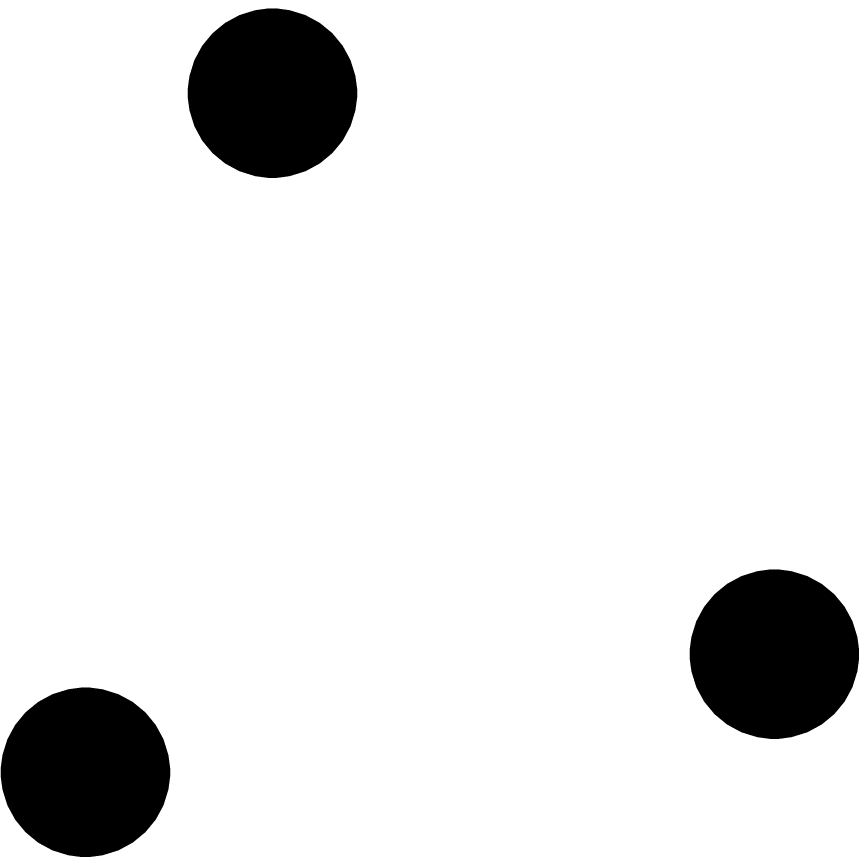} \\
  \end{tabular}
  \\
  \hline 
 $\mathbb C\times \frac{\mathbb C[\varepsilon]}{(\varepsilon^2)}$ 
& $1$  & $\mathbb C\times\mathbb C$
  & $(y^2,xy,-xz)$&
    \begin{tabular}{c}\vspace{-0.3cm}\\
  \includegraphics[width=6mm,height=6mm]{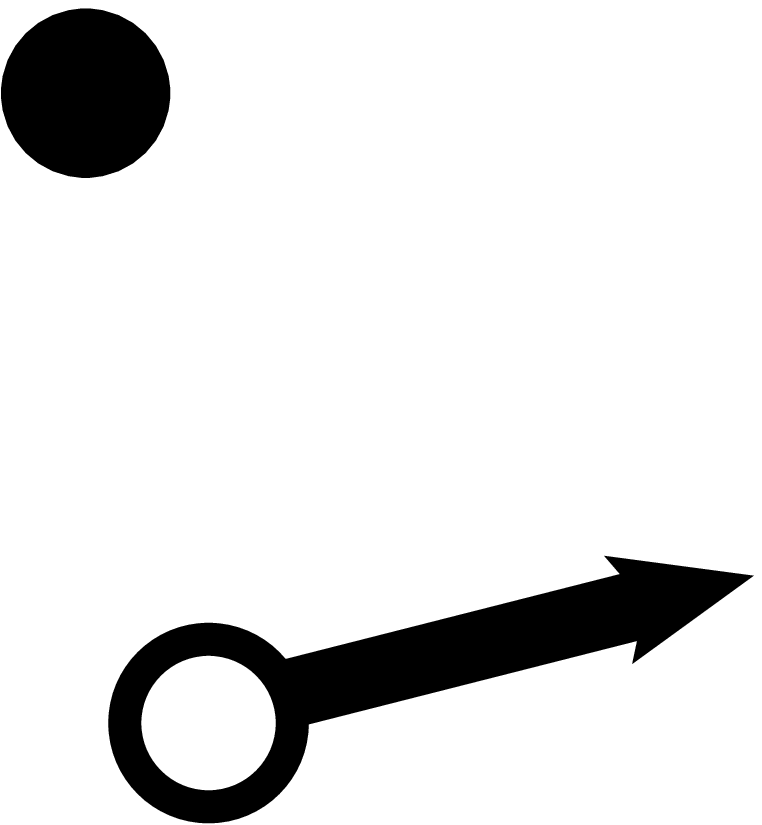} \\
  \end{tabular}
\\ \hline 
$\frac{\mathbb C[\varepsilon]}{(\varepsilon^3)}$   & $2$  & $\mathbb C$  &   $(x^2,-xy,y^2-xz)$  & 
  \begin{tabular}{c}\vspace{-0.3cm}\\
  \includegraphics[width=6mm,height=6mm]{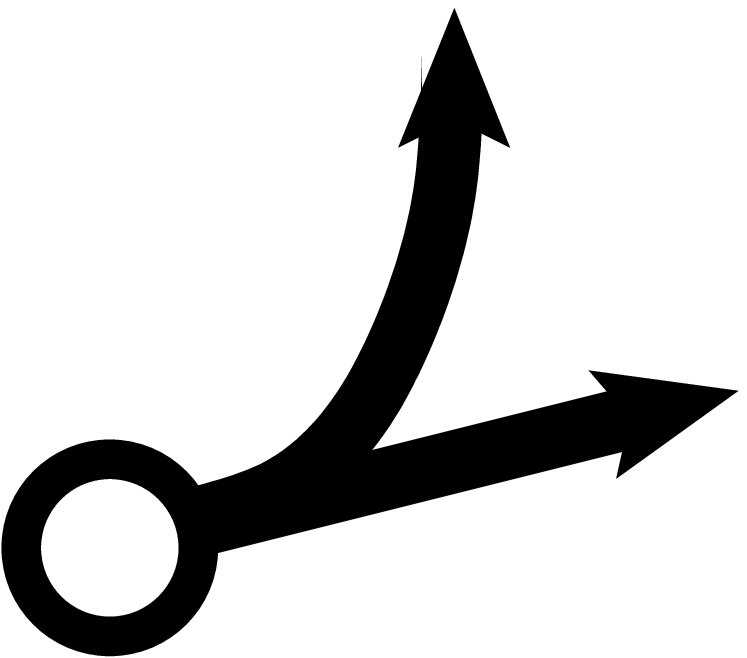} \\
  \end{tabular}
 \\
 \hline 
\end{tabular} \vspace{0.3cm}
\label{tablefp2}
\caption{Classification of rank 3 Jordan algebras of dimension three or equivalently,    of quadro-quadric Cremona transformations of $\mathbb P^2$.}
\end{table}

\subsection{\bf Rank 3 Jordan algebras of dimension 4 and quadro-quadric Cremona transformations of $\mathbb P^3$}
\label{S:JordanDim4}

The classification of $(2,2)$ Cremona transformations in $\p^3$ has been considered recently  in \cite{PRV} and classically by Enriques and Conforto (see all the references in \cite{PRV}).

In this section, we give the classification of rank 3 Jordan algebra in dimension 4.  This classification is also classical and for this reason some cases will be left to the reader. Recent references on the subject are  \cite{wesseler,Kashuba,PIRIO}.
Concerning the more particular case of  associative algebras, one can consult the classical papers \cite{study,scorzaALGEBRA4} or the more recent one 
 \cite{gabriel} (see also the references therein). 
\sk

\begin{table}[h]
  \centering
  \hspace{-0.4cm}
  \begin{tabular}{|c|c|c|c|c|} 
   \hline
   {\bf Algebra} $  {\boldsymbol J}$   &  $  {\boldsymbol{ \dim  R}}$ & $  {\boldsymbol J}_{\! \rm ss}$& {\bf Adjoint}  $\boldsymbol{(x,y,z,t)^\#}$ &{\bf Base locus}    \\ \hline
     \hline
     $\mathcal J^4_1=\mathbb C\times \mathcal J^3_{q,2}$  &  0 & $\mathbb C\times \mathcal J^3_{q,2}$ & $
\big(y^2+z^2+t^2\,,\, xy \,,\,  -xz \,,\,  -xt\big)$   &   
\begin{tabular}{c} \vspace{-0.2cm}\\
  \includegraphics[width=20mm,height=15mm]{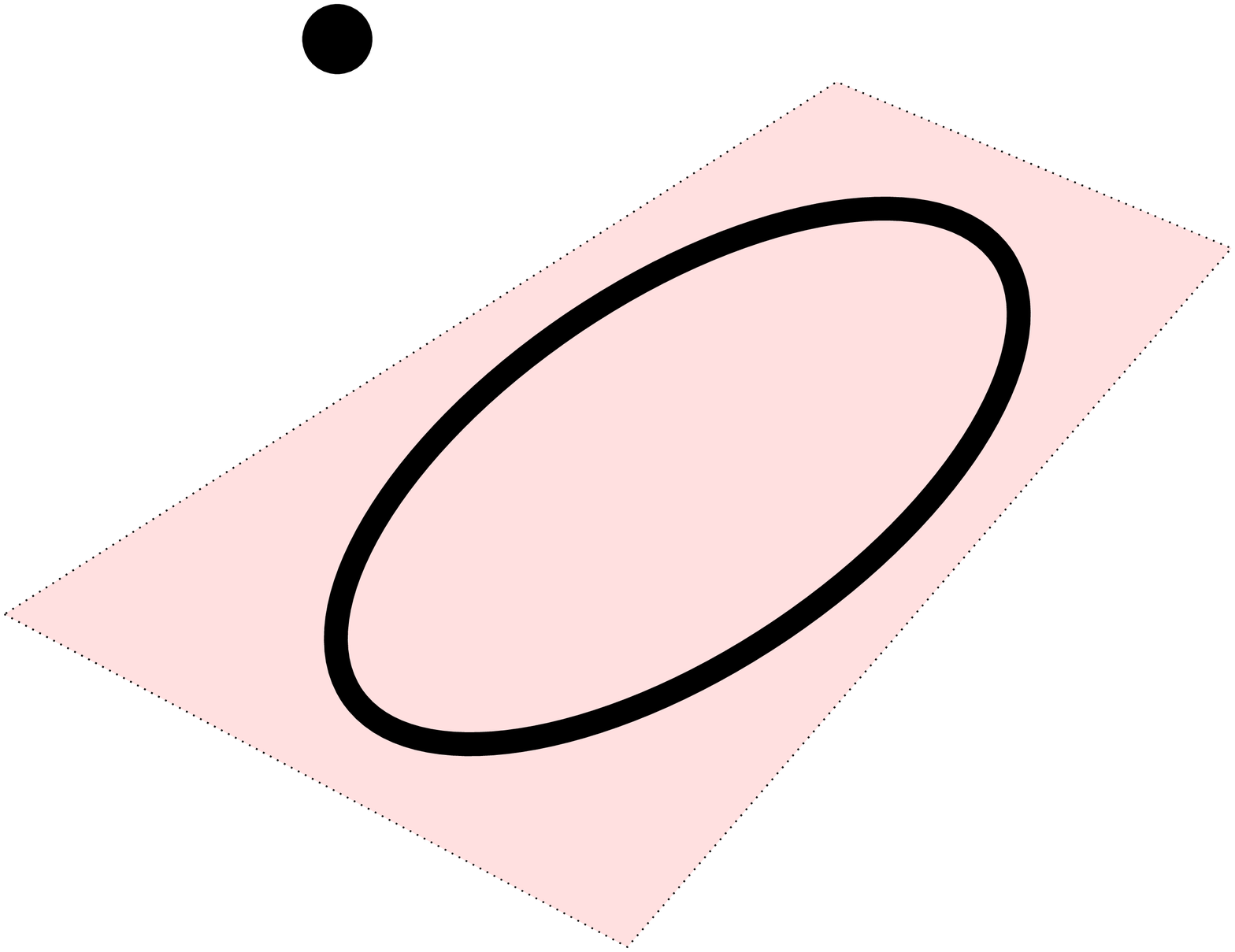}  
 \\
\end{tabular} 
   \\ 
\hline \hline
$\mathcal J^4_2= \mathbb C\times \mathcal J^3_{q,1}$  &  1 & $\mathbb C\times \mathbb C\times \mathbb C$&$\big(yz\,,\, xz\,, xy\,,\,xt\big)$   &  
\begin{tabular}{c} \vspace{-0.2cm}\\
  \includegraphics[width=20mm,height=15mm]{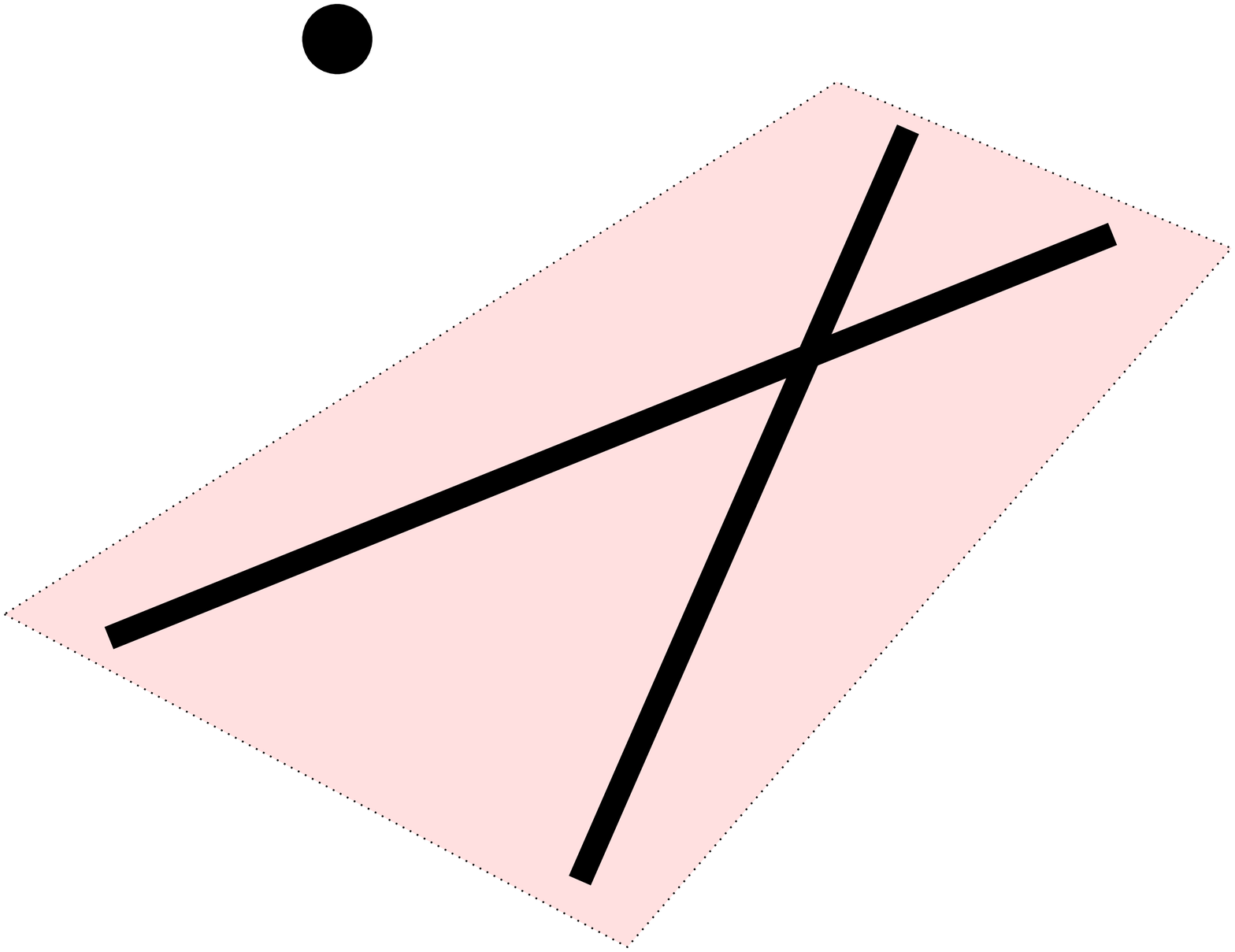}  
 \\
\end{tabular}
    \\ \hline \hline
    $\mathcal J^4_3=\mathbb C\times \mathcal J^3_{q,0}$   & 2 & $\mathbb C\times \mathbb C$ &  $\big({y}^2\,,\,xy\,, xz\,,zt\big)$   &  
\begin{tabular}{c} \vspace{-0.2cm}\\
  \includegraphics[width=20mm,height=15mm]{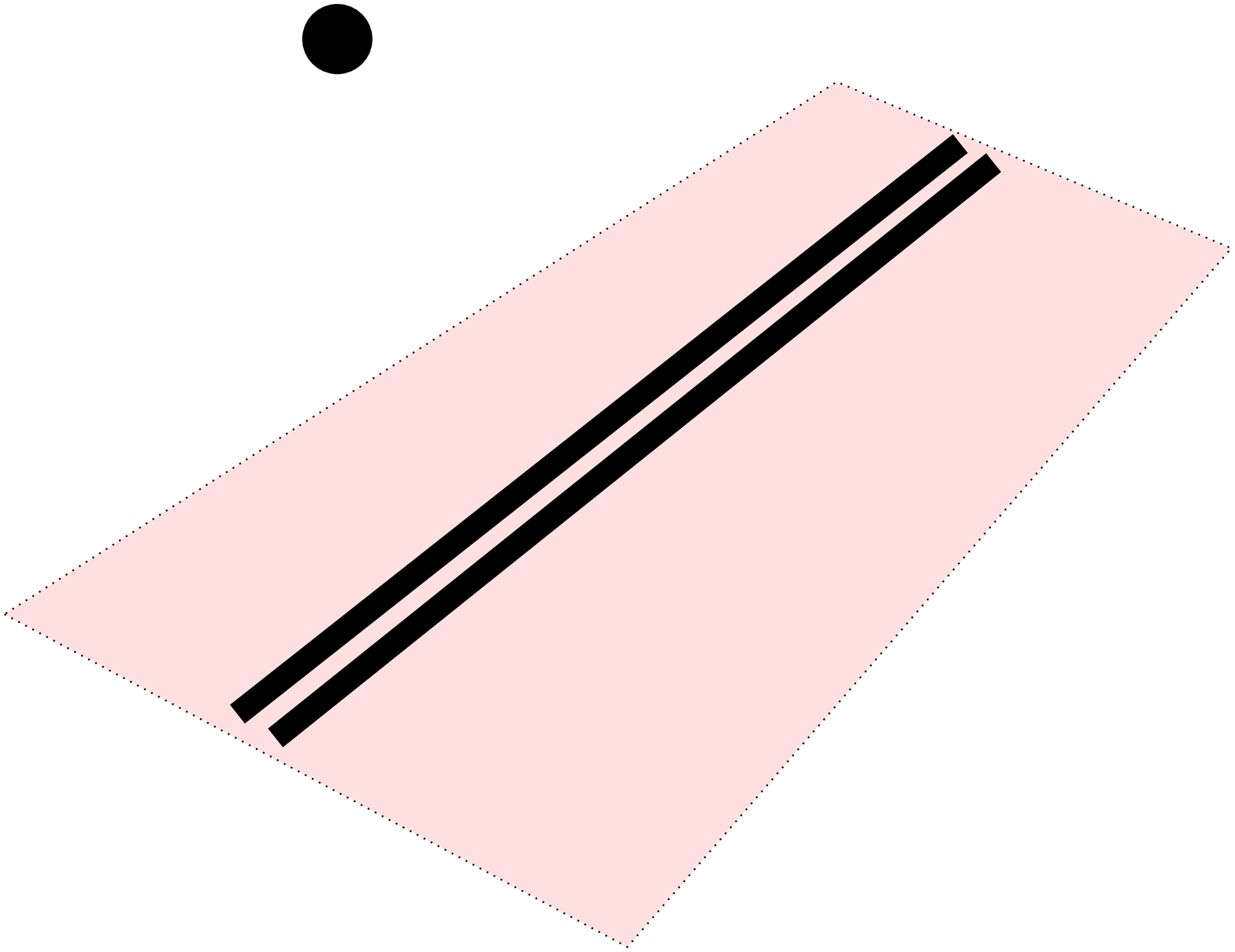}  
 \\
\end{tabular}
\\ 
\hline
   $\mathcal J^4_4$
   & 2 & $\mathbb C\times \mathbb C$& $\big(xy\, , \,x^2\, , \,   t^2-yz \, , \,  xt\big)$   &    
\begin{tabular}{c} \vspace{-0.2cm}\\
\psfrag{eta}[][][1]{$    \scriptstyle{{\xi}} \,  $}
 \includegraphics[width=20mm,height=15mm]{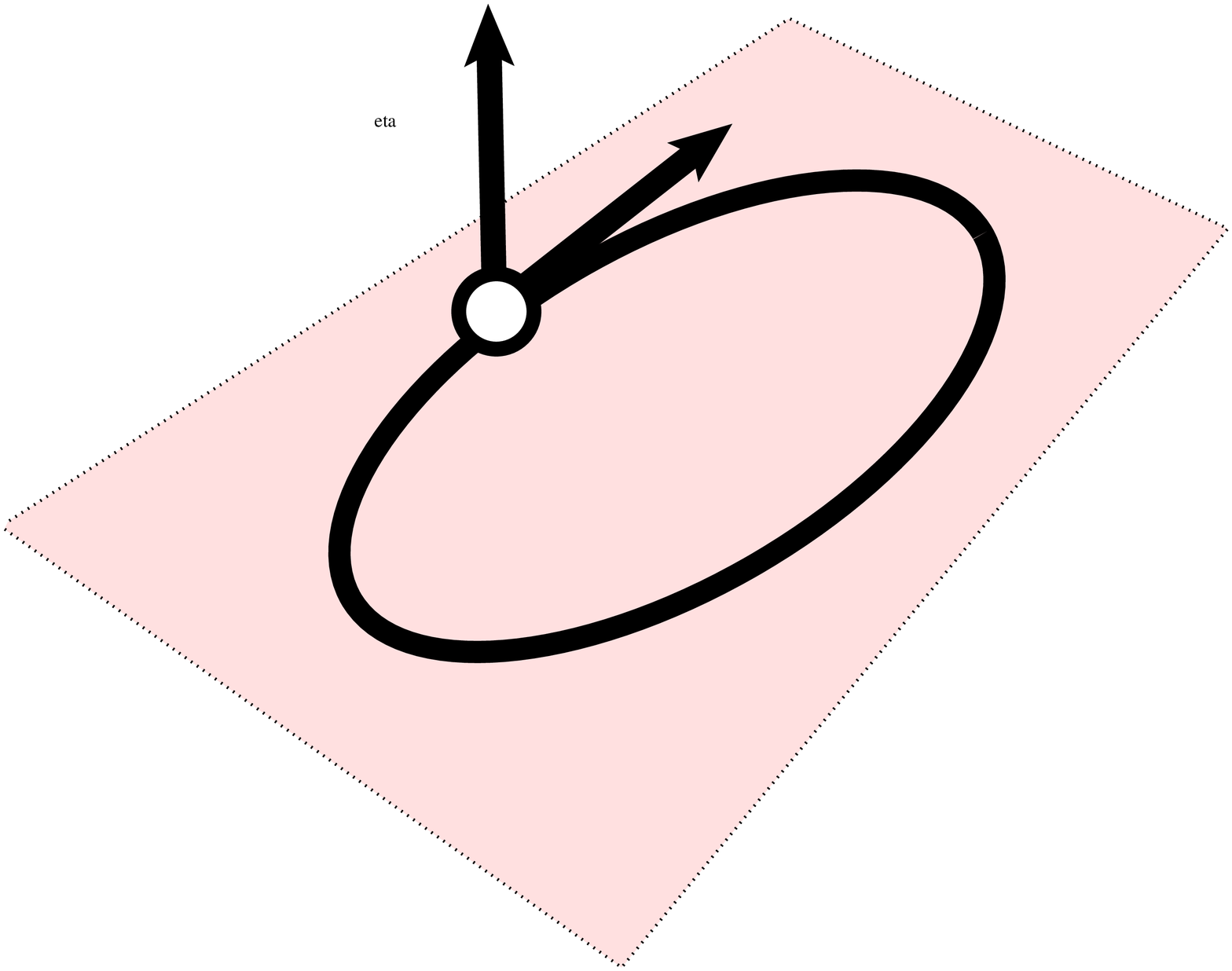}  
 \\
\end{tabular}
  \\ \hline
$\mathcal J^4_5$  
& 2 &  $\mathbb C\times \mathbb C$ & $\big(xy,x^2,yz,xt\big)$   &    
\begin{tabular}{c} \vspace{-0.2cm}\\
\psfrag{tau}[][][1]{$    \scriptstyle{{\tau}} \,  $}
 \includegraphics[width=20mm,height=15mm]{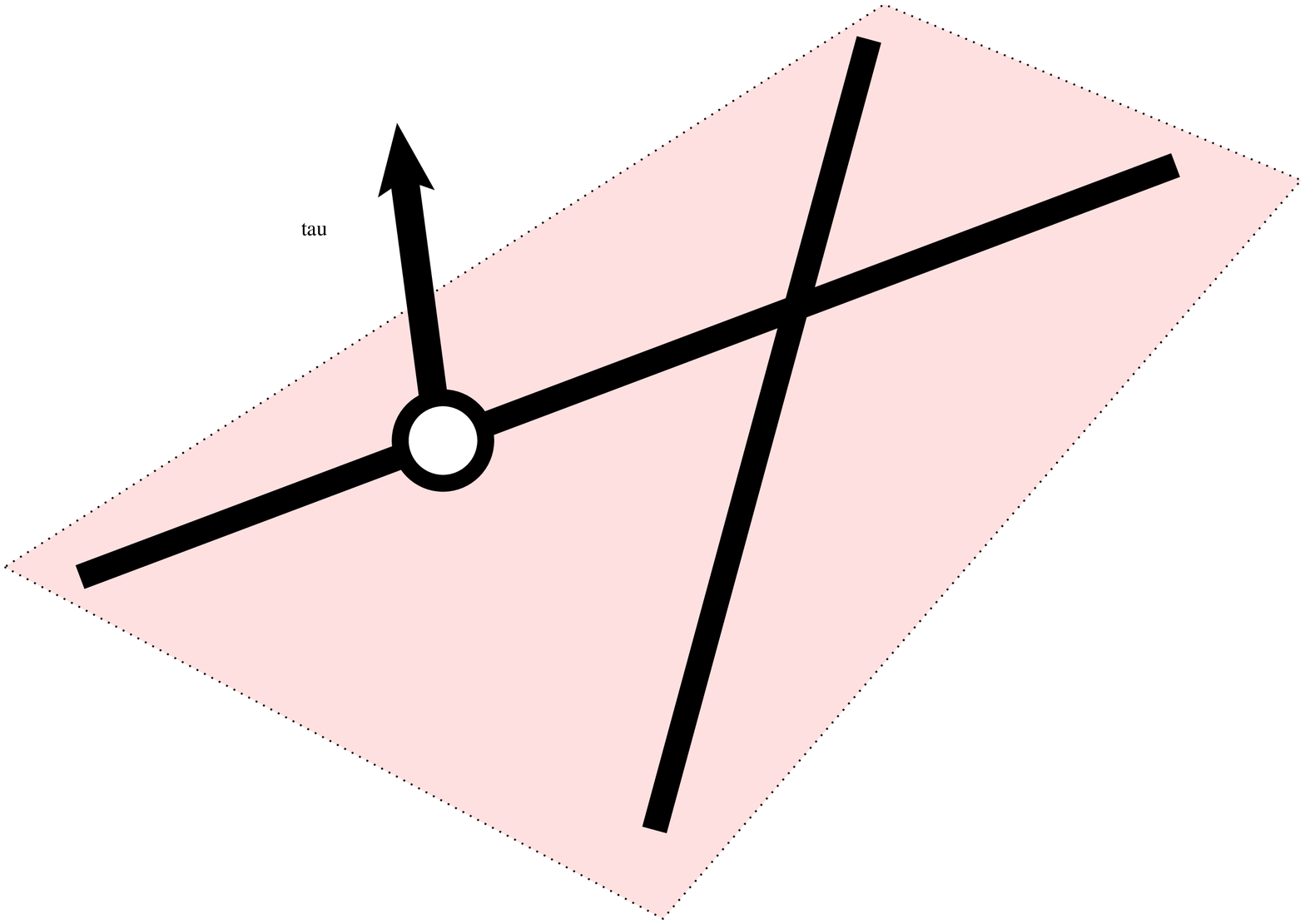}  
 \\
\end{tabular}  \\ 
\hline \hline
$\mathcal J^4_6$
&  3&  $\mathbb C$   & $
\big(x^2\,, -xy\,, -xz\,,\, 2yz-xt\big) 
$   &   
\begin{tabular}{c} \vspace{-0.2cm}\\
\psfrag{chi}[][][1]{$    \scriptstyle{{\chi}} \; \,  $}
 \includegraphics[width=20mm,height=15mm]{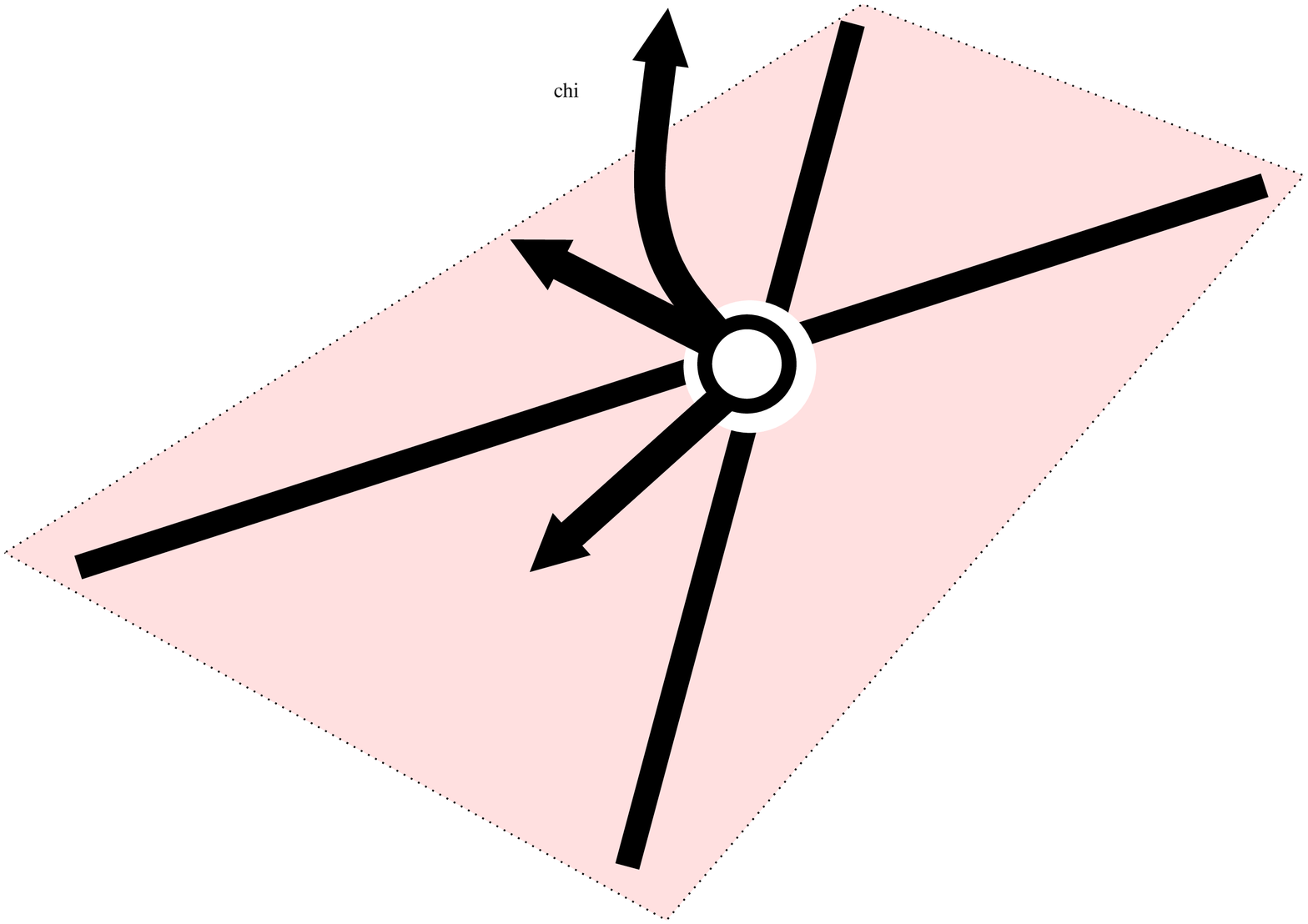}  
 \\
\end{tabular}  \\ 
\hline
$\mathcal J^4_7$  
&  3  & $\mathbb C$  & $
\big(x^2\,, -xy\,, -xz\,, \,{y}^2-xt\big) 
$   &  
\begin{tabular}{c} \vspace{-0.2cm}\\
\psfrag{eta}[][][1]{$    \scriptstyle{{\eta}} \;\,   $}
 \includegraphics[width=20mm,height=15mm]{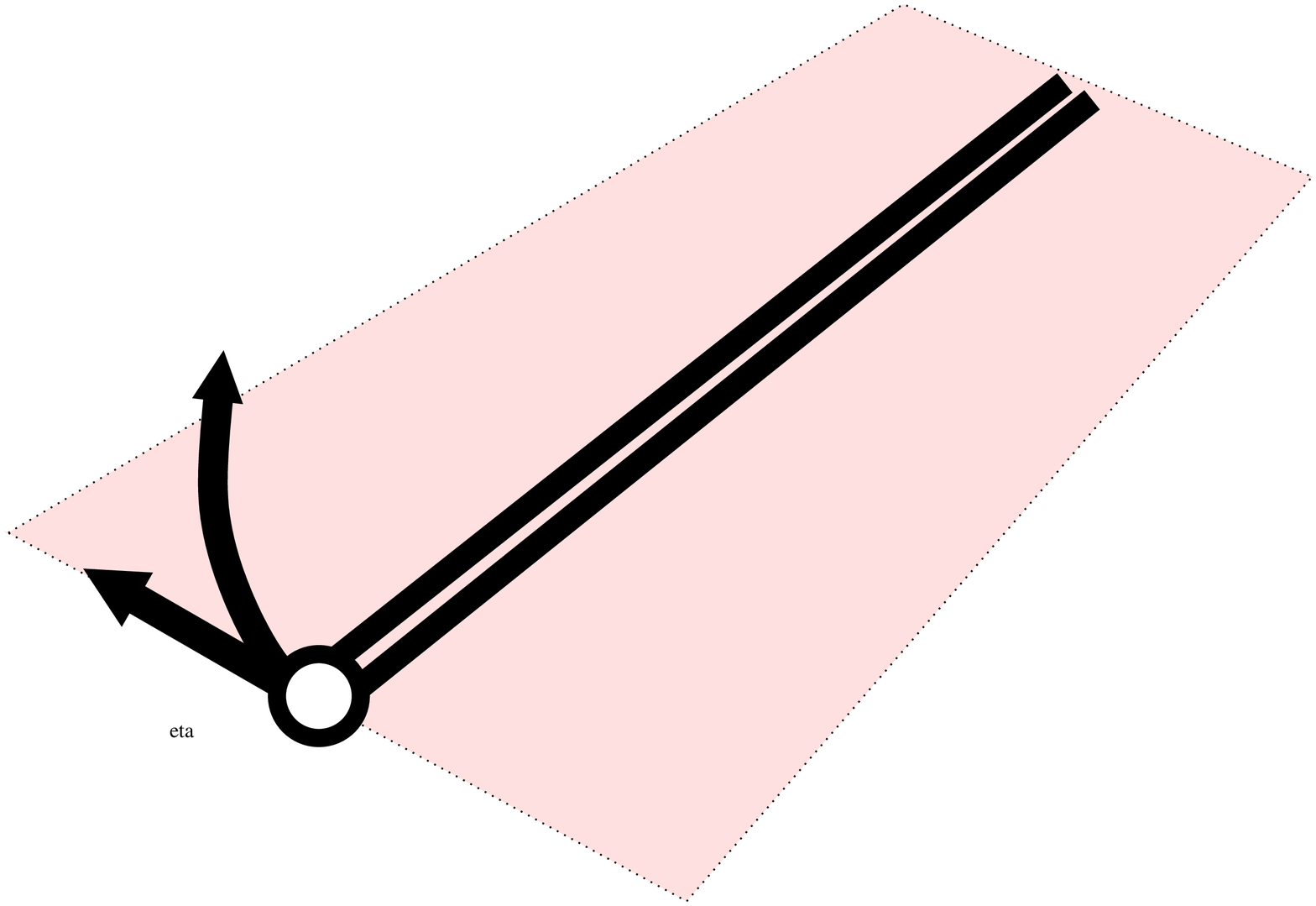}  
 \\
\end{tabular} \\ \hline
\end{tabular} \vspace{0.35cm}
\label{Ta:dim4}
\caption{Classification of rank 3 Jordan algebras of dimension four or equivalently,    of quadro-quadric Cremona transformations of $\mathbb P^3$.}
\end{table}


The classification given by the table above is easy to obtain. 
 Let $J$ be a rank 3 Jordan algebra of dimension 4. 
First of all, when $R={\rm {\rm Rad}}(J)$ is trivial, $J$ is semi-simple hence it is the direct product of $\mathbb C$ with  $\mathcal J^3_{q,2}$.  \smallskip 

When  $\dim R=1$, the semi-simple part $J_{\rm ss}$ of $J$ has dimension 3. It cannot be of rank 1 (it has dimension 3) nor of rank 2: in the latter case, it would be isomorphic to $\mathcal J^3_{q,2}$ that does not admit a cubic form. Hence 
$J_{\rm ss}$ has rank 3 and can be assumed to be the direct product of three copies of $\mathbb C$.  Clearly $R^2=0$ hence the Jordan product is completely determined by the three complex numbers $\alpha_1,\alpha_2$ and $\alpha_3$  such that $e_ir=\alpha_i r$ for $i=1,2,3$, where $(e_1,e_2,e_3)$ stands for the standard basis of $J_{\rm ss}=\mathbb C\times \mathbb C\times \mathbb C$ and where $r$ is a non trivial element of $R$.  It is easy to prove that there exists  only one possibility (up to isomorphisms) for the $\alpha_i$'s, namely $\alpha_1=\alpha_2=1$ and $\alpha_3=0$.   One easily verifies that the corresponding  algebra  is isomorphic to $\mathcal J^4_2=\mathbb C\times 
\mathcal J^3_{q,1}$.\mk

We now assume  $\dim R=2$.  In this case the semi-simple part is $J_{\rm ss}=\mathbb C\times \mathbb C$. Let $(e_1,e_2)$ stands for the image of the standard basis of $\mathbb C^2=J_{\rm ss}$ in an embedding $J_{\rm ss}\hookrightarrow  J$.
This is a set of pairwise irreducible idempotents.   Let $J=J_{11}\oplus J_{12}\oplus J_{22}$ be the associated Peirce decomposition.  \sk

If $J_{12}\cap R=(0)$ there are two possibilities:  either  $(a)$ $R\subset J_{22}$,   or  $(b)$ $R= {\rm Rad}(J_{11})\oplus   {\rm Rad}(J_{22})$ with $\dim {\rm Rad}(J_{ii})\neq 0$ for $i=1,2$.   In any case, one has $J_{12}=(0)$  and $J$ is isomorphic to the direct product $J_{11}\times J_{22}$ 
by \eqref{E:Jij}.  If $R\subset J_{22}$, then 
$J_{11}$ is isomorphic to $\mathbb C$ for dimensional reasons and 
 $J_{22}$ is a Jordan algebra of rank 2, dimension 3 with radical of dimension 2. Hence $J_{22}\simeq \mathcal J_{q,0}^3$ 
 and $J$ is isomorphic to the algebra $\mathcal J^4_3= \mathbb C \times \mathcal J_{q,0}^3$.   Case $(b)$ does not occur. Indeed, in this case  $J_{ii}$ had dimension 2 and rank 2 for $i=1,2$  so that  $J\simeq J_{11}\times J_{22}$  would have rank 4, contradicting  our assumption.\sk
 
 We now treat the case   $\dim (J_{12}\cap R)=1$.  Let $a$ be such that $R\cap J_{12}=\mathbb C a$.  By \eqref{E:RJii}, one can assume that $\dim {\rm Rad}(J_{11})=1$ and $\dim {\rm Rad}(J_{22})=0$. Let $b$ such that ${\rm Rad}(J_{11})=\mathbb Cb$.  For dimensional reasons, it follows that $J_{11}=\mathbb C e_1\oplus \mathbb Cb$, $J_{12}=\mathbb C a$ and $J_{22}=\mathbb Ce_2$.  Since $b^2\in R\cap J_{11}=\mathbb Cb$ and because $b^3=0$ (recall that  $J$ has rank 3 and  that $b\in {\rm Rad}(J)$) we get $b^2=0$.  The structure of $J$ will be completely determined when the two product $ab$ and $a^2$ are. 
 According to \eqref{E:Jii}, one has $ab\in J_{12}\cdot J_{11}\subset J_{12}=\mathbb Ca$ hence there exists $\kappa_1\in \mathbb C$ such that $ab=\kappa_1a$. Hence $L_b(a)=\kappa_1 a$ thus $\kappa_1=0$ because the multiplication $L_b$ by $b$ is also nilpotent, as one easily sees.  Finally, one has $a^2\in R\cap J_{12}^2\subset R\cap (J_{11}\oplus J_{22})=\mathbb Cb$ so that there exists $\kappa\in C$ such that $a^2=\kappa b$.  In the coordinate system  associated to the basis $(e_1,b,a,e_2)$ of $J$, the Jordan product is given by : 
  $$
 (x_1,\beta,\alpha,x_2)\cdot (x_1',\beta',\alpha',x_2')=\Big(x_1x_1',
 \beta x_1'+x_1\beta'+
 \kappa \alpha\alpha'    ,    
 \frac{1}{2}\big( (x_1+x_2)\alpha'+\alpha(x_1'+x_2')
 \big)
   ,   x_2x_2' \Big).
 $$
 One verifies that the corresponding Jordan algebras have rank 3 and that for any $\kappa\in \mathbb C$, the adjoint is given by 
 $$
  (x_1,\beta,\alpha,x_2)^\#=\big( x_1x_2, \kappa \alpha^2-\beta x_2   , -\alpha x_1   ,   x_1^2\big). 
 $$
Up to isomorphisms, there are only two cases to be considered, namely $\kappa=1$ and $\kappa=0$.  We denote respectively by $\mathcal J^4_4$ and $\mathcal J_5^4$ the two corresponding Jordan algebras. 
 \mk 
 
Finally, when $\dim R=3$ the multiplicative structure  of $J$ is completely determined by  that of $R$. Thus there are two possibilities in this case, denoted by $\mathcal J^4_6$ and $\mathcal J^4_7$ in the table above, that correspond respectively to the  unitalizations of the nilalgebras $\mathcal R_2^3$ and $\mathcal R_1^3$ of T{\small{ABLE}} 1. 
\mk

\subsection{\bf Quadro-quadric Cremona transformations of $\mathbb P^4$ and rank 3 Jordan algebras of dimension 5}
\label{S:Bir22P4}
\label{S:JordanDim5}

In this section, we use the same strategy  to obtain the main result of this paper: a complete and explicit classification of quadro-quadric Cremona transformations of $\mathbb P^4$.

\subsubsection{The three generic quadro-quadric Cremona transformations of $\mathbb P^4$}
\label{S:qqP4-Bruno-Verra}
In \cite{brunoverra}, Bruno and Verra give a modern proof of the following result firstly proved by Semple in \cite{semple}, providing the classification of  the base locus schemes of general elements of ${\bf Bir}_{2,2}(\mathbb P^4)$.

\begin{thm} 
\label{T:BV1}
The base locus scheme of a general quadro-quadric Cremona transformation of $\mathbb P^4$ is projectively equivalent to  one of the  subschemes $\mathcal B_I,\mathcal B_{II}$ an $\mathcal B_{III}$ of $\mathbb P^4$ where: 
\begin{enumerate}
\item[(I)] $\mathcal B_I$ is the disjoint union   of a smooth quadric surface  $Q$ with  a point $p$ lying outside the hyperplane $\langle Q \rangle$;\smallskip 
\item[(II)] $\mathcal B_{II}$ is the union  of a 2-plane $\pi$ with  two skew lines $\ell_1$ and $\ell_2$,  each one intersecting $\pi$  in one point;\smallskip 
\item[(III)] $\mathcal B_{III} $ is the scheme theoretic  union 
 of    a double line $\mathcal L$ in a hyperplane  $H$  with a smooth conic $C$   tangent to $H$ at the point $C\cap \mathcal L_{\rm red}$.
\end{enumerate}
\end{thm}

Let us recall  that the type $T_f$ of a Cremona transformation $f:\mathbb P^n\dashrightarrow \mathbb P^n$ is (the specification of) the irreducible component of ${\rm Hilb}(\mathbb P^n)$ containing $\mathcal B_f$. 
The proof of the previous result also implies  that there are exactly  three types of  elements in ${\bf Bir}_{22}(\mathbb P^4)$, which will be denoted by $I,II$ and $III$. Moreover,
for any type $T\in \{ I,II,III   \}$, two general elements in the irreducible component of ${\rm Hilb}(\mathbb P^4)$ containing  $\mathcal B_T$ are projectively equivalent.
 This implies that  to  $T$ there corresponds what we call the {\it generic   Cremona transformation} $f_T\in 
{\bf Bir}_{22}(\mathbb P^4)$.  Normal forms   for $f_I,f_{II}$ and $f_{III}$ as well as the corresponding multidegrees are given in the  table below. 

\begin{table}[H]
  \centering
  \begin{tabular}{|c|c|l|c|}\hline
  {\bf Type}   $\boldsymbol{T}$ &    {\bf Base locus}   $\boldsymbol{\mathcal B_T}$   &\quad     {\bf   Cremona involution} $\boldsymbol{f_T(x,y,z,t,u)}$  & {\bf Multidegree}
  \\
   \hline
   \hline
 $I$  &      $Q\sqcup \{p\}$        &  \, $\big( y^2+z^2+t^2+u^2, xy,-xz,-xt,-xu\big)$        &    $(2,2,2)$    \\ \hline 
 $II$    &    $\pi\cup \ell_1 \cup \ell_2$              &  \,  $   (    yz \, , \, xz \, , \, xy\, , \, -zt\, , \, -yu    ) $ &   $(2,3,2)$\\\hline
   $III$  &     $  \mathcal L \cup C$            & \,  $   (xy, x^2,-yz+u^2, -yt,-xu)  $  
   &    $(2,4,2)$   \\    \hline
           \end{tabular}
\label{Tab:jordan}
\caption{The three generic quadro-quadric Cremona transformations of $\mathbb P^4$.}
\end{table}

The method used by Semple (and later independently by Bruno and Verra) uses induction on the dimension and can be roughly described as follows:  given   a point $o\in \mathbb P^m$ where a given  $f\in {\bf Bir}_{2,2}(\mathbb P^m)$ is an isomorphism, one takes  a general quadric $Q$ passing through $o$ and belonging to  the homaloidal  linear system $f^{-1}\lvert  \mathcal O_{\mathbb P^m}(1) \lvert$. 
 Then $P=f(Q)$ is a general hyperplane through $f(o)$ and if $\pi:Q\dashrightarrow \mathbb P^{m-1}$ stands for the restriction to $Q$ of the linear projection from $o$, one proves that $h=\pi\circ (f^{-1}\lvert _P): P\dashrightarrow  \mathbb P^{m-1}$ is birational, of bidegree $(2,n)$ with $n\in \{2,3,4\}$ and that its base locus scheme $\mathcal B_h$ is the union of $f(o)$ with the intersection of $P\simeq \p^{m-1}$ with the base locus scheme $\mathcal B_{f^{-1}}$ of $f^{-1}$ (cf. \cite[Section 2]{brunoverra} for more details). \sk 
 
As remarked by Bruno and Verra, the base locus $\mathcal B_h$ of a Cremona  map $h$ obtained by Semple's construction described above cannot be too special since it  contains $f(o)$ as an isolated point. As we shall see below and as it was already shown by the classification of ${\bf Bir}_{2,2}(\p^3)$,  a lot of interesting examples appear by degenerating   the isolated point to an infinitely near point of  the support of the general base locus scheme. Moreover, 
 there is no general description of quadratic Cremona transformations of degree $(2,3)$ and $(2,4)$ in dimension greater than 5 so that  Semple's method can work effectively only in dimension at most  $4$. However, Semple's approach yields  quite easily that any quadro-quadric Cremona transformations of $\mathbb P^4$ is a  degeneration of one of the $f_T$'s in the table above, although it does not say anything precise on the possible degenerations.
 In other terms, Semple's inductive method does not allow to obtain the complete lists of Cremona transformations of $\mathbb P^n$ for  $n\geq 4$. 
 \sk 
 
On the contrary the classification of  5-dimensional Jordan algebras of rank 3 is not difficult. By using the material of Section \ref{S:notions-tools-Jordan-algebras}, it  
amounts to elementary but a bit lengthy exercises in linear algebra.

\subsubsection{On the classification of rank 3 Jordan algebras of dimension 5}
We now classify Jordan algebras $J$ of rank 3 and of dimension 5. We shall  consider the different subcases according  to the possible dimension of the radical $R$ of $J$.
\sk 

Let us begin with the case $\dim R=4$.  Then  $J$ is the  unitalization of one of the 
 five nilalgebra $\mathcal R^4_1,\ldots,\mathcal R^4_5$ in T{\small{ABLE}} 1 (the case when $R^2=0$  would imply that $J$ has rank 2 hence it has to be excluded).  For $i=1,\ldots,5$, let us  denote by $\mathcal J^5_i$ the algebra with $\mathcal R^4_i$ as radical.  The  $\mathcal J^5_i$'s  are associative algebras of rank 3. 
 Explicit expressions for the adjoints of these algebras in the coordinate system  $(x,y,z,t,u)$ associated to the basis $(e,v_1,\ldots,v_4)$ are given in   T{\small{ABLE}} \ref{Tab:jordanP4} below. 
 \sk 

Let us now consider the case  $\dim R=3$. This case  is not more complicated than the other ones but requires several pages of elementary arguments of linear algebra that  are outlined in \cite[Section 5.0.4]{PIRIO}. Since these details do not present any real conceptual interest, we have decided to exclude them and  to state the corresponding results.   
  Let us define  $\mathcal J_6^5,\ldots,\mathcal J^5_{12}$ as the algebras whose multiplicative tables in a certain basis  denoted by $(e_1,e_2,a,b,d)$ are given in T{\small{ABLE}} 5 below. 
  The algebras $\mathcal J_k^5$'s for $k\in \{   6,\ldots,12\}$ are rank 3 Jordan algebras with 3-dimensional radicals. Moreover, any  Jordan algebra of this type is isomorphic to exactly one of these seven algebras. 
  \begin{table}[H]
\centering
\begin{tabular}{ccc}
\begin{tabular}{c}
  \begin{tabular}{c|c|c|c|c|c|}
 $\mathcal J_6^5$  & $e_1$ & $e_2$ &  $a$ & $b$ &  $d$      \\  \hline
$e_1$ & $e_1$   &   &$a $ & $\frac{1}{2}b$   &  $\frac{1}{2}d $   \\  \hline
$e_2$ &    &  $e_2$  & $  $ & $\frac{1}{2}b$   &  $\frac{1}{2}d $       \\  \hline
$a$ &   $a $ &$  $   &    &    &     \\  \hline
$b$ & $\frac{1}{2}b $   & $\frac{1}{2}b $  &    &    &     \\  \hline
$d$ & $\frac{1}{2}d $   &   $\frac{1}{2}d $&    &    &     \\  \hline
 \end{tabular} 
 \end{tabular}
 \quad & \quad 
  \begin{tabular}{c|c|c|c|c|c|}
$\mathcal J_7^5$   & $e_1$ & $e_2$ &  $a$ & $b$ &  $d$      \\  \hline
$e_1$ & $e_1$   &   &$a $ & $\frac{1}{2}b$   &  $\frac{1}{2}d $   \\  \hline
$e_2$ &    &  $e_2$  & $  $ & $\frac{1}{2}b$   &  $\frac{1}{2}d $       \\  \hline
$a$ &   $a $ &$  $   &    &    &     \\  \hline
$b$ & $\frac{1}{2}b $   & $\frac{1}{2}b $  &    & $a$     &     \\  \hline
$d$ & $\frac{1}{2}d $   &   $\frac{1}{2}d $&    &    &     \\  \hline
 \end{tabular}
 \quad & \quad 
    \begin{tabular}{c|c|c|c|c|c|}
 $\mathcal J_8^5$  & $e_1$ & $e_2$ &  $a$ & $b$ &  $d$      \\  \hline
$e_1$ & $e_1$   &   &$a $ & $\frac{1}{2}b$   &  $\frac{1}{2}d $   \\  \hline
$e_2$ &    &  $e_2$  & $  $ & $\frac{1}{2}b$   &  $\frac{1}{2}d $       \\  \hline
$a$ &   $a $ &$  $   &    &    &     \\  \hline
$b$ & $\frac{1}{2}b $   & $\frac{1}{2}b $  &    &  $a$  &     \\  \hline
$d$ & $\frac{1}{2}d $   &   $\frac{1}{2}d $&    &    &   $a$  \\  \hline
 \end{tabular}  \medskip    \bigskip\\
   \begin{tabular}{c|c|c|c|c|c|}
  $\mathcal J_9^5$  & $e_1$ & $e_2$ &  $a$ & $b$ &  $d$      \\  \hline
$e_1$ & $e_1$   &   &$a $ & $\frac{1}{2}b$   &  $\frac{1}{2}d $   \\  \hline
$e_2$ &    &  $e_2$  & $  $ & $\frac{1}{2}b$   &  $\frac{1}{2}d $       \\  \hline
$a$ &   $a $ &$  $   &    & $d$   &     \\  \hline
$b$ & $\frac{1}{2}b $   & $\frac{1}{2}b $  &  $d$  &    &     \\  \hline
$d$ & $\frac{1}{2}d $   &   $\frac{1}{2}d $&   $$ &    &     \\  \hline
 \end{tabular}
 \quad & \quad 
  \begin{tabular}{c|c|c|c|c|c|}
 $\mathcal J_{10}^5$   & $e_1$ & $e_2$ &  $a$ & $b$ &  $d$      \\  \hline
$e_1$ & $e_1$   &   &$a $ & $ b $   &  $\frac{1}{2}d $   \\  \hline
$e_2$ &    &  $e_2$  & $  $ & $  $   &  $\frac{1}{2}d $       \\  \hline
$a$ &   $a $ &$  $   &    &    &   $$  \\  \hline
$b$ & $  b$   & $   $  &    &    &     \\  \hline
$d$ & $\frac{1}{2}d $   &   $\frac{1}{2}d $&   $ $ &    &   $ $   \\  \hline
 \end{tabular}
 \quad & \quad 
  \begin{tabular}{c|c|c|c|c|c|}
  $\mathcal J_{11}^5$ & $e_1$ & $e_2$ &  $a$ & $b$ &  $d$      \\  \hline
$e_1$ & $e_1$   &   &$a $ & $ b $   &  $\frac{1}{2}d $   \\  \hline
$e_2$ &    &  $e_2$  & $  $ & $  $   &  $\frac{1}{2}d $       \\  \hline
$a$ &   $a $ &$  $   &    &    &   $$  \\  \hline
$b$ & $  b$   & $   $  &    &    &     \\  \hline
$d$ & $\frac{1}{2}d $   &   $\frac{1}{2}d $&   $ $ &    &   $a $   \\  \hline
 \end{tabular}     \medskip    \bigskip\\ 
  \quad & \quad 
   \begin{tabular}{c|c|c|c|c|c|}
  $\mathcal J_{12}^5$ & $e_1$ & $e_2$ &  $a$ & $b$ &  $d$      \\  \hline
$e_1$ & $e_1$   &   &$ a$ & $  b$   &  $  d $   \\  \hline
$e_2$ &    &  $e_2$  & $  $ & $  $   &  $ $       \\  \hline
$a$ &   $ a $ &$  $   &    &    &   $$  \\  \hline
$b$ & $ b $   & $   $  &    &    &     \\  \hline
$d$ & $ d$   &   $ $&   $ $ &    &   $  $   \\  \hline
 \end{tabular} 
   \quad & \quad 
  \end{tabular}
\caption{Multiplication tables for 5-dimensional rank 3 Jordan algebras with 3-dimensional radical (where an empty entry means that the corresponding product is equal to zero).}
\end{table}

Explicit expressions in the coordinate system  $(x,y,z,t,u)$ associated to the basis $(e_1,e_2,a,b,d)$ for the adjoints of the seven  algebras $\mathcal J_6^5,\ldots,\mathcal J^5_{12}$  are given in  T{\small{ABLE}} \ref{Tab:jordanP4} below.  
\sk

Let us assume now  $\dim R=2$. 
In this case, $ J_{\rm ss}$ is a semi-simple Jordan algebra of dimension 3 and of rank 2 or 3. Since it admits  a cubic norm, it is necessarily of rank 3 and  one can assume that   $ J_{\rm ss}=\mathbb C\times \mathbb C\times \mathbb C$.
 The standard basis $(e_1,e_2,e_3)$  of $J_{\rm ss}=\mathbb C^3$ gives an irreducible decomposition $e=e_1+e_2+e_3$  of the unity  by pairwise primitive orthogonal idempotents. Let us consider the corresponding Peirce decomposition 
  (\ref{E:PeirceDecomposition}) and  discuss according to the value of 
  $$\theta=\dim \big(R\cap \sum_{i<j }  J_{ij}\big)\in \{ 0,1,2  \}.$$

If $\theta=0$ then $ J=\oplus_{i=1}^3 J_{ii}$ so it follows from Proposition \ref{P:Peirce}
 that $ J$ is the direct product of the  $ J_{ii}$'s.  Since at least  one of these three Jordan algebras 
 has dimension $>1$, it follows   that the rank of $J$ is at least 
$1+1+2=4$, contradicting the assumption $\rk(J)=3$. Thus the case when $\theta=0$   does not occur.

Assume now that  $\theta=1$. One can suppose that $R={\rm Rad}(J_{ii})\oplus  J_{23}$ for a certain $i\in\{1,2,3\}$,  with ${\rm Rad}( J_{ii})$  of dimension 1.   
If  $i\in \{ 2,3\}$, say $i=3$, then set $ J'= J_{22}\oplus  J_{23}\oplus  J_{33}$. 
It follows from Proposition \ref{P:Peirce} that  $J'$ is a Jordan algebra since  $ J_{11}=\mathbb C\, e_1$. Moreover,  from $  J' \cdot  J_{11}=0$, it comes that $ J= J_{11} \times    J' $ for dimensional reasons.  Here $ J'$ is a rank 2 Jordan algebra of dimension 4 with a 2-dimensional radical. It can be proved that with respect to a suitable basis, the product of 
 $J'$ is given  by $(y,z,t,u)\cdot (y',z',t',u')=(yy'+zz',yz'+zy',yt'+ty',yu'+uy')$.
 Moreover, one has $e_2=\frac{1}{2}(1,1,0,0)$ and $e_3=\frac{1}{2}(1,-1,0,0)$  in the corresponding  coordinates.  But then easy computations show that  $  J'(e_j)=J_{jj}=\mathbb C\, e_j$  for $j=2,3$,  contradicting the assumption 
$\dim ( {\rm Rad}( J_{33}))=1$.  Thus this case does not occur. 

Let us now  assume that $i=1$, {\it i.e.} that $\dim {\rm Rad}( J_{ii})=1$.  As above, one proves that $ J= J_{11} \times    J' $, but now with $\dim  J_{11}=\dim   J'=2$. This would imply that $ J$ has rank $2+2=4$, excluding also this case.\sk

We now treat the case  $\theta=2$.  Then one has ${\rm Rad}( J_{ii})=0$ for $i=1,2,3$ so that $R= J_{12}\oplus  J_{13}\oplus J_{23}$.  Assume first that  no $ J_{ij}$'s (for $i<j$) has dimension 2.  
Thus one can assume that $J_{23}=0$ and that $J_{12}$ and $J_{13}$ are 1-dimensional. Then there exists $a,b$ such that $ J_{12}=\mathbb C\, a$,  $ J_{13}=\mathbb C\, b$. Since $R=J_{12}\oplus J_{13}$,  using 
Proposition \ref{P:Peirce},  one deduces easily that the multiplication table of $ J$ is the following:   
\begin{equation*}
  \begin{tabular}{c|c|c|c|c|c|}
   & $e_1$ & $e_2$ &  $e_3$ & $a$ &  $b$      \\  \hline
$e_1$ & $e_1$   &   &$  $ & $ \frac{1}{2}a $   &  $\frac{1}{2}b $   \\  \hline
$e_2$ &    &  $e_2$  & $  $ & $   \frac{1}{2}a $   &  $ $       \\  \hline
$e_3$ &   $ $ &$  $   &  $e_3$  &    &   $  \frac{1}{2}b $  \\  \hline
$a$ & $  \frac{1}{2}a $   & $ \frac{1}{2}a  $  &    &    &     \\  \hline
$b$ & $\frac{1}{2}b $   &   $ $&   $ \frac{1}{2}b$ &    &   $ $   \\  \hline
 \end{tabular}
\end{equation*}\smallskip 

This is the multiplication table of a rank 3 Jordan algebra that will be denoted by ${\mathcal J}_{13}^5$.  
The expression of the adjoint in the coordinate system associated to the basis $(e_1,e_2,e_3,a,b)$ of ${\mathcal J}_{13}^5$ is given in T{\small{ABLE}} \ref{Tab:jordanP4}.

Finally, assume that 
one of the spaces in the  decomposition $R= J_{12}\oplus  J_{13}\oplus  J_{23}$ (say $ J_{23}$) has dimension 2. 
Then $R=J_{23}$, $ J_{1j}=0$ for $j=2,3$ thus $J_{11}=\mathbb C\, e_1$. 
As above, one proves that  $ J$ is isomorphic to the direct product  $ J_{11} \times J'$.  Since $J$ has rank 3, $J'$ has rank 2 and $R= J_{23}={\rm Rad}(J')$ is 2-dimensional. Thus $J$ is isomorphic to the Jordan algebra 
${\mathcal J}_{14}^5=\mathbb C\times \mathcal J^4_{q,1}$. The associated Jordan adjoint (in standard coordinates) is given in T{\small{ABLE}} \ref{Tab:jordanP4}.\sk 
 
Let us now consider the case  $\dim R=1$.  In this case, $J_{\rm ss}$ has  rank  2 or 3. 
Moreover, since it admits a non-trivial cubic norm, it cannot be isomorphic to $\mathcal J^4_{q,3}$ (which does not admit any)  hence  it is necessarily of rank 3.  It follows that  one can assumes that $ J_{\rm ss}$  is the direct product $\mathbb  C\times ( \mathbb C\oplus W )$ where $\mathbb C\oplus W=\mathcal J^3_{q,2}$. Let $a$ be a non-trivial element of $R$. Since $a^3=0$ and $a^2\in R^2\subset R= \mathbb C\, a$, it follows that $a^2=0$.  Set $W=\mathbb C^2$ and let $(x,y,z,t,u) $ be the usual system of coordinates on $J=(\mathbb  C\times ( \mathbb C\oplus W ))\oplus \mathbb C\, a$ (these coordinates  are such that $(0,y,z,t,0)\cdot (0,\tilde y,\tilde z,\tilde t,0)=(0,y\tilde y+z\tilde z+t\tilde t,y\tilde z+z\tilde y,y\tilde t+t\tilde y,0)$ for every $y,\tilde y,z,\tilde z,t,\tilde t\in \mathbb  C$).  Then $e_1=(1,0,0,0,0)$, $e_2=\frac{1}{2}(0,1,1,0,0)$ and  $e_3=\frac{1}{2}(0,1,-1,0,0)$ are pairwise orthogonal primitive idempotents of $J$ such that $e=e_1+e_2+e_3$.

 The direct sum  $J'=\oplus_{i,j=2}^3 J_{ij}$ is a subalgebra of $J$. If $a\in J'$, one proves   that $J$ is isomorphic to the direct product of $J_{11}$ with $J'$.  If $\dim J_{11}>1$, we would have 
 $3={\rm rk}(J)={\rm rk}(J_{11})+{\rm rk}(J')\geq 2+2=4$, a contradiction. Hence $J_{11}\simeq \mathbb C$  and $J'$ is a  rank 2 Jordan of dimension  4 with a 1-dimensional  radical. Hence $J$ is isomorphic to the Jordan algebra $ \mathcal J_{15}^5=\mathbb C\times \mathcal J_{q,2}^4$ and there is  then no difficulty to get an explicit expression for the adjoint  in some coordinates  (see  T{\small{ABLE}} \ref{Tab:jordanP4} below).

Suppose now that   $a\in J_{11}\oplus J_{12}\oplus J_{13}$.  The case  $a\in J_{11}$ does not occur (this would imply that $J$ is isomorphic to $J_{11}\times J'$ hence we would have ${\rm rk}(J)={\rm rk}(J_{11})+{\rm rk}(J')\geq 2+2$, a contradiction)   so that   $a\in J_{12}\oplus J_{13}$.  By  exchanging $e_2$ and $e_3$, if necessary, one can assume that $a\in J_{12}$.  
Set $\tilde e=(0,0,0,1,0)\in J_{\rm ss}= (\mathbb  C\times ( \mathbb C\oplus W ))$. One verifies that $\tilde e\in J_{23}$ and that $\tilde e^2=e_2+e_3$. Moreover,  since $\tilde{e}\cdot a\in J_{23}\cdot J_{12}\subset J_{13}$ 
and since $J_{13}=0$ (for dimensional reasons),  it follows that $\tilde e\cdot a=0$.  From this one deduces that the multiplication table of $J$ in the basis  $(e_1,e_2,e_3,\tilde e,a)$ is the following:  
\begin{equation*}
  \begin{tabular}{c|c|c|c|c|c|}
   & $e_1$ & $e_2$ &  $e_3$ & $\tilde e$ &  $a$      \\  \hline
$e_1$ & $e_1$   &   &$  $ & $  $   &  $\frac{1}{2}a $   \\  \hline
$e_2$ &    &  $e_2$  & $  $ & $   \frac{1}{2}\tilde e $   &  $   \frac{1}{2}a$       \\  \hline
$e_3$ &   $ $ &$  $   &  $e_3$  &$\frac{1}{2}\tilde e$    &   $ $  \\  \hline
$\tilde e$ & $ $   & $ \frac{1}{2}\tilde e  $  & $ \frac{1}{2}\tilde e $   & $e_2+e_3$
   &     \\  \hline
$a$ & $\frac{1}{2}a $ &   $ \frac{1}{2}a$ &    &   $ $  &   $ $   \\  \hline
 \end{tabular}\mk
\end{equation*}

By direct computations, one verifies that the multiplicative product defined by this table does not satisfied the Jordan identity.  Thus the case when  $a\in J_{11}\oplus J_{12}\oplus J_{13}$ does not occur. 
\sk

Finally, when $\dim R=0$, $J$ is semi-simple hence  is isomorphic to the direct product $\mathcal J_{16}^5=\mathbb C\times \mathcal J^4_{q,3}$.
\mk

We have thus finally obtained the  classification of rank 3 Jordan algebras of dimension 5:

\begin{thm}
A rank 3 Jordan algebras of dimension 5 is isomorphic to exactly one of the  sixteen algebras $\mathcal J^5_k$'s described above.
\end{thm}

\subsubsection{Classification of quadro-quadric Cremona transformations of $\mathbb P^4$}
 Using the $JC$-equivalence, one deduces from 
 the previous result the complete classification of quadro-quadric Cremona transformations on $\mathbb P^4$.  Furthermore, one verifies that the Hilbert polynomials of the base locus schemes $\mathcal B_I,\mathcal B_{II}$ and $\mathcal B_{III}$ are respectively 
$$
 h_{I}=P_0-P_1+2P_2\,,  \qquad 
  h_{II}=- 2P_0  + 2P_1  + P_2
\qquad \mbox{ and }
\qquad h_{III}=  -5P_0+  5 P_1.  
 $$
In particular, these are distinct thus the type of a   quadro-quadric Cremona transformation of $\mathbb P^4$ given explicitly can be determined easily by computing  the Hilbert polynomial of its   base locus scheme (using the software {\it MacCaulay2} \cite{MC2} for instance).  
 Similarly, there is no difficulty to determine the multidegrees of such a $f$: 
the intersection of two distinct quadrics in $f^{-1}\lvert  \mathcal O_{\mathbb P^4}(1)  \lvert$ is the scheme theoretic  union of $\mathcal  B_f$ with another scheme whose  degree is  the integer $k$ such that  ${\rm mdeg}(f)=(2,k,2)$.

\begin{thm}\label{classP4}
A quadro-quadric Cremona transformation of  $\mathbb P^4$ is linearly equivalent to the projectivization of one of the the sixteen (pairwise non linearly equivalent) Jordan involutions in the following table: 
\end{thm}

\begin{table}[H]
  \centering
  \begin{tabular}{|c|c|l|c|c|}\hline
  \label{Tab:toto}
   {\bf Algebra} &     ${\boldsymbol{\dim R}}$ &  \qquad \, {\bf Jordan involution} $\boldsymbol{(x,y,z,t,u)^\#}$     & {\bf  Type}& {\bf  Multidegree}     \\
   \hline   \hline
${\mathcal J}_{ 1}^5$   &     4& $ (   x^2 \, , \,   -xy   \, , \,   -xz    \, , \,  y^2-xt     \, , \,  2yz-xu )   $         &             $II$    &                       
$(2, 3  , 2)$  \\    \hline
${\mathcal J}_{ 2}^5$    &      4   &    $  (   x^2 \, , \,   -xy   \, , \,   -xz    \, , \,  y^2+z^2-xt     \, , \,  2yz-xu )  $          &             $III$             &    $(2, 4 , 2)$        \\   \hline
  ${\mathcal J}_{3}^5$ &         4 &$  (     x^2 \, , \,   -xy   \, , \,   -xz    \, , \,  -xt     \, , \,  y^2-xu    )  $          &             $I$ &               
  $(2, 2 , 2)$ \\   \hline
    ${\mathcal J}_{ 4}^5$ &      4 &   $   (  x^2 \, , \,   -xy   \, , \,   -xz    \, , \,  -xt     \, , \,   y^2+z^2-xu )   $         &             $I$     &     $(2, 2 , 2)$       \\   \hline
    ${\mathcal J}_{5}^5$ &      4  &    $  (  x^2 \, , \,   -xy   \, , \,   -xz    \, , \,  -xt     \, , \,  y^2+z^2+t^2-xu    )  $          &             $I$    &      $(2,2  , 2)$       \\    \hline \hline
    ${\mathcal J}_{6}^5$  &    3 &    $   (   xy \, ,\,  x^2 \, , \,  -yz\, , \, -xt\, , \, -xu   ) $     &       $I$   & $(2, 2 , 2)$ \\    \hline
      ${\mathcal J}_{7}^5$ & 3  &    $   (   xy \, ,\,  x^2 \, , \, t^2 -yz\, , \, -xt\, , \, -xu     ) $     &       $I$  & $(2, 2 , 2)$  \\    \hline  
       ${\mathcal J}_{8}^5$ &    3  &    $   (   xy \, ,\,  x^2 \, , \, t^2 +u^2-yz\, , \, -xt\, , \, -xu     ) $     &       $I$  & $(2,2  , 2)$  \\    \hline   
      ${\mathcal J}_{ 9}^5$&    3   &    $   (  xy \, ,\,  x^2 \, , \,  -yz\, , \, -xt\, , \, 2zt-xu     ) $     &       $II$  & $(2,3  , 2)$ \\    \hline  
     ${\mathcal J}_{ 10}^5$  &    3&    $   (   xy \, , \,   x^2   \, , \,   -yz    \, , \,  -yt     \, , \,  -xu     ) $     &       $II$   &  $(2, 3 , 2)$\\    \hline 
     ${\mathcal J}_{ 11}^5$  &    3&    $   ( xy \, , \,   x^2   \, , \,   u^2-yz    \, , \,  -yt     \, , \,  -xu     ) $     &       $III$    & $(2, 4 , 2)$ \\    \hline 
${\mathcal J}_{ 12}^5$ &    3 &    $   (    xy \, , \,   x^2   \, , \,   -yz    \, , \,  -yt     \, , \, -yu     ) $     &       $I$   & $(2, 2 , 2)$ \\    \hline  \hline
${\mathcal J}_{13}^5$  &    2 &    $   (    yz \, , \, xz \, , \, xy\, , \, -zt\, , \, -yu    ) $     &       $II$   &   $(2, 3 , 2)$\\    \hline 
${\mathcal J}_{14}^5$   &    2&    $   (    y^2+z^2 \, , \,   xy   \, , \,   -xz    \, , \,  -xt     \, , \, -xu    ) $     &       $I$   & $(2, 2 , 2)$  \\    \hline  \hline
${\mathcal J}_{15 }^5$ &    1  &    $   (       y^2+z^2+t^2 \, , \,   xy   \, , \,   -xz    \, , \,  -xt     \, , \, -xu     ) $     &       $ I   $     &   $(2, 2 , 2)$ \\    \hline  \hline
${\mathcal J}_{16}^5
   $   &    0 &    $   (    y^2+z^2+t^2+u^2 \, , \,   xy   \, , \,   -xz    \, , \,  -xt     \, , \, -xu    ) $     &       $I$  & $(2, 2 , 2)$ \\    \hline 
 \end{tabular}
\caption{Classification of quadro-quadric Cremona transformations of $\mathbb P^4$.}\label{Tab:jordanP4}
\end{table}

The classification above is  completely explicit but does not say much about the geometry of the corresponding Cremona transformations. In the next subsections, for each Jordan involution $f_{\mathcal J_k^5}$ in this table, we describe as geometrically   (and so little schematically) as possible  its base locus scheme $\mathcal B(\mathcal J^5_k)$ as well as  the   linear system  of quadric hypersurfaces 
$$\big\lvert  \mathcal I_{B(\mathcal J^5_k)}(2) \big\lvert=f_{\mathcal J_k^5}^{-1}\big\lvert  \mathcal O_{\mathbb P^4}(1) \big\lvert.$$

\subsubsection{Tables  of quadro-quadric Cremona transformations of $\mathbb P^4$}
We will deal with the three types $I$, $II$ and $III$ separately. The results are collected  in some tables below according to the  type in a  similar way to the classification tables of Cremona transformations of type $(2,d)$ of $\mathbb P^3$ obtained in  \cite{PRV}. 
When  possible, we also offer graphic representations of the associated  base locus schemes. These drawings could be helpful to understand better the various cases geometrically, especially the most degenerated ones.
\medskip 

Some  Cremona maps of type II or III are rather complicated since their  base locus schemes  have a non-reduced irreducible component whose schematic structure is more subtle than for type I.
Before presenting the aforementioned tables of Cremona transformations, we introduce some terminology to deal with  some non-reduced schemes of dimension 1.
\medskip

  \paragraph{$-$ {\it Double structures on a line  in $\mathbb P^3$.}}  
  Let $\mathscr L$ be a double structure supported on a line 
  $\ell=\mathscr L_{\rm red}\subset \mathbb P^3$ (see \cite[\S 1]{PanRusso}): the associated ideal sheaf $\mathcal I_{\mathscr L}$ verifies $(\mathcal I_\ell)^2\subset 
\mathcal I_{\mathscr L} \subset \mathcal I_\ell$ and is such that the sheaf $\mathcal J_{\mathscr L}=
\mathcal I_{\ell}/\mathcal I_{\mathscr L}$ is invertible.  Thus there is a surjection $\mathcal I_\ell/(\mathcal I_\ell)^2\rightarrow \mathcal J_{\mathscr L}$ to which it corresponds a unique section $\sigma_{\mathscr L}\in H^0(\ell,N_{\ell/\mathbb P^3})$.  It defines a line bundle on $\ell\simeq \mathbb P^1$ whose degree $\delta(\mathscr L)\geq 0$ is projectively attached to $ \mathscr L\subset \mathbb P^3$.  In fact, it is the unique invariant of such a scheme since setting $\delta=\delta(\mathscr L)$  for simplicity,  one knows  (cf. \cite[page 136]{EisenbudHarris}  for instance) that    $\mathscr L$ is projectively  isomorphic to
$$ \mathscr L_{\delta}={\rm Proj}\bigg( \frac{\mathbb C[a,b,d,c]}{(a^2,ab,b^2,ad^\delta-bc^\delta )}\bigg) \subset \mathbb P^3. $$

One can describe geometrically $ \mathscr L_{\delta}$ as the union of the  rational family of punctual schemes $\{\tau_x\}_{x\in \ell}$, all isomorphic to $\tau\subset \mathbb P^3$ and such that the projectivized tangent space   $T_x \tau_x$ is {\it normal} to $\ell$ for  every $x$.  As mentioned  in \cite{EisenbudHarris}, $\mathscr L$ can be thought as $\ell$ `plus a projective line $\mathcal L$ infinitely close to it', the integer $\delta$ specifying how fast $\mathcal L$ `twists along $\ell$' in $\mathbb P^3$. When $\delta=0$,  there is no twisting  at all so that, using the terminology introduced in Section \ref{Notation}, $\mathscr L_0$ is nothing but a 2-multiple line in
a 2-plane, {\it i.e.} 
$\mathscr L_0\simeq {\rm Proj}(\mathbb C[a,b,c]/(a^2))\subset \mathbb P^2$.  This contrasts with the general case  since  $\langle \mathscr L_\delta\rangle=\mathbb P^3  $ 
for every positive integer  $\delta$.\mk 

In what follows, we will consider only the schemes $\mathscr L_\delta$ for 
$\delta=0,1$, linearly embedded in $\mathbb P^4$. Based on the above discussion,  it is reasonable to  represent these two schemes as: a line plus  another one glued along the first, in a rectilinear way for $\delta=0$; twisted one time when $\delta=1$: 

  \begin{figure}[H]
\begin{center}
\begin{tabular}{ccc}
\begin{tabular}{c}
\psfrag{L0}[][][1]{$ \;\;\,    \scriptstyle{\mathscr L_0}  $}
\psfrag{<L0>}[][][0.8][39]{$ \quad   \scriptstyle{\langle \mathscr L_0 \rangle}  $}
 \includegraphics[width=6cm,height=5cm]{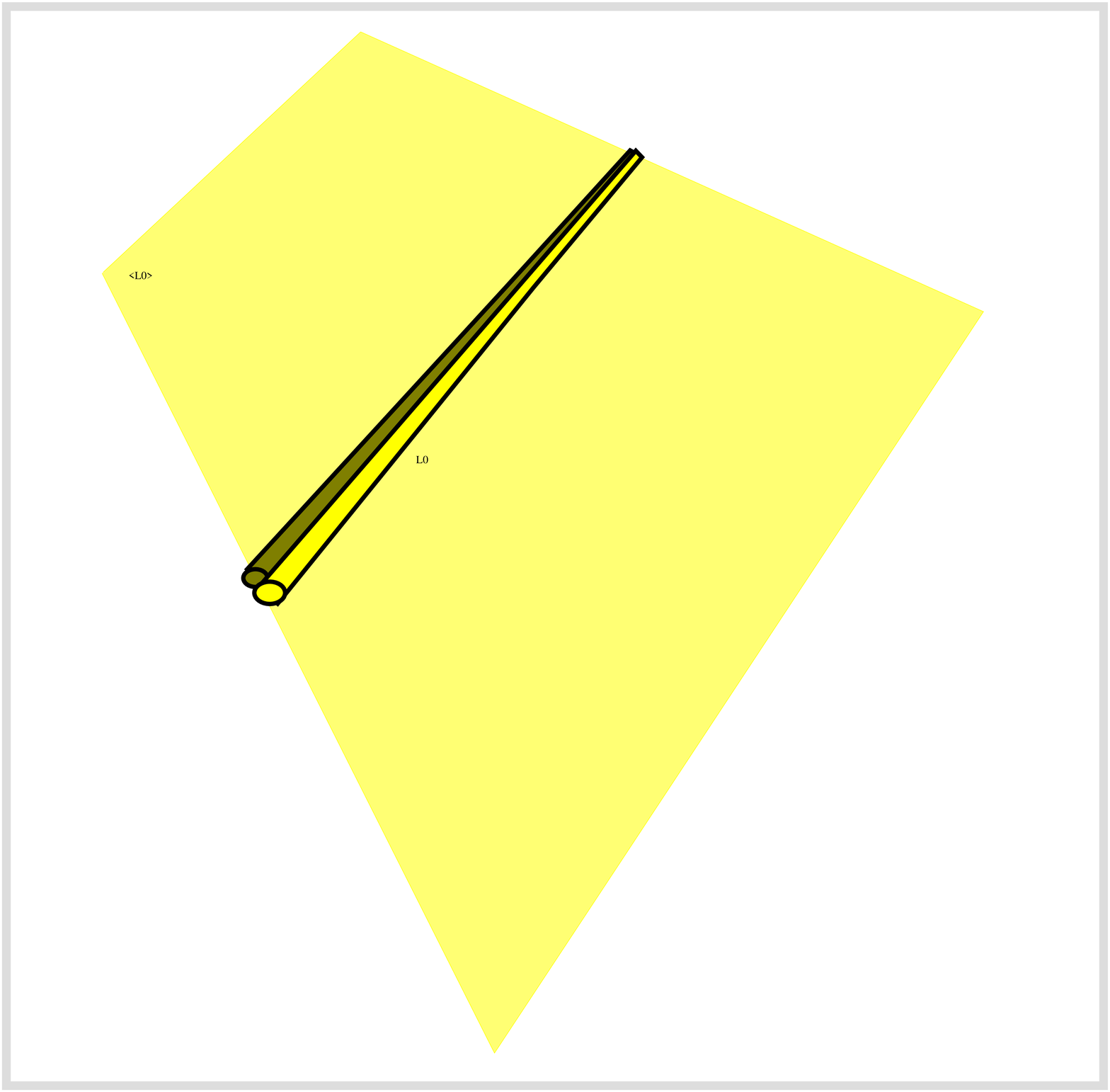} \\
 Pictural representation of  $\mathscr L_0 \subset \mathbb P^4$  \\ (the 
    2-plane $\langle \mathscr L_0 \rangle$ is pictured in yellow)
 \end{tabular}
 &  &
 \begin{tabular}{c}
\psfrag{L0}[][][1]{$ \;\;\,    \scriptstyle{\mathscr L_1}  $}
\psfrag{<L0>}[][][0.8][39]{$ \quad \;    \scriptstyle{\langle \mathscr L_1 \rangle}  $}
 \includegraphics[width=6cm,height=5cm]{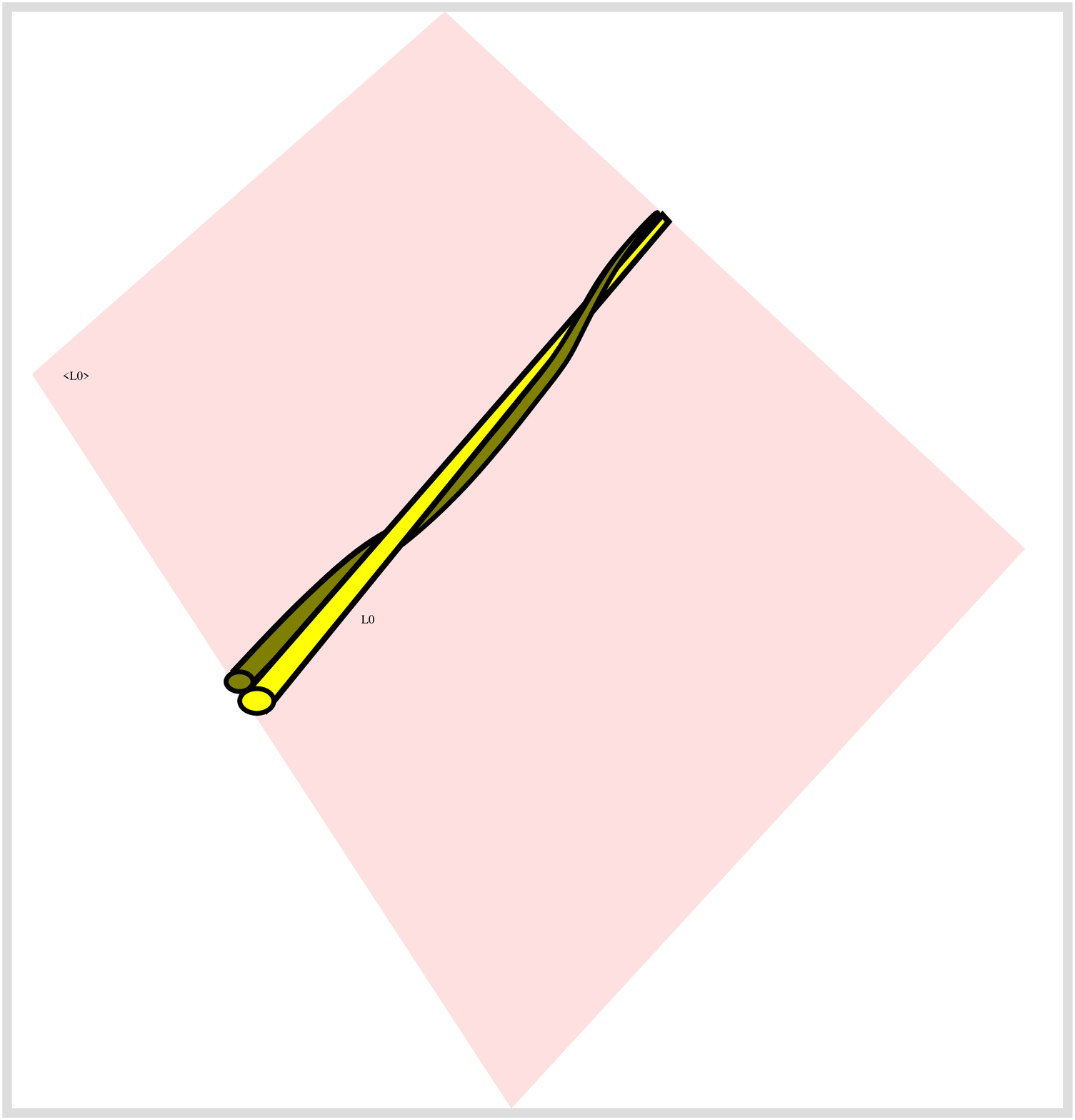}\\
 Pictural representation of  $\mathscr L_1\subset \mathbb P^4$  \\  (the   
 hyperplane  $\langle \mathscr L_1 \rangle $ is pictured in 
   pink)
 \end{tabular}
\end{tabular}
\end{center}
\end{figure}

\paragraph{\it $-$ The non-reduced degenerate  rational quartic curve $\mathscr L^4 \subset \mathbb P^4$.}

Let 
$$
 \mathscr L^4={\rm Proj}\Big(  \mathbb C[x,y,z,t,u]
/  (  x^2 \, , \,   xy   \, , \,   xz    \, , \,  y^2-xt     \, , \,  2yz-xu, z^2)
 \Big)\subset \mathbb P^4.
$$

The scheme $\mathscr L^4$ is irreducible and has dimension 1.  It is an arithmetically Cohen-Macaulay (briefly ACM) non reduced curve of degree 4 having arithmetic genus 0 and  supported on the line $\ell=\mathscr L^4_{\rm red}=V(x,y,z)\subset \p^4$.  For any point $p\in \ell$, the projectivized tangent space  $T_p \mathscr L^4$ of $\mathscr L^4$ at $p$ is the coordinate hyperplane $H=V(x)$. The (schematic) intersection $\mathscr L^4\cap H$ is the scheme defined by the homogeneous ideal  $(x,y^2,yz,z^2)$: it is a double-line in $H$.  The residual intersection of $H$ with $\mathscr L^4$ is the reduced line $\ell$. \sk

A natural question is to know if $\mathcal L^4$ is a degeneration of  a  rational normal quartic curve $C_4=v_4(\p^1)\subset \mathbb P^4$, which has the same Hilbert polynomial.
This should follow from \cite[Proposition 2.2]{MD-P} but  the algorithm presented there for producing  a  smoothing of  the  scheme  $\mathscr L^4$ does not work because in this case 
 a non flat family appears\footnote{Actually, the deformation $\mathcal C$ given in part 2) of the proof of  \cite[Proposition 2.2]{MD-P}  is even not equidimensional in general.}.   Thus as far as we know, till now it  is unknown if $\mathcal L^4$ belongs to the irreducible component of the Hilbert scheme containing $C_4$. This shows how subtle the analysis of the most degenerated examples can be.
\bk 

In the next pages, the reader will find the aforementioned  tables of quadro-quadric Cremona transformations of $\p^4$ as well as the corresponding graphic representations of the associated  base locus schemes. \sk 

We begin by considering the quadro-quadric Cremona transformations of type I. The descriptions of the linear systems in 
T{\small{ABLE}} 7 is complete and very similar to the ones  of  `Tableau 1' of \cite{PRV}\footnote{Note however
 that  the description  of the linear system of type ${\rm tan}^{[3]}(/\!/)$ is not correct in \cite{PRV}.}.
\mk 

Then we consider the case of quadro-quadric Cremona transformations of type II.  

All the associated linear systems of quadrics are described in T{\small{ABLE}} 8, with the exception of $\lvert  \mathcal I_{\mathcal B(\mathcal J^5_1)} \lvert$ whose description is too long to be put in the table. In order to describe it geometrically,  let's introduce the following terminology: if $C$ is a curve included in the smooth locus of a surface $S \subset \mathbb P^4$, we will say that a hypersurface $\mathcal H\subset \p^4$ {\it osculates normally $S$ along $C$} if any generic  subscheme $\eta\subset S$  normal to $C$ is also a subscheme of $\mathcal H$.  More concretely, this means that for 
$c\in C$ generic and any regular germ of curve $\mu: (\mathbb C,0)\rightarrow (S,c)$  whose image is transverse to $C$, the  composition  $h\circ \mu$ has valuation  at least 2 at 0 if $h$ is a generator of $\mathcal O_{\mathcal H,c}$. \sk 

Now let $\ell$ be  a fixed line   of the  ruling of a rational normal scroll $S= S_{1,2}$ in $\mathbb P^4$ and  denote by $\pi$ the 2-plane  spanned by $\ell$ and the directrix  line of $S$.  Then up to projective equivalence, 
 \begin{equation}
\label{E:LinearSystemJ1}
 \begin{tabular}{l}
 {\it the homaloidal linear system  $\lvert  \mathcal I_{\mathcal B(\mathcal J^5_1)} \lvert$ is formed by quadric hypersurfaces   in $\p^4$ } \\ {\it     containing the 2-plane $\pi$ and osculating  normally the rational normal scroll $S$ along  $\ell$.}
\end{tabular} \end{equation}
 
 This claim can be verified by easy direct computations\footnote{To recover exactly the linear system $\langle x^2, xy , xz , y^2-xt , 2yz-xu, z^2 \rangle$ 
 appearing in  
 Table 8, one has to take  $S=V(y^2-xt, 2zt-yu,2yz-xu)$ and $\ell=V(x,y,z)$.}. 
\smallskip 

The base locus schemes of type II are pictured in F{\small{IGURE}} 6, at the exception of $\mathcal B(\mathcal J_1^5)$: we have not been able to figure a way to represent this scheme. 
\newpage 

\begin{landscape}
\begin{table}
\begin{center}
\scalebox{0.90}{
  \begin{tabular}{|c|c|c|c|c|c|}\hline
  {\bf Algebra}  $\boldsymbol{\mathcal J}    $ & 
  {\bf  Jordan adjoint}  $\boldsymbol{(x,y,z,t,u)^\#}$
     &  {\bf Primary decomposition of} $ {\boldsymbol{ \mathcal I_{
     {\mathcal B(\mathcal J)}
 }}}$     &  {\bf Geometrical description of}  $ {\boldsymbol{ {\mathcal B(\mathcal J)}}}$    & {\bf  Linear system}  $ 
      {\boldsymbol{ |  \mathcal I_{\mathcal B(\mathcal J)}(2) | }}  $     \\
   \hline  \hline
        $\mathcal J_{3}^5$ & $   (   x^2 \, , \,   -xy   \, , \,   -xz    \, , \,  -xt     \, , \,  y^2-xu     ) $     
      & $ (x,y^2)  \cap    (t,z,xy,x^2,y^2-xu)$ &
 \begin{tabular}{l}   \vspace{-0.3cm}\\
 A rank 1 quadric surface  $S$ plus 
 a  \\ scheme  $\eta$  
 supported   on   $S_{\rm red}$ 
 such \\  that
 $\langle \eta \rangle \cap \langle S\rangle$ is the line tangent to $\eta$\\   that is 
 moreover  transverse to   $S_{\rm red}$  \vspace{0.15cm}
 \end{tabular}
          &     \begin{tabular}{l} 
          Rank 3 quadric hypercones   contain- \\ -ing $S_{\rm red}$ with  vertex   line  included  \\ in $S_{\rm red}$  
        and   osculating  along $\eta$    
     \end{tabular}     
      \\  \hline 
           $\mathcal J_{4}^5$ & $   ( x^2 \, , \,   -xy   \, , \,   -xz    \, , \,  -xt     \, , \,   y^2+z^2-xu     ) $     
      &  
      \begin{tabular}{c}  
      $(x,y-i\,z)  \cap      (x,y+i\, z)$ \vspace{0.15cm}\\ 
     $ \cap    \,  (t,z,xy,x^2,y^2-xu)$   
       \end{tabular} 
      &
 \begin{tabular}{l}  \vspace{-0.3cm}\\
 A rank 2 quadric surface $S$ plus a \\
   scheme $\eta$  
 supported   on the line    $S_{\rm sing}$  \\ 
 such that 
 $\langle \eta \rangle \cap \langle S\rangle$ is the line tangent \\  to $\eta$  that is 
  not contained in  $S$  \vspace{0.15cm}
 \end{tabular}
          &     \begin{tabular}{l}  Hyperquadrics  
        containing  \\ $S$  and osculating  along $\eta$
     \end{tabular}    
          \\  \hline 
      $\mathcal J_{5}^5$ & $   (    x^2 \, , \,   -xy   \, , \,   -xz    \, , \,  -xt     \, , \,  y^2+z^2+t^2-xu  ) $     
      &
      \begin{tabular}{l}
       $(x,y^2+z^2+t^2)\,  \cap  $ \vspace{0.15cm}\\ $      (t,z,xy,x^2,y^2-xu)$   
       \end{tabular}&
 \begin{tabular}{l} 
  \vspace{-0.3cm}\\
 A rank 3   quadric surface $S$ plus 
 a \\ scheme $\eta$  
 supported   at  the vertex   \\ of   $S$ 
 such that 
 $\langle \eta \rangle \cap \langle S\rangle$  is the line  \\Êtangent to $\eta$  that is 
  not included in   $S$ \vspace{0.15cm}
 \end{tabular} 
          &     \begin{tabular}{l}  Hyperquadrics  
        containing  \\ $S$   and osculating  along $\eta$ 
         \end{tabular}          \\  \hline 
   $\mathcal J_{6}^5$ & $   (     xy \, ,\,  x^2 \, , \,  -yz\, , \, -xt\, , \, -xu   ) $     
   &  $(x,y)  \cap (x,z)  \cap      (y,t,u,x^2)$
     &
 \begin{tabular}{l} 
   \vspace{-0.3cm}\\
 A rank 2  quadric surface $S$ 
 plus a \\   scheme  $\tau$  supported  at a point 
  $p\in S_{\rm reg}$ \\ and 
 spanning a line transverse to $\langle S\rangle $ \vspace{0.15cm}
 \end{tabular} 
          &     \begin{tabular}{l} Hyperquadrics 
     containing   $S$  \\ and tangent 
     to $\langle T_p S,\tau \rangle$ at $p$
     \end{tabular}  
             \\  \hline   
$\mathcal J_{7}^5$ &     $   (   xy \, ,\,  x^2 \, , \, t^2 -yz\, , \, -xt\, , \, -xu     ) $     
&  $(x,t^2-yz)  \cap     (x^2,y,u,xt,t^2)$  &
 \begin{tabular}{l} 
   \vspace{-0.3cm}\\
 A rank 3  quadric surface $S$ 
 plus a \\  scheme   $\xi$
 supported at  a point  $p\in  S_{\rm reg}$ \\  
 such that  
 $\langle  \xi \rangle$  cuts 
 $\langle S\rangle$    along a line   tan- \\ -gent  to    $S$ at $p$ but 
 not included in $S$  \vspace{0.15cm}
  \end{tabular}
      & 
   \begin{tabular}{l} Hyperquadrics 
     containing   $S$ \\   and tangent 
     to $\langle  T_p S,  \xi \rangle$ at $p$ 
     \end{tabular} 
       \\    \hline 
$\mathcal J_{8}^5$  &     $   (   xy \, ,\,  x^2 \, , \, t^2 +u^2-yz\, , \, -xt\, , \, -xu      ) $     
& $(x,t^2+u^2-yz)  \cap  (x^2,y,u,xt,t^2)$ &
 \begin{tabular}{l} 
   \vspace{-0.3cm}\\
A  smooth  quadric surface $S$  
 plus  a   \\  scheme  $\xi$  
 supported    at a point $p\in    S$ \\ such  that  
 $\langle  \xi \rangle$  intersects 
 $\langle S\rangle$    along a line   \\ tangent    to  $S$ at $p$ but 
 not included in $S$  \vspace{0.15cm}
 \end{tabular}          &   
 \begin{tabular}{l} Hyperquadrics 
     containing   $S$ \\  and  tangent 
     to $\langle  T_p S,  \xi \rangle$ at $p$ 
     \end{tabular} 
      \\    \hline 
$\mathcal J_{12}^5$ &$   (    xy \, , \,   x^2   \, , \,   -yz    \, , \,  -yt     \, , \, -yu     ) $     
&  $ ( x^2,y)  \cap ( x,z,t,u   )  $ &
 \begin{tabular}{l} A  double plane  $S$  in  a   $\mathbb P^3$ \\
 plus a point $p$ outside $\langle S\rangle$\end{tabular} 
          &     \begin{tabular}{l} 
            \vspace{-0.3cm}\\
          Hyperquadrics tangent   to
             $\langle S\rangle $     \\   along $S_{red}$  and 
           passing through    $ p $  \vspace{0.15cm}
     \end{tabular}
      \\  \hline   
$\mathcal J_{14}^5$  &    $   (    y^2+z^2 \, , \,   xy   \, , \,   -xz    \, , \,  -xt     \, , \, -xu    ) $     
& $ ( x,y^2+z^2 )  \cap ( y,z,t,u   ) $   &
\begin{tabular}{l} 
  \vspace{-0.3cm}\\
A  rank 2 quadric  surface  $S$  \\
 plus a point $p$ outside $\langle S\rangle$
  \vspace{0.15cm}
 \end{tabular}          &   
 \begin{tabular}{l}
   \vspace{-0.3cm}\\
    Hyperquadrics in $\p^4$\\ 
     containing   $S\cup \{p \}$ \vspace{0.15cm}
     \end{tabular}  
      \\    \hline 
$\mathcal J_{15 }^5$&       $   (       y^2+z^2+t^2 \, , \,   xy   \, , \,   -xz    \, , \,  -xt     \, , \, -xu     ) $     
&  $ ( x,y^2+z^2+t^2 )  \cap ( y,z,t,u   ) $    &
\begin{tabular}{l}   \vspace{-0.3cm}\\
A  quadric cone  $S$ in $ \p^3$\\
 plus a point $p$ outside $\langle S\rangle$
 \vspace{0.15cm}
 \end{tabular}          &      
 \begin{tabular}{l} 
   \vspace{-0.3cm}\\
   Hyperquadrics in $\p^4$\\ 
     containing   $S\cup \{p \}$\vspace{0.15cm}
     \end{tabular} 
          \\    \hline 
$\mathcal J_{16}^5$ &      $   (    y^2+z^2+t^2+u^2 \, , \,   xy   \, , \,   -xz    \, , \,  -xt     \, , \, -xu    ) $     
&  $ ( x,y^2+z^2+t^2+u^2 )  \cap ( y,z,t,u   ) $  &
 \begin{tabular}{l} 
   \vspace{-0.3cm}\\
 A smooth quadric surface $S$\\
 plus a point $p$ outside $\langle S\rangle$\vspace{0.15cm}
     \end{tabular}        &   \begin{tabular}{l} 
       \vspace{-0.3cm}\\
     Hyperquadrics in $\p^4$\\ 
     containing   $S\cup \{p \}$ \vspace{0.15cm}
     \end{tabular}     \\     \hline 
         \end{tabular}}
\end{center}\bigskip
\caption{Classification of elementary quadro-quadric Cremona transformations of $\mathbb P^4$ (Type I).}
\end{table}
\end{landscape}

\begin{figure}[H]
\begin{center}
\begin{tabular}{ccc}
\begin{minipage}{4.0cm}
\begin{tabular}{c}
\psfrag{eta}[][][1]{$\;\;\;    \scriptstyle{\boldsymbol{\eta}} \; \; $}
\psfrag{S}[][][1]{$ \; \, \scriptstyle{\boldsymbol{S}} $}
\psfrag{<S>}[][][1]{}
\includegraphics[width=4.4cm,height=4.0cm]{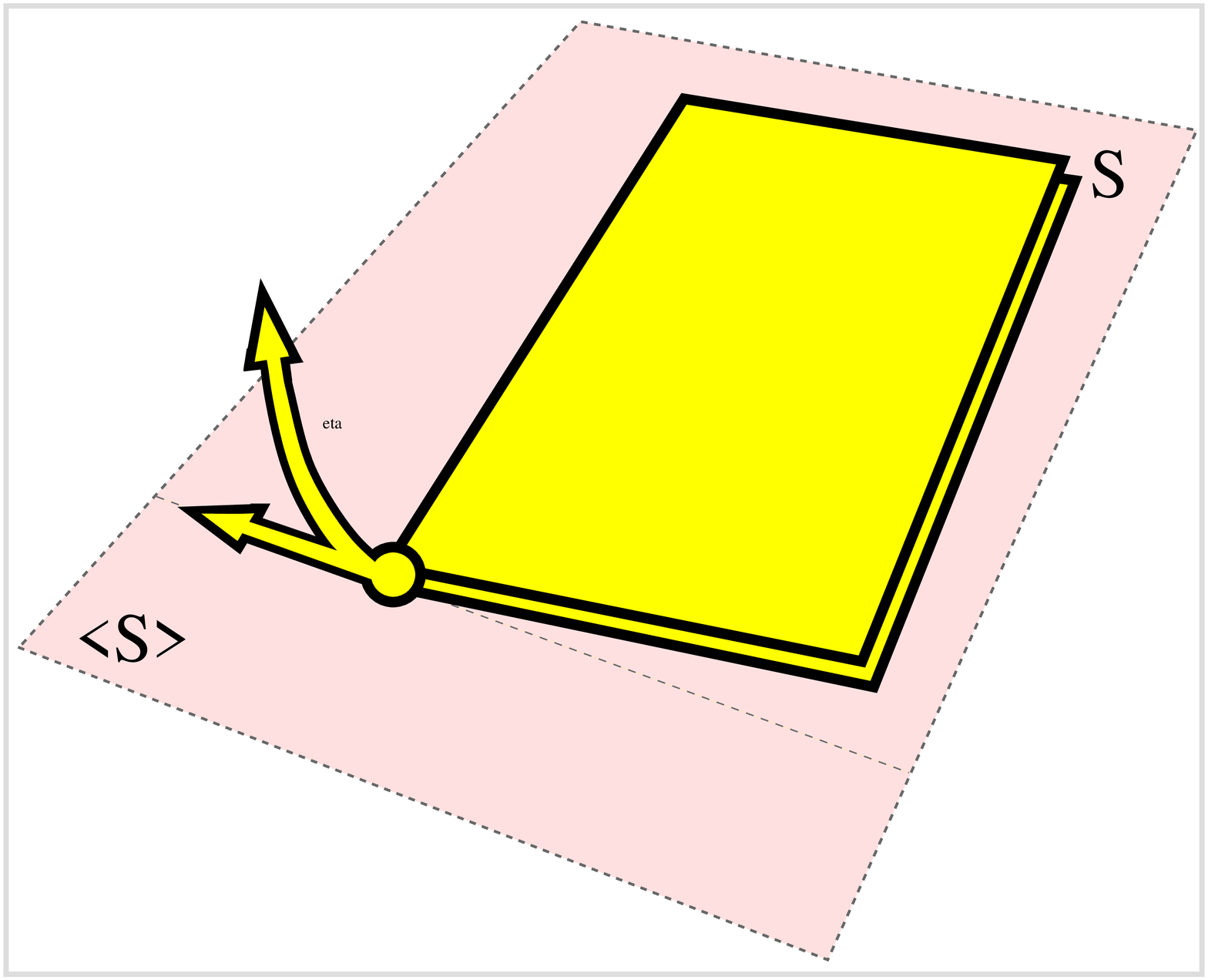} \\
\small{Base locus scheme $\mathcal B(\mathcal J_{3}^5)$}
\end{tabular}
\end{minipage} \quad 
&
 \quad 
 \begin{minipage}{4.0cm}
\begin{tabular}{c}
\psfrag{eta}[][][1]{$\,  \scriptstyle{\boldsymbol{\eta}} \; \; $}
\psfrag{S}[][][1]{$\, \,  \scriptstyle{\boldsymbol{S}} $}
\psfrag{<S>}[][][1]{}
\includegraphics[width=4.4cm,height=4.0cm]{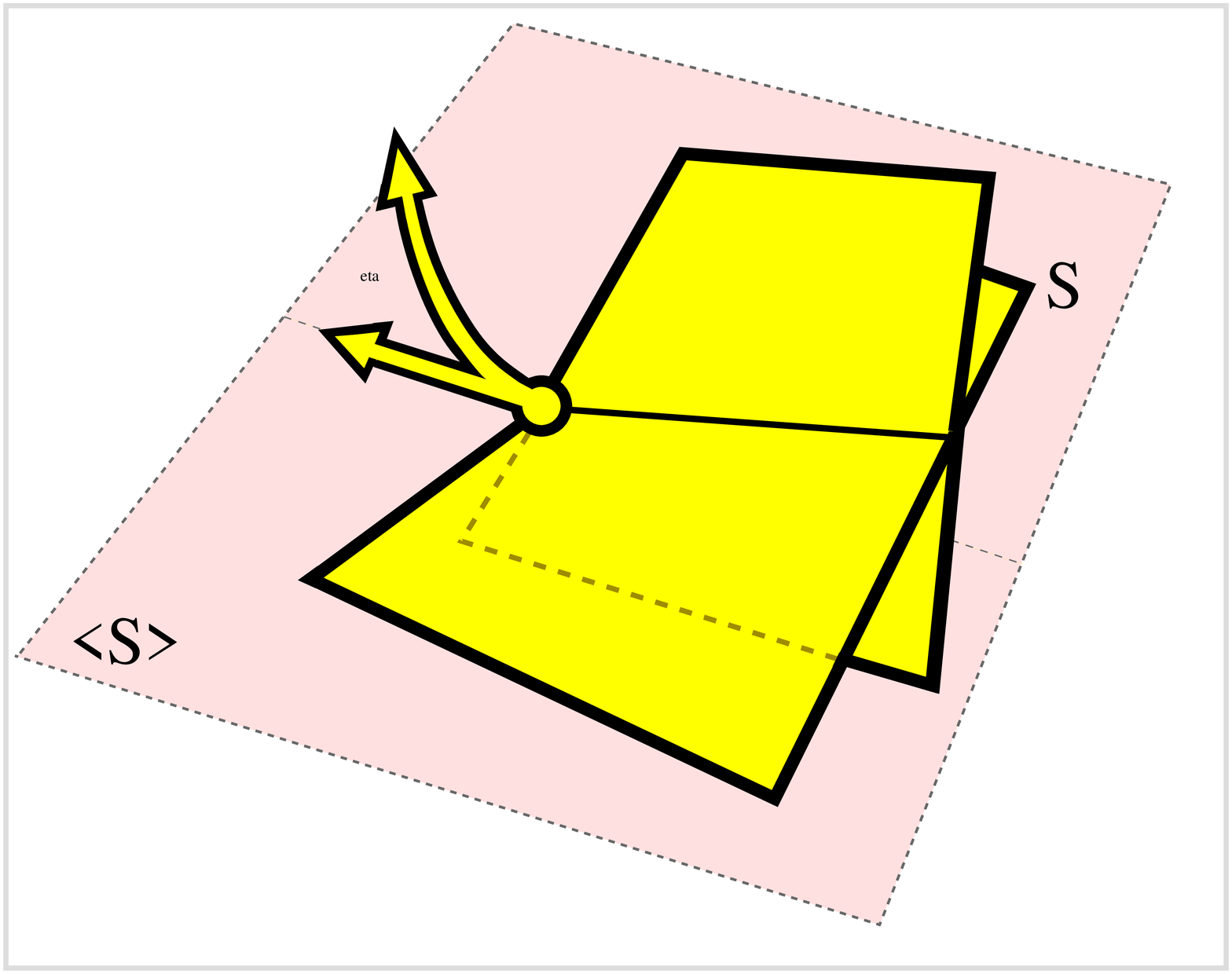} \\
\small{Base locus scheme $\mathcal B(\mathcal J_{4}^5)$}
\end{tabular}
\end{minipage}
\quad 
&  
\quad 
\begin{minipage}{4.0cm}
\begin{tabular}{c}
\psfrag{eta}[][][1]{$  $}
\psfrag{eta2}[][][1]{$\,  \scriptstyle{\boldsymbol{\eta}} \; \; $}
\psfrag{S}[][][1]{$\;  \scriptstyle{\boldsymbol{S}} $}
\psfrag{P}[][][1]{$\;\;\,   \scriptstyle{p} $}
\psfrag{<S>}[][][1]{}
\includegraphics[width=4.4cm,height=4.0cm]{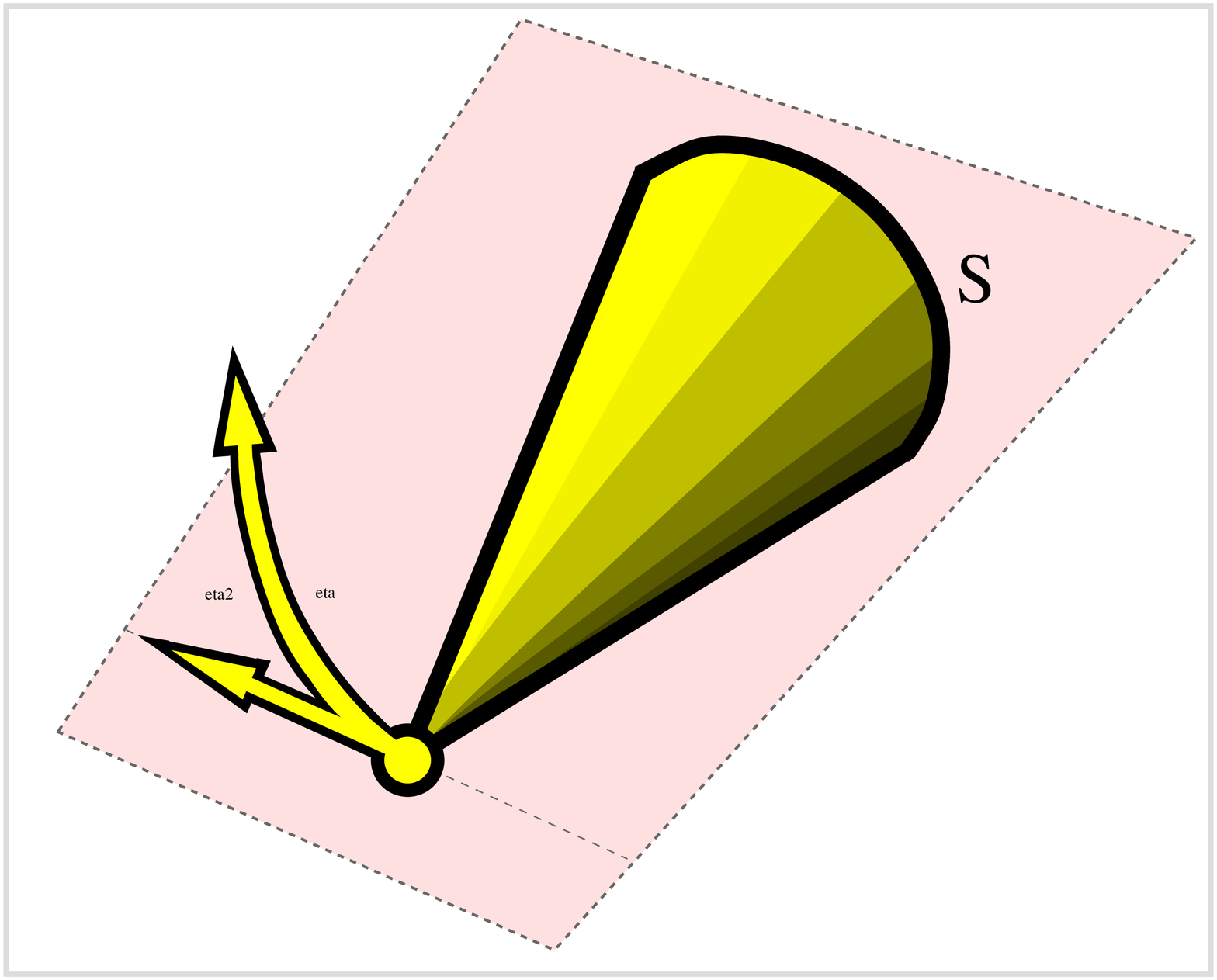} \\
\small{Base locus scheme $\mathcal B(\mathcal J_{5}^5)$}
\end{tabular}
\end{minipage}
\vspace{0.5cm}\\
\begin{minipage}{4.0cm}
\begin{tabular}{c}
\psfrag{tau}[][][1]{$ \scriptstyle{\boldsymbol{\tau}} \; $}
\psfrag{S}[][][1]{$ \; \scriptstyle{\boldsymbol{S}} $}
\psfrag{<S>}[][][1]{}
\includegraphics[width=4.4cm,height=4.0cm]{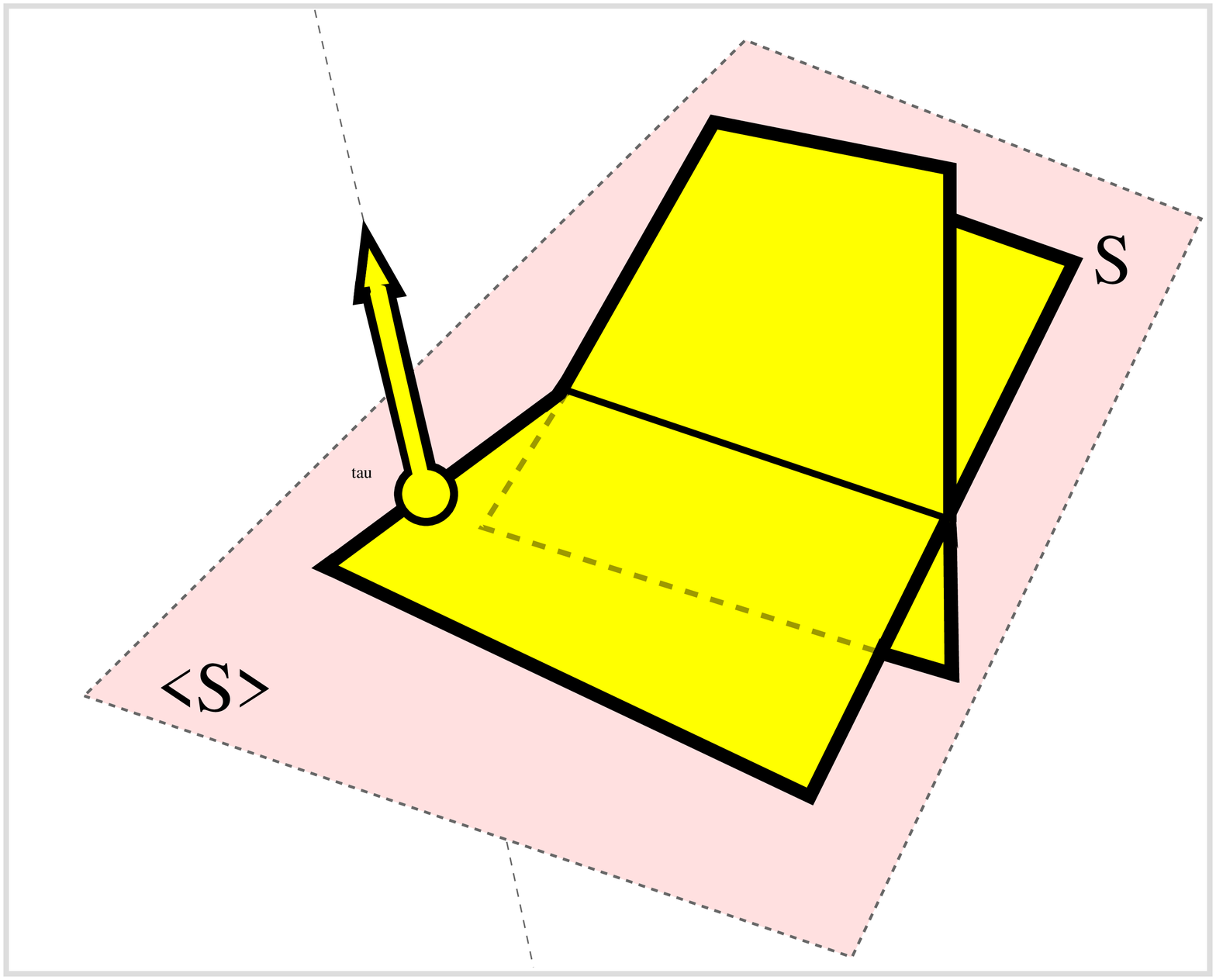} \\
\small{Base locus scheme $\mathcal B(\mathcal J_{6}^5)$}
\end{tabular}
\end{minipage} \quad 
&
 \quad 
 \begin{minipage}{4.0cm}
\begin{tabular}{c}
\psfrag{xi}[][][1]{$\,  \scriptstyle
{\boldsymbol{\xi}} \; \; $}
\psfrag{S}[][][1]{$\,  \scriptstyle{\boldsymbol{S}} $}
\psfrag{P}[][][1]{$\;\;\,   \scriptstyle{p} $}
\psfrag{<S>}[][][1]{}
\includegraphics[width=4.4cm,height=4.0cm]{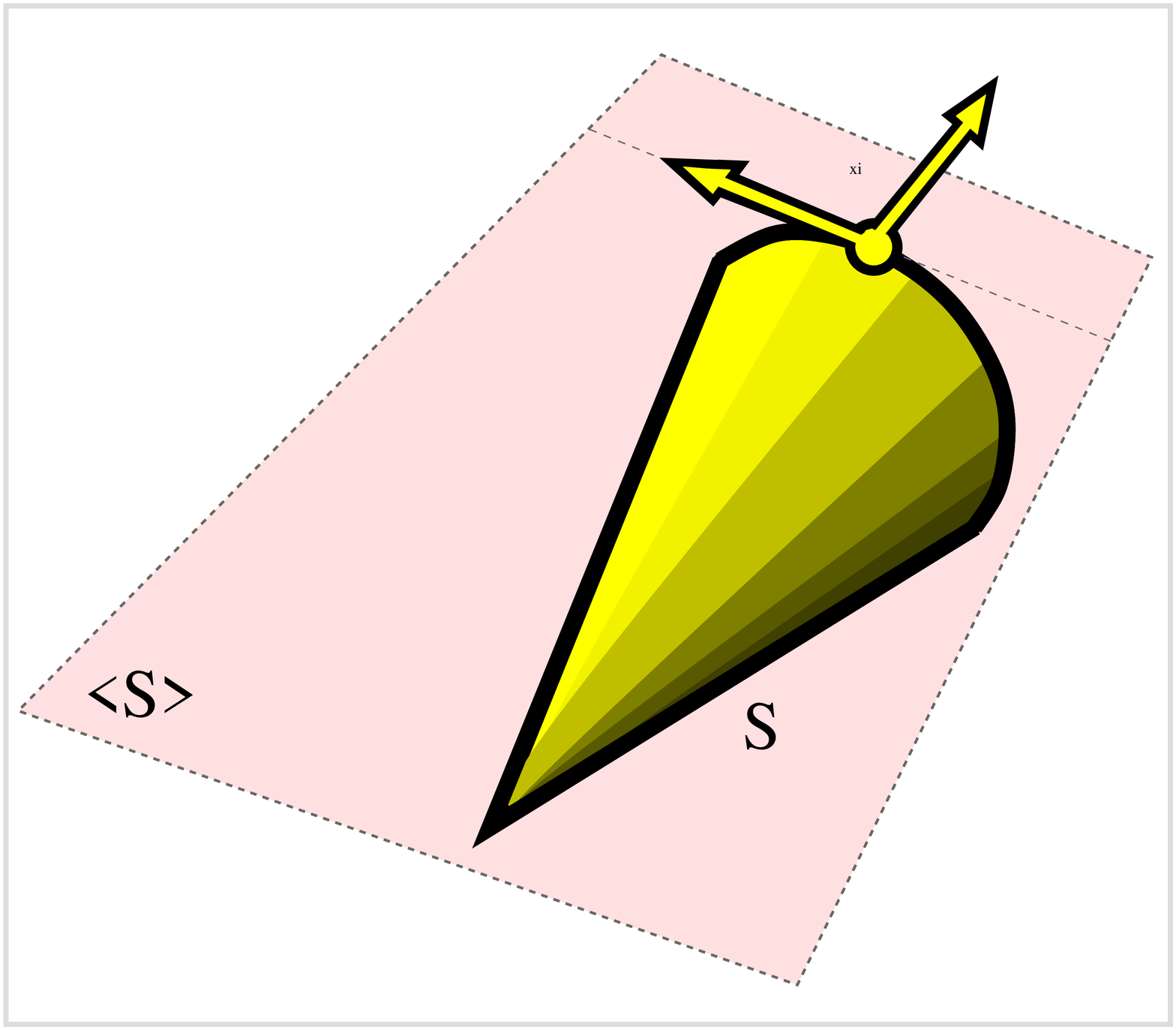} \\
\small{Base locus scheme $\mathcal B(\mathcal J_{7}^5)$}
\end{tabular}
\end{minipage}
\quad 
&  
\quad 
\begin{minipage}{4.0cm}
\begin{tabular}{c}
\psfrag{xi}[][][1]{$\,  \scriptstyle{\boldsymbol{\xi}} \; \; $}
\psfrag{S}[][][1]{$\,  \scriptstyle{\boldsymbol{S}} $}
\psfrag{P}[][][1]{$\;\;\,   \scriptstyle{\boldsymbol{p}} $}
\psfrag{<S>}[][][1]{}
\includegraphics[width=4.4cm,height=4.0cm]{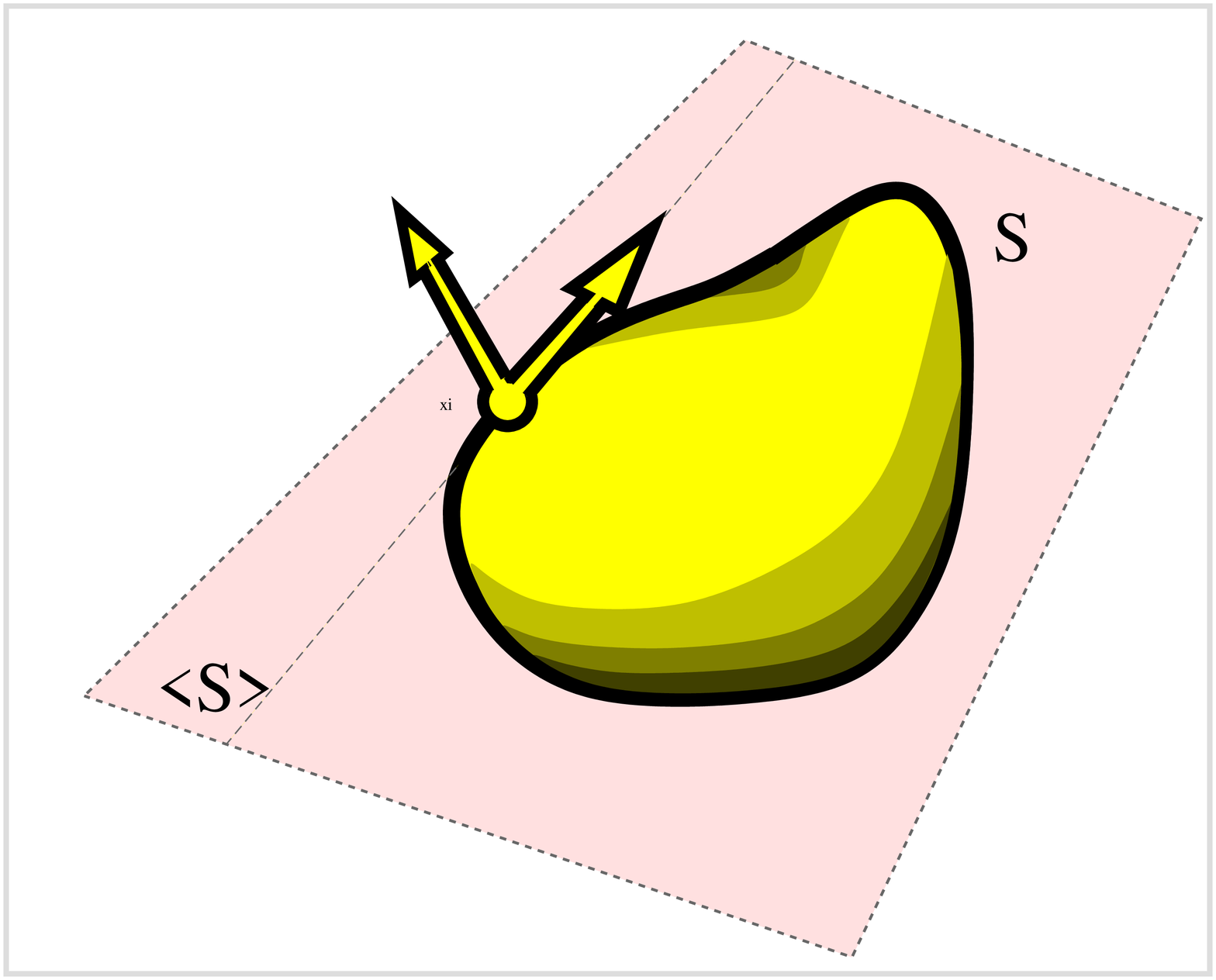} \\
\small{Base locus scheme $\mathcal B(\mathcal J_{8}^5)$}
\end{tabular}
\end{minipage}
\vspace{0.5cm}\\
\begin{minipage}{4.0cm}
\begin{tabular}{c}
\psfrag{S}[][][1]{$\,  \scriptstyle{\boldsymbol{S}} $}
\psfrag{P}[][][1]{$\;\;  \scriptstyle{\boldsymbol{p}} $}
\psfrag{<S>}[][][1]{}
\includegraphics[width=4.4cm,height=4.0cm]{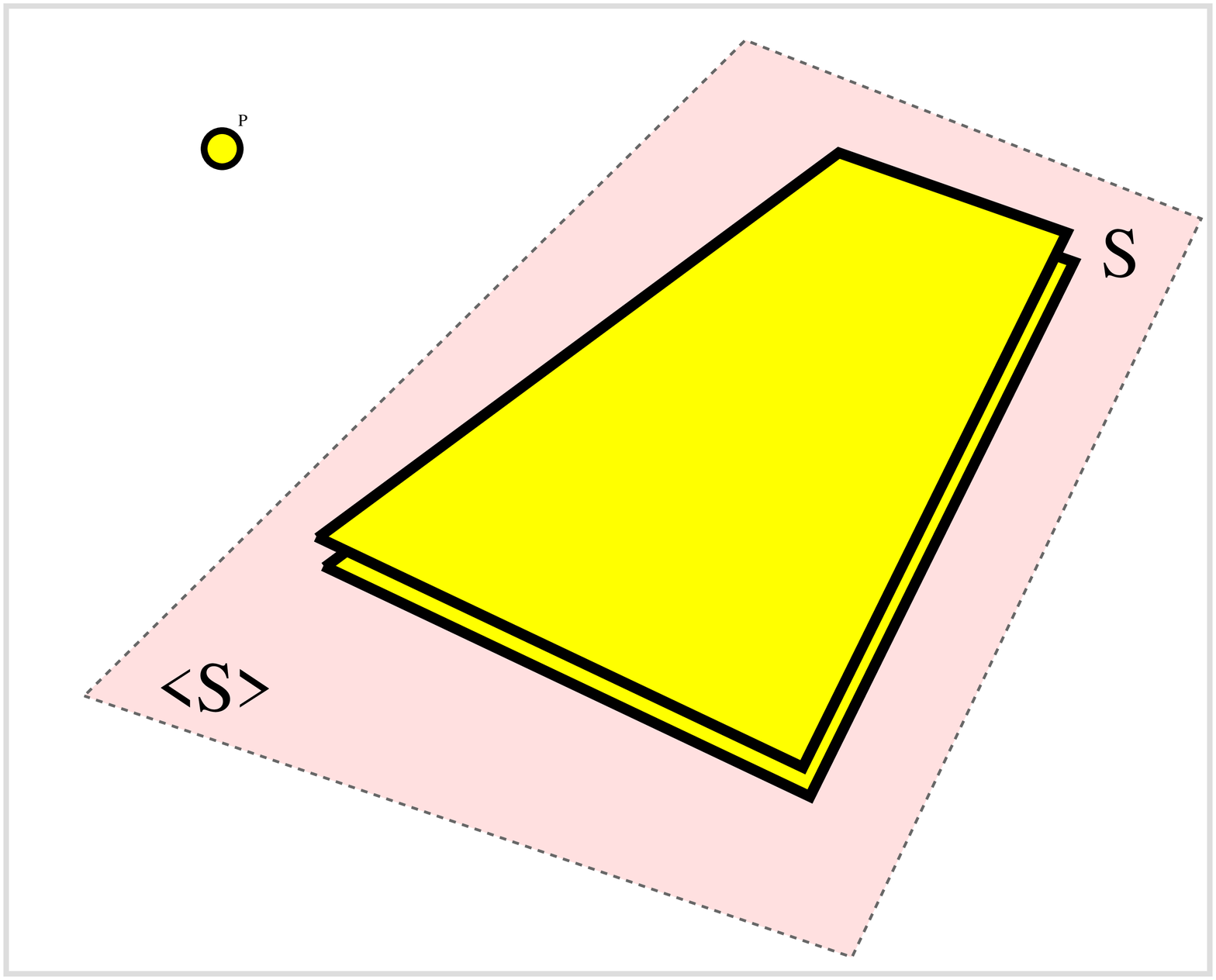} \\
\small{Base locus scheme $\mathcal B(\mathcal J_{12}^5)$}
\end{tabular}
\end{minipage} \quad 
&
 \quad 
 \begin{minipage}{4.0cm}
\begin{tabular}{c}
\psfrag{S}[][][1]{$\,  \scriptstyle{\boldsymbol{S}} $}
\psfrag{P}[][][1]{$\;\;\,   \scriptstyle{\boldsymbol{p}} $}
\psfrag{<S>}[][][1]{}
\includegraphics[width=4.4cm,height=4.0cm]{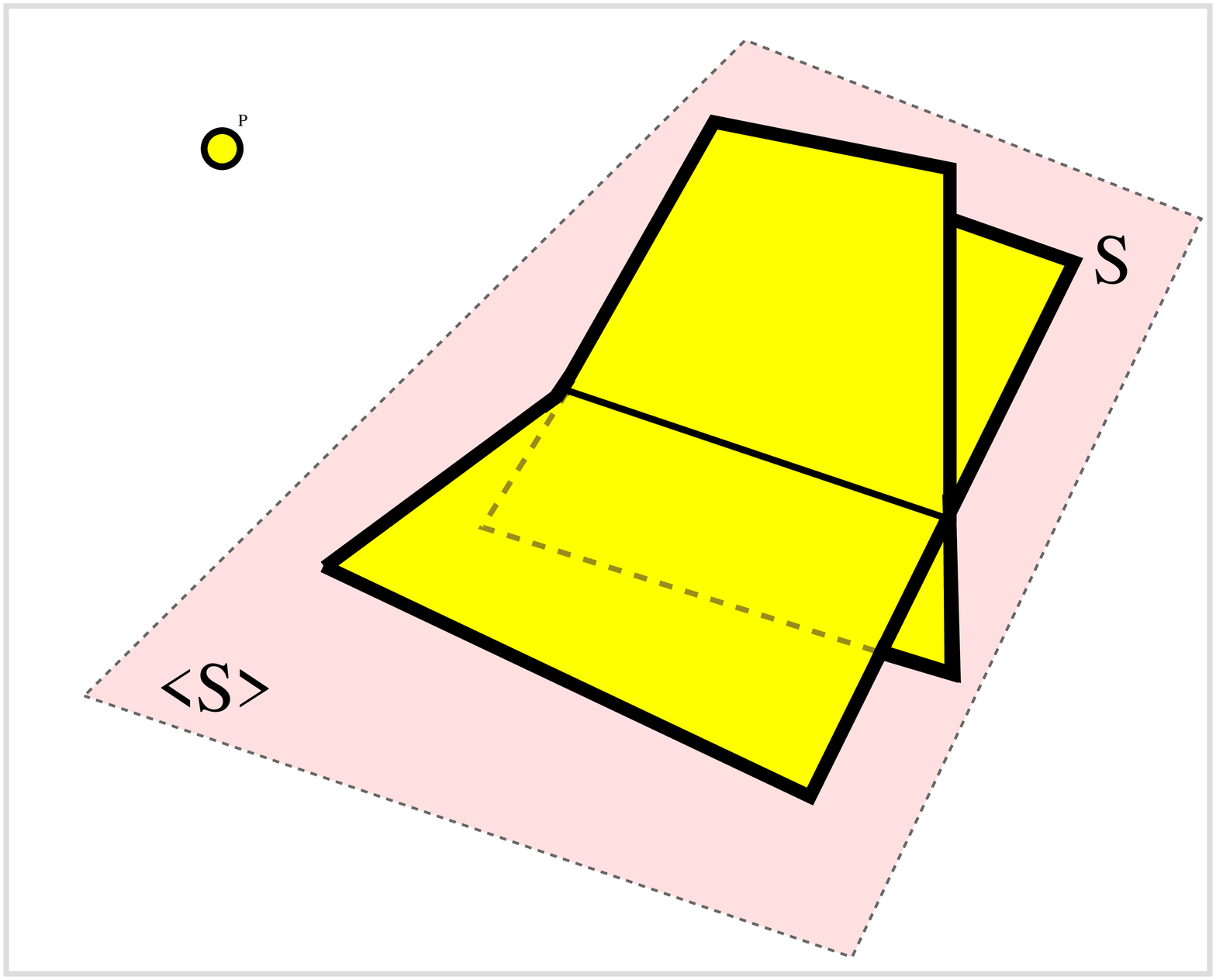} \\
\small{Base locus scheme $\mathcal B(\mathcal J_{14}^5)$}
\end{tabular}
\end{minipage}
\quad 
&  
\quad 
\begin{minipage}{4.0cm}
\begin{tabular}{c}
\psfrag{S}[][][1]{$\,  \scriptstyle{\boldsymbol{S}} $}
\psfrag{P}[][][1]{$\;\;\,   \scriptstyle{\boldsymbol{p}} $}
\psfrag{<S>}[][][1]{}
\includegraphics[width=4.4cm,height=4.0cm]{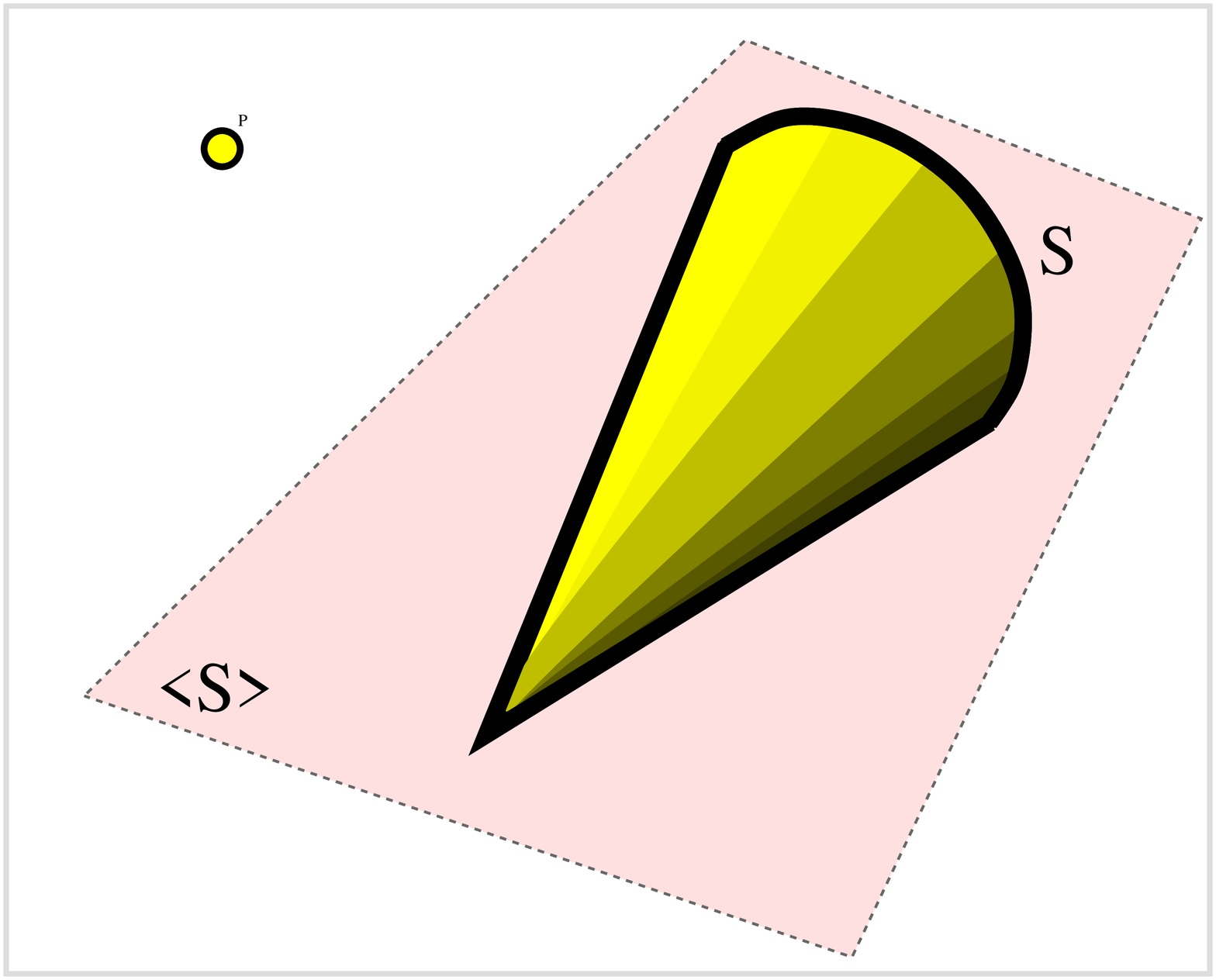} \\
\small{Base locus scheme $\mathcal B(\mathcal J_{15}^5)$}
\end{tabular}
\end{minipage}
\vspace{0.5cm}\\
&
\begin{minipage}{4.0cm}
\begin{tabular}{c}
\psfrag{S}[][][1]{$ \scriptstyle{\boldsymbol{S}} $}
\psfrag{P}[][][1]{$\;\; \,  \scriptstyle{\boldsymbol{p}} $}
\psfrag{<S>}[][][1]{}
\includegraphics[width=4.4cm,height=4.0cm]{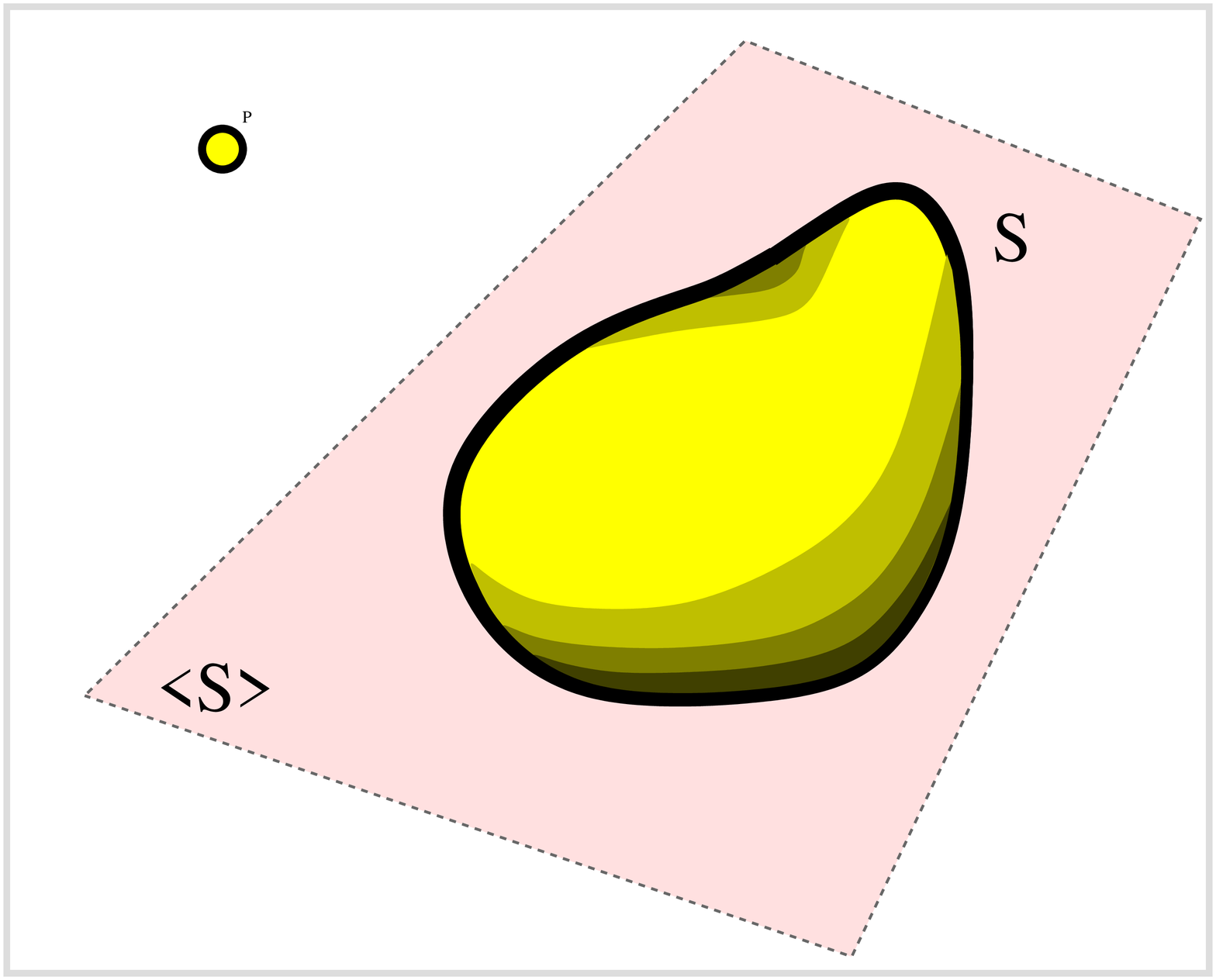} \\
\small{Base locus scheme $\mathcal B(\mathcal J_{16}^5)$}
\end{tabular}
\end{minipage}& 
\end{tabular}\medskip 
\caption{Graphic representations of the base locus schemes of elementary quadro-quadric Cremona transformations of $\mathbb P^4$ (these base loci are pictured in yellow and  the hyperplanes spanned by the quadric surface components are pictured in pale pink).}
\end{center}
\end{figure}

\begin{landscape}
\begin{table}
\begin{center}
\scalebox{0.85}{
  \begin{tabular}{c}  
  \begin{tabular}{|c|c|c|c|c|}\hline
  {\bf Algebra} $ {\boldsymbol{ \mathcal J}}$ & 
  {\bf  Quadro-quadric Cremona involution} 
     &  {\bf Primary decomposition of} $ {\boldsymbol{ \mathcal I_{\mathcal B(\mathcal J)}}}$     &  {\bf Geometrical description of}  $ {\boldsymbol{ {\mathcal B(\mathcal J)}}}$    & {\bf  Linear system}  $ 
      {\boldsymbol{ |  \mathcal I_{\mathcal B(\mathcal J)}(2) | }}  $      \\
   \hline  \hline
   $\mathcal J_{1}^5$ & $   \big(     x^2 \, , \,   -xy   \, , \,   -xz    \, , \,  y^2-xt     \, , \,  2yz-xu  \big) $     
   &  $(x,y)  \cap     (  x^2 \, , \,   xy   \, , \,   xz    \, , \,  y^2-xt     \, , \,  2yz-xu, z^2)$   &
 \begin{tabular}{l}  \vspace{-0.25cm} \\
 A 2-plane $\pi$ plus a  degenerate rational \\ quartic  
 curve  $\mathscr L^4$ with ${\mathscr L}^4_{red}$ included in $\pi$ \vspace{0.15cm}
 \end{tabular} 
          &     \begin{tabular}{l}  
     Cf. \eqref{E:LinearSystemJ1}
     \end{tabular}   
       \\  \hline 
      $\mathcal J_{9}^5$  &     $   \big(  xy \, ,\,  x^2 \, , \,  -yz\, , \, -xt\, , \, 2zt-xu  \big   ) $    
      & $(x,z)  \cap (y,x^2,xt,t^2,2zt-xu)$ &
 \begin{tabular}{l} \vspace{-0.25cm} \\
  A 2-plane $\pi$ plus a  double structure $\mathscr L _{1}$ \\   on   a line    such 
  that the hyperplane    $\langle \mathscr L _{1} \rangle $  \\   cuts $\pi$ along 
a line $L$ that meets $(\mathscr L _{1})_{\rm red}$ \vspace{0.15cm}
  \end{tabular}     
       &   
 \begin{tabular}{l} 
 \vspace{-0.25cm} \\
  Hyperquadrics containing $\pi$ and tangent     \\ 
 along  $(\mathscr L _{1})_{\rm red}$ to a fixed    smooth quadric  \\  surface
in $\langle \mathscr L _{1} \rangle$   containing $L\cup (\mathscr L _{1})_{\rm red}$ \vspace{0.15cm}
     \end{tabular}
      \\    \hline 
       $\mathcal J_{10}^5$ &$   \big(    xy \, , \,   x^2   \, , \,   -yz    \, , \,  -yt     \, , \, -xu     \big) $     
&  $(x,y)  \cap (x,z,t)  \cap   (u,y,x^2)$ &
 \begin{tabular}{l}  \vspace{-0.25cm} \\
 A 2-plane $\pi$ plus a line  
 $\ell$  intersecting $\pi$ plus a \\     double line $\mathscr L_0$   such that the  2-plane  $\langle \mathscr L_0\rangle$    \\ does not meet $\ell$ and   
    intersects $\pi$ along  $(\mathscr L_0)_{\rm red}$ \vspace{0.15cm}
 \end{tabular} 
          &    
          \begin{tabular}{l} 
           \vspace{-0.4cm} \\
          Hyperquadrics  
             containing  $\pi\cup \ell $  \\ and 
                      tangent to   $\langle \mathscr L_0 \rangle$  along    $(\mathscr L_0)_{\rm red}$          
     \end{tabular}
          
      \\  \hline  
         $\mathcal J_{13}$ &      $  \big (    yz \, , \, xz \, , \, xy\, , \, -zt\, , \, -yu   \big ) $     
& $ (u, z, x)\cap (t, y, x)\cap (z, y)$     &
\begin{tabular}{l}     
 \vspace{-0.25cm} \\
 A 2-plane $\pi$  plus  two skew 
 lines $\ell,\ell'$   \\ intersecting  $\pi$ in two distinct points   \vspace{0.15cm}
\end{tabular}
      & 
      \begin{tabular}{l}
        \vspace{-0.25cm} \\ 
       Hyperquadrics in $\p^4$\\ 
     containing   $\pi\cup \ell\cup \ell'$  \vspace{0.15cm}
     \end{tabular}        \\    \hline 
 \end{tabular}
 \end{tabular}
}
\end{center}\bigskip
\caption{Classification of quadro-quadric Cremona transformations of $\mathbb P^4$ of type II.}
\end{table}
\bigskip 
\bigskip 
\bigskip 
\bigskip 
\begin{figure}[H]
\begin{center}
\begin{tabular}{ccc}
\begin{minipage}{4.0cm}
\begin{tabular}{c}
\psfrag{L}[][][0.7]{$\quad \,   \scriptstyle{\boldsymbol{ L}}  $}
\psfrag{L1}[][][1]{$\quad \,   \scriptstyle{\boldsymbol{\mathscr L_1}}  $}
\psfrag{pi}[][][1]{$\quad \,  \scriptstyle{\boldsymbol{\pi}} $}
\psfrag{<L1>}[][][0.85]{\rotatebox{20}{\; \quad $ \scriptstyle{\boldsymbol{\langle \mathscr L_1\rangle}}\quad $}}
\includegraphics[width=5cm,height=4.6cm]{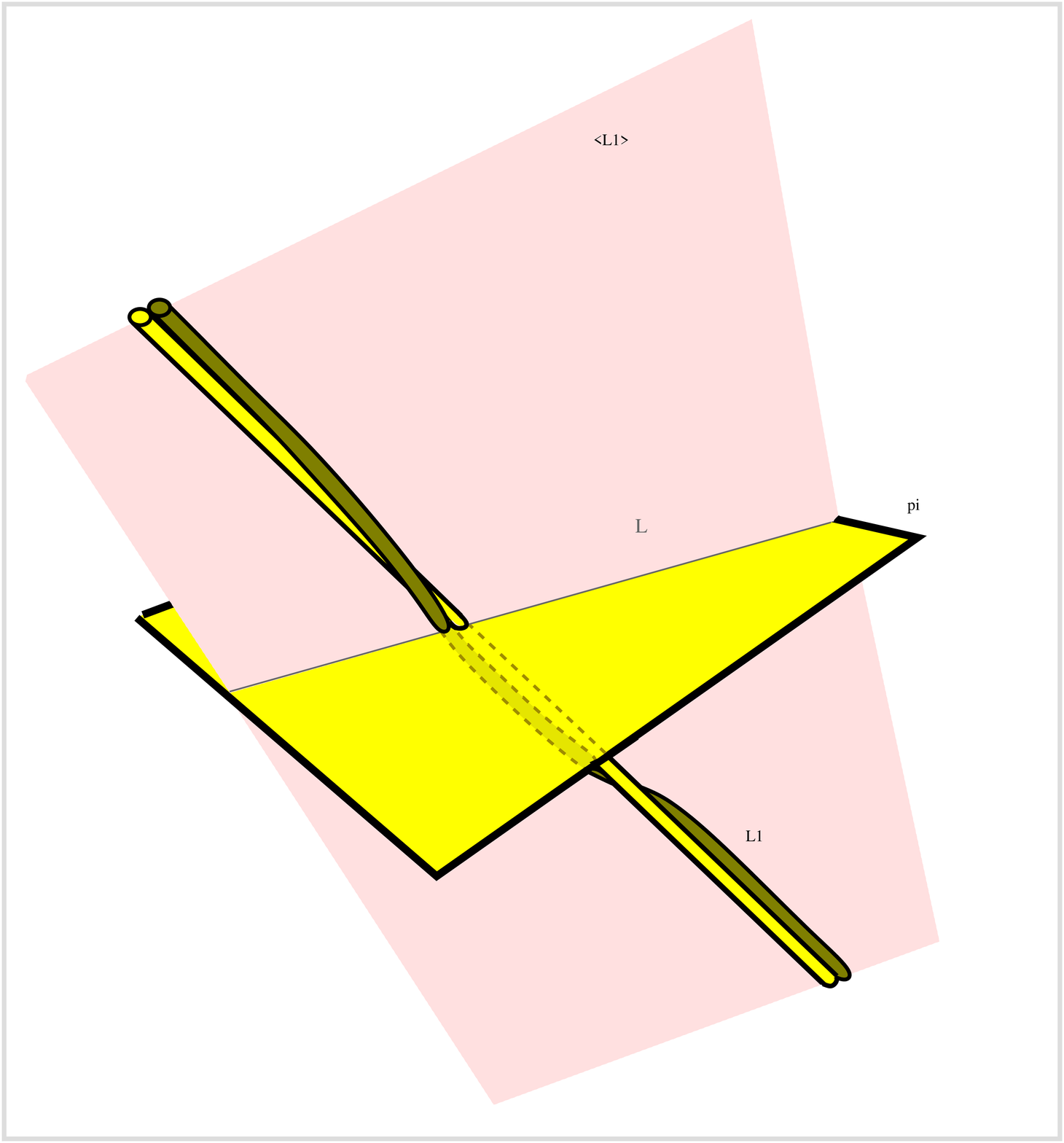} \\
\small{Base locus scheme $\mathcal B(\mathcal J_{9}^5)$}
\end{tabular}
\end{minipage}
\qquad 
\qquad 
&  
\qquad 
\begin{minipage}{4.0cm}
\begin{tabular}{c}
\psfrag{l1}[][][1]{$\; \, \scriptstyle{\boldsymbol{\ell}}  $}
\psfrag{L}[][][0.9]{$\; \; \;     \scriptstyle{\boldsymbol{\mathscr L_0}}  $}
\psfrag{pi}[][][1]{$\;  \scriptstyle{\boldsymbol{\pi}} $}
\psfrag{<L>}[][][0.85]{\rotatebox{0}{\quad \quad $ \scriptstyle{\boldsymbol{\langle \mathscr L_0\rangle}}\quad $}}
\includegraphics[width=5cm,height=4.6cm]{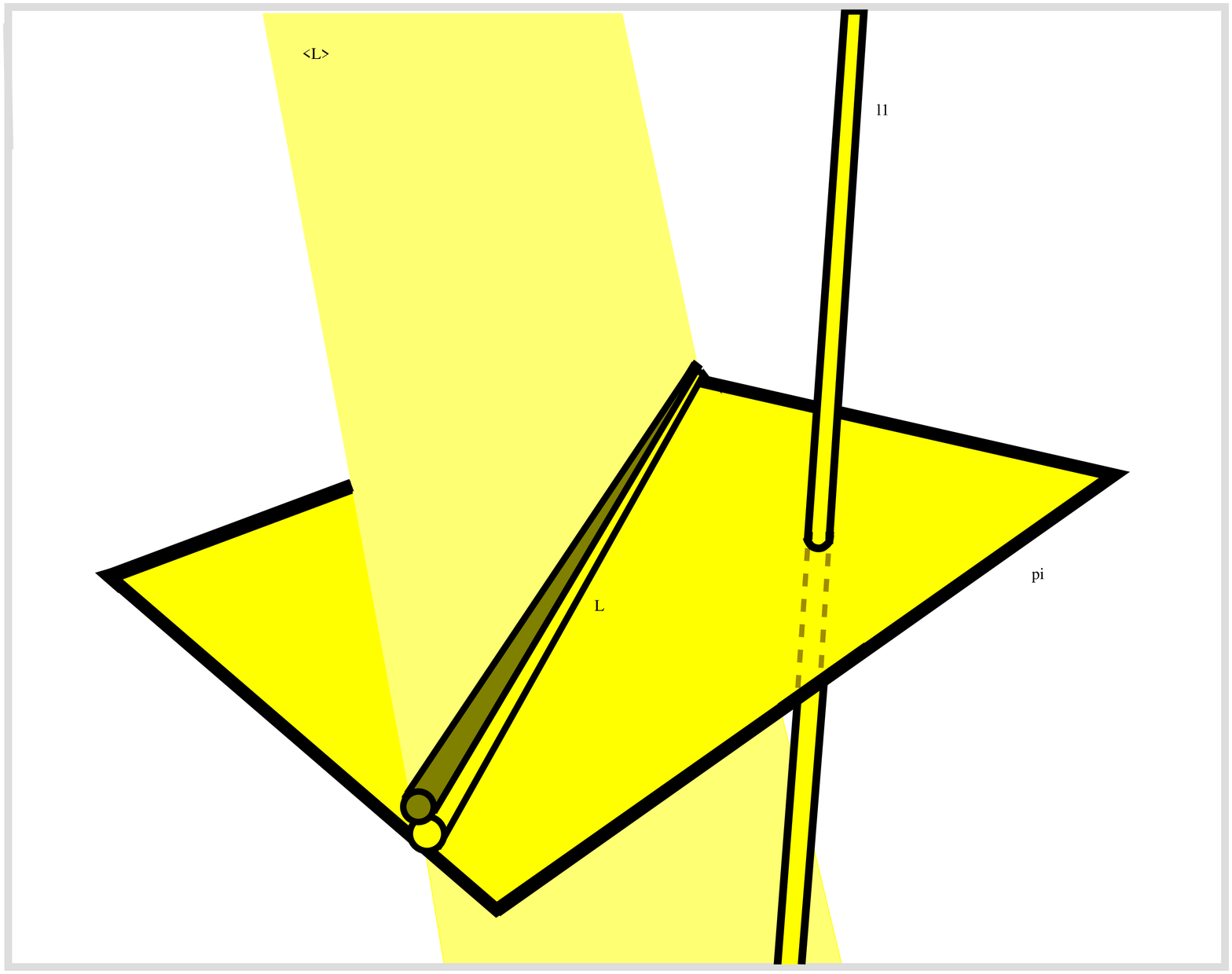} \\
\small{Base locus scheme $\mathcal B(\mathcal J_{10}^5)$}
\end{tabular}
\end{minipage}
\qquad   \qquad 
&  
\qquad  
\begin{minipage}{4.0cm}
\begin{tabular}{c}
\psfrag{l1}[][][1]{$\;\,   \scriptstyle{\boldsymbol{\ell}}  $}
\psfrag{l2}[][][1]{$\quad \,  \scriptstyle{\boldsymbol{\ell'}}  $}
\psfrag{pi}[][][1]{$\;  \scriptstyle{\boldsymbol{\pi}} $}
\includegraphics[width=5cm,height=4.6cm]{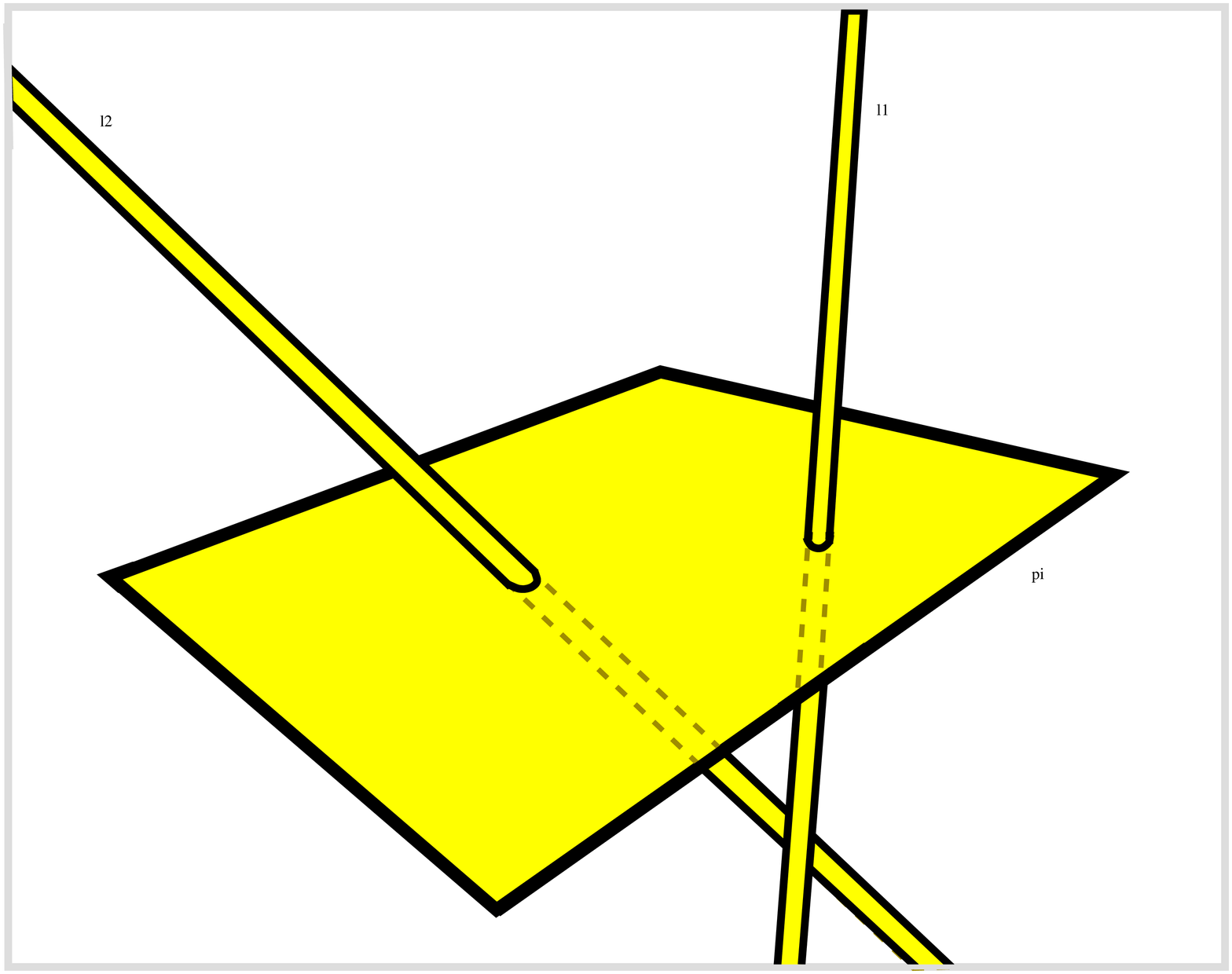} \\
\small{Base locus scheme $\mathcal B(\mathcal J_{13}^5)$}
\end{tabular}
\end{minipage}
\end{tabular}\medskip 
\end{center}
\caption{Pictural representations (in yellow) of the base locus schemes  of quadro-quadric Cremona transformations of $\mathbb P^4$ of type II.}
\end{figure}
\end{landscape}

The classification shows that up to linear equivalence, there are only two quadro-quadric Cremona transformations of $\p^4$ of type $III$.  The generic one is well-understood:  the associated algebra is $\mathcal J_{11}^5$,  the adjoint express as $(x,y,z,t,u)^\#=(xy,x^2,u^2-yz,-yt,-xu)$ in a suitable system of coordinates, the primary decomposition of the associated ideal is $\mathcal I_{\mathcal B(\mathcal J_{11}^5)}=(x,t,u^2-yz)\cap (y,x^2,xu,u^2)$,  the base locus scheme $ \mathcal B(\mathcal J_{11}^5)$ is the scheme $\mathcal B_{III}$ of   Theorem \ref{T:BV1} and the associated linear system is formed by   hyperquadrics containing $C$ and tangent to the line $\mathcal L_{\rm red}$ at  the intersection locus $C\cap \langle \mathcal L\rangle$,  that is just one point.
\sk 

In contrast we have almost nothing to say on the  quadro-quadric Cremona transformations of $\p^4$ associated to the algebra $\mathcal J_{2}^5$: its base locus scheme $\mathcal B(\mathcal J_{2}^5)$ is a 
degree 5 multiplicity 2 scheme of arithmetic genus 1 supported on the line $V(x,y,z)$.
 This case 
 appears to be the most degenerate and the most  mysterious one.  
 
 Note however than if for $\lambda\in \mathbb C$, one defines 
 $$
 f_\lambda: \p^4\dashrightarrow \p^4\, , [x:y:z:y:u]\mapsto [x^2:-xy:-xz:y^2+\lambda z^2-xt: 2yz-xu]\, , 
 $$
 one obtains an algebraic 1-dimensional family of quadro-quadric Cremona involutions of $\p^4$ such that $f_0=f_{\mathcal J_1^5}$ and with 
 $f_\lambda$ linearly equivalent to $f_1=f_{\mathcal J_2^5}$ for every $\lambda\neq 0$.  Thus $f_{\mathcal J_1^5}$ appears naturally as a degeneration of $f_{\mathcal J_2^5}$ in    contrast 
 with what is going on at the schematic level. Indeed,  the associated family of base locus schemes $\{  \mathcal B_{f_\lambda}\}_{\lambda \in \mathbb C}$  is not flat (it is neither equidimensional since $\dim \mathcal B_{f_0}=2$ whereas $\dim \mathcal B_{f_\lambda}=1$ for every $\lambda\neq 0$). 
 \sk 
 
 These few remarks show that the study of the `moduli space of quadro-quadric Cremona transformations of $\p^n$'  is certainly quite difficult from a scheme theoretic perspective as soon as  $n$ increases.

\section{\bf Quadro-quadric Cremona transformations of $\mathbb P^5$}
\label{S:ClassifDim5}

In \cite{brunoverra}, Bruno and Verra apply Semple's approach (see  Section \ref{S:qqP4-Bruno-Verra}) to the case of $\mathbb P^5$. 
Using the description of quadric Cremona transformations of $\mathbb P^4$ given by Semple in \cite{semple}, they prove the following result:
\begin{thm} 
\label{T:BV2}
The base locus scheme of a quadro-quadric Cremona transformation of $\mathbb P^5$ belongs to  the irreducible components of ${\rm Hilb}(\p^5)$ containing 
 one  of the  subschemes $\mathcal B_I,\ldots,\mathcal B_{IV}$ of $\mathbb P^5$ where: \mk
\begin{enumerate}
\item[(I)] $\mathcal B_I$ is the disjoint union   of a smooth quadric threefold  $Q$ with  a point $p$ lying outside the hyperplane $\langle Q \rangle$;\medskip 
\item[(II)] $\mathcal B_{II}$ is the union $\Pi \cup \pi\cup \ell$   of a 3-plane $\Pi$ with  a 2-plane $\pi$  and a line  $\ell$ with  relative positions as follows: $\pi$ intersects $\Pi$ along a line,  $\ell$ is disjoint from $\pi$ and meets $\Pi$ at one point;\mk 
\item[(III)] $\mathcal B_{III} $ is the schematic union 
 of    a double plane $\mathcal P$ in a hyperplane  $H$  with a smooth conic $C$ that is  tangent to $H$ at $C\cap \mathcal P_{\rm red}$ that is one point; \mk
 \item[(IV)] $\mathcal B_{IV} $ is the Veronese surface $v_2(\mathbb P^2)\subset \mathbb P^5$.
 \end{enumerate}
\end{thm}

The preceding result implies in particular that there are exactly four  types for elements of ${\bf Bir}_{22}(\mathbb P^5)$, that will be denoted by $I,II,III$ and $IV$. 
As in dimension four,  to these types corresponds four generic quadro-quadric Cremona transformation 
    of $\mathbb P^5$.  Explicit involutive normal forms   for these  as well as the corresponding multidegrees are given in the following table: 
\begin{table}[H]
  \centering
  \begin{tabular}{|c|c|l|c|}\hline
    {\bf Type}   $\boldsymbol{T}$ &    {\bf Base locus}   $\boldsymbol{\mathcal B_T}$    &\quad  {\bf  Generic Cremona transformation} $\boldsymbol{f_T}$ & 
  {\bf Multidegree}
  \\
   \hline
   \hline
 $I$  &      $Q\sqcup \{p\}$          &  $\big( y^2+z^2+t^2+u^2+v^2, xy,-xz,-xt,-xu,-xv\big)$        &    $(2,2,2,2)$    \\ \hline 
 $II$    &    $\Pi\cup \pi\cup \ell$         &   $(zy, xz, xy, -tz, -uz, -vy)$ &   $(2,3,3,2)$\\\hline
   $III$  &     $  \mathcal P \cup C$             &  $   (xy, x^2,-yz+v^2, -yt, -yu,-xv)  $  
   &    $(2,4,4,2)$   \\    \hline
    $IV$   &   $v_2(\mathbb P^2)$            &  $\big(  yz-v^2, xz-u^2,xy-t^2,uv-zt,tv-uy,tu-xv\big)$ & $(2,4,4,2)$     \\
       \hline                
 \end{tabular}
\label{Tab:jordan}
\caption{The four generic quadro-quadric Cremona transformations of $\mathbb P^5$.}
\end{table}

A complete  classification (up to isomorphisms) of rank 3 Jordan algebras of dimension 6  is obtained  in \cite{PIRIO}    (following the more general classification done by Wesseler in \cite{wesseler}). 
One deduces from it the  complete classification of quadro-quadric Cremona transformations of $\mathbb P^5$.

\begin{thm} A quadro-quadric Cremona transformation of $\mathbb P^5$ is linearly equivalent to one of the 29  
Jordan involutions listed in T{\small{ABLE}} 10 below.
\end{thm}

A careful and systematic  study of the Cremona maps in T{\small{ABLE}} 10  will be considered eventually elsewhere.

\newpage

\begin{table}[h]
 \hspace{-1cm}  \centering
 \begin{tabular}{|l|l|c|c|c|}\hline
  {\bf Algebra}  &  \qquad\qquad{\bf Jordan adjoint}  $\boldsymbol{(x,y,z,t,u,v)^\#}$ &    {\bf Type}  &     {\bf Multidegree}  &$ \boldsymbol{\dim(R)}$   \\
   \hline  \hline
 \; $\mathcal J_{\rm s}
 $&$\big(  yz-v^2, xz-u^2,xy-t^2,uv-zt,tv-uy,tu-xv\big)$   &  $IV$  &  $(2, 4,4 ,2)$   & 0\\    \hline  
\; $\mathcal J_{\rm ss}$& $\big( y^2+z^2+t^2+u^2+v^2, xy,-xz,-xt,-xu,-xv\big)$   &  $I$  &  $(2, 2,2 ,2)$  & 0  \\    \hline \hline
\; $\mathbb C\times \mathcal J_{q,4}^5$ & $\big(y^2+z^2+t^2+u^2,xy,-xz,-xt,-xu ,-xv \big)$   &  $I$  &  $(2, 2,2 ,2)$  & 1
  \\    \hline  \hline
\; $\mathbb C\times \mathcal J_{q,3}^5$ &   $\big(y^2+z^2+t^2,xy,-xz,-xt,-xu ,-xv \big)$   &  $I$  &  $(2, 2,2 ,2)$  & 2 \\    \hline
\; $\mathcal J_{92}$ &   $\big(yz-t^2,xz,xy,-xt,tv-uz,tu-yv \big)$   &  $IV$  &  $(2, 4,4 ,2)$  & 2 \\    \hline  \hline
\; $\mathcal J_{100}$ &   $ (yz, xz, xy, -xt, -xu, -xv)    $   &  $ I  $  &  $(2,2  ,2  ,2)$  & 3 \\    \hline
\; $\mathcal J_{101}$  &$ (yz, xz, xy, -zt, -zu, -yv)    $   &  $  II $  &  $(2,3  ,3  ,2)$   & 3  \\    \hline
\; $\mathcal J_{102}^0$  &$   (yz, xz, xy, -zt, -yu, -xv)  $   &  $  IV  $  &  $(2,4  ,4  ,2)$   & 3\\    \hline
\; $\mathcal J_{102}^1$ &$   (yz, xz, xy, -zt+2uv, -yu, -xv)   $   &  $IV  $  &  $(2, 4 ,4  ,2)$  & 3 \\    \hline
  \hline
\; $\mathcal J_{113}^{000}$ &$ (xy, x^2, -yz, -ty, -yu,-xv)$ &  $II$  &  $(2,3,3,2)$   & 4\\    \hline
\; $\mathcal J_{100}^{100}$ & $(xy, x^2,-yz+v^2, -yt, -yu,-xv)$    & $III$   &  $(2, 4,4 ,2)$   & 4\\    \hline
 \; $\mathcal J_{122}^{a}$    &   $  \big(xy, x^2,-yz, -ty,-xu,-xv\big)   $   & $II$  &  $  (2,3,3,2)  $& 4 \\    \hline
 \; $\mathcal J_{122}^{b}$     &   $  \big(xy, x^2,-yz+u^2, -yt,-xu,-xv\big)   $   & $IV$ &     $(2,4,4,2)$ & 4  \\    \hline
    \; $\mathcal J_{122}^{c}$ &     $  \big(xy, x^2,-yz+2uv,-yt, -xu,-xv\big)   $  & $IV$   &  $(2,4,4,2)$ & 4\\    \hline
   \; $\mathcal J_{122}^{d}$   &     $  \big( xy, x^2,-yz+u^2,-yt+v^2, -xu,-xv    \big)   $     &  $ IV$   &    $(2,4 ,4 ,2)$& 4    \\    \hline  
 \; $\mathcal J_{122}^{e}$    &   $  \big(xy, x^2,-yz+u^2, -yt+2uv,-xu,-xv\big)   $  &  $  IV $    &    $(2,4 ,4 ,2)$ & 4 \\    \hline
 \; $\mathcal J_{124}^0$&    $(xy, x^2,-yz, -yt,-xu,-xv+2tu)$       &  $IV$  &  $(2, 4, 4,2)$   & 4\\    \hline
\; $\mathcal J_{124}^1$ &    $(xy, x^2,-yz+ u^2, -yt,-xu,-xv+2tu)$           &  $IV$  &  $(2, 4,4 ,2)$   & 4\\    \hline
\; $\mathcal J_{126}^0$ &    $(xy, x^2,-yz, -yt,-xu,-xv+2zu+2tu) $    &  $IV$   &  $(2, 4, 4,2)$   & 4\\    \hline
\; $ \mathcal J_{126}^1$ &  $ (xy, x^2,-yz+ u^2, -yt- u^2,-xu,-xv+2zu+2tu)$    & $IV$    &  $(2,4 ,4 ,2)$   & 4\\    \hline
\; $\mathcal J_{140}^{(0)}$ &     $ (xy, x^2,-yz, -xt,-xu,-xv)$    & $I$   & $(2,2,2,2)$    & 4\\    \hline
\; $\mathcal J_{140}^{(1)}$ &     $ (xy, x^2,-yz+  t^2, -xt,-xu,-xv)$    & $I$   & $(2,2,2,2)$    & 4\\    \hline
\; $\mathcal J_{140}^{(2)}$ &     $ (xy, x^2,-yz+ t^2+ u^2, -xt,-xu,-xv)$    & $I$   & $(2,2,2,2)$    & 4\\    \hline
\; $\mathcal J_{140}^{(3)}$ &     $ (xy, x^2,-yz+ + t^2+ u^2+v^2, -xt,-xu,-xv)$    & $I$   & $(2,2,2,2)$    & 4\\    \hline
\; $\mathcal J_{143}$ &     $(xy, x^2,-yz, -xt,-xu,-xv+2zt) $    &   $II$ & $(2,3 ,3 ,2)$    & 4\\    \hline  \hline
\; $  \mathcal J_{ES}^a $ &  $\big( x^2, -xy, -xz,-xt+2yu-zv, -xu+z^2, -xv+2yz\big)$   &  $IV$  &  $(2, 4,4 ,2)$  & 5  \\    \hline  
\; $  \mathcal J_{ES}^b $ & $\big( x^2, -xy, -xz,-xt+2yz, -xu+y^2, -xv+z^2\big)$      &   $IV$  &  $(2, 4,4 ,2)$   & 5 \\    \hline
 \; $\mathcal J_{ES}^{c, (-1)}$&   
 $\big( x^2, -xy, -xz,-xt, -xu-2yz+ 2zt , -xv-y^2+t^2\big) $
    &    $II $    &  $(2,3,3,2)$  & 5      \\    \hline
   \; $\mathcal J_{ES}^{c, (0)}$&   
   $\big( x^2, -xy, -xz,-xt, -xu+2zt , -xv+t^2 \big) $
    &    $II $    &  $(2,3,3,2)$    & 5    \\    \hline 
       \; $\mathcal J_{ES}^{c, (d)}$&   
   $\big( x^2, -xy, -xz,-xt, -xu+2zt , -xv+z^2+t^2 \big) $
    &    $IV$    &  $(2,4,4,2)$    & 5    \\    \hline 
         \; $\mathcal J_{ES}^{c, (1)}$&  
     $\big( x^2, -xy, -xz,-xt, -xu+y^2+2zt , -xv+t^2 \big) $    
    &    $IV $    &  $(2,4,4,2)$    & 5    \\    \hline 
         \; $\mathcal J_{ES}^{c, (2)}$&  
         $\big( x^2, -xy, -xz,-xt, -xu+2zt , -xv+ y^2+t^2 \big) $
    &    $ IV$    &  $(2,4,4,2)$    & 5    \\    \hline 
         \; $\mathcal J_{ES}^{c, (3)}$&  
         $\big( x^2, -xy, -xz,-xt, -xu+2  yz+2zt , -xv+t^2 \big) $
    &    $IV $    &  $(2,4,4,2)$   & 5     \\    \hline 
         \; $\mathcal J_{ES}^{c, (4)}$&   
         $\big( x^2, -xy, -xz,-xt, -xu+2zt , -xv+2yz +t^2 \big) $
             &    $IV $    &  $(2,4,4,2)$     & 5   \\    \hline 
         \; $\mathcal J_{ES}^{c, (5)}$&  
         $\big( x^2, -xy, -xz,-xt, -xu+y^2+2zt , -xv+z^2+t^2 \big) $
    &    $ IV$    &  $(2,4,4,2)$   & 5     \\    \hline 
\; $\mathcal J_{ES}^{d,000}$ &$\big( x^2, -xy, -xz,-xt, -xu, -xv+y^2 \big)$  &  $I$  &   $(2, 2,2 ,2)$  & 5   \\    \hline
\; $\mathcal J_{ES}^{d,100}$ &$\big( x^2, -xy, -xz,-xt, -xu, -xv+y^2+ z^2 \big)$   &  $I$  &   $(2, 2,2 ,2)$& 5    \\    \hline
\; $\mathcal J_{ES}^{d,110}$&$\big( x^2, -xy, -xz,-xt, -xu, -xv+y^2+z^2+ t^2 \big)$   &  $I$  &   $(2, 2,2 ,2)$ & 5   \\    \hline
\; $\mathcal J_{ES}^{d,111}$ &$\big( x^2, -xy, -xz,-xt, -xu, -xv+y^2+ z^2+ t^2+ u^2 \big)$   &  $I$  & $(2, 2,2 ,2)$& 5   \\    \hline
 \end{tabular}\vspace{0.3cm}
\caption{Involutive normal forms for quadro-quadric Cremona transformations of  $\mathbb P^ 5$ \hspace{1cm } ($ \dim R$ stands for the dimension of the radical of the corresponding  Jordan algebra, this one being labelled  with  the notation used in  \cite{PIRIO}).}
\end{table}

\end{document}